\documentclass[reqno, 8pt]{amsart}

\usepackage{amsmath,amssymb,amsthm,amsfonts,mathrsfs}
\usepackage{fullpage}
\usepackage{xcolor}
\usepackage[colorlinks=true,
linkcolor=blue,
anchorcolor=blue,
citecolor=red
]{hyperref}
\allowdisplaybreaks 
\newtheorem{theorem}{Theorem}[section]
\newtheorem{lemma}[theorem]{Lemma}

\newtheorem*{conjecture*}{Conjecture}
\theoremstyle{definition}

\theoremstyle{remark}

\newtheorem*{remark*}{remark}

\usepackage{colonequals}

\author{Runbo Li}
\address{International Curriculum Center, The High School Affiliated to Renmin University of China, Beijing, China}
\email{carey.lee.0433@gmail.com}

\makeatletter
\@namedef{subjclassname@2020}{\textup{}2020 Mathematics Subject Classification}
\makeatother

\title[]{Remarks on additive representations of natural numbers}

\begin{document}
	
\begin{abstract}
For two relatively prime square-free positive integers $a$ and $b$, we study integers of the form $a p+b P_{2}$ and give a new lower bound for the number of such representations, where $a p$ and $b P_{2}$ are both square-free, $p$ denote a prime, and $P_{2}$ has at most two prime factors. We also consider some special cases where $p$ is small, $p$ and $P_2$ are within short intervals, $p$ and $P_2$ are within arithmetical progressions and a Goldbach-type upper bound result. Our new results generalize and improve previous results.
\end{abstract}

\maketitle


\tableofcontents

\section*{Remarks on the 2025 version}
This preprint was completed by the author in 2023/2024. After that, the author got a series of improvements on this topic, see \cite{LRB1733} in 2024 and \cite{LRB197} in 2025. However, as a completed preprint, we decided not to add any new things and keep the main results unchanged in this preprint. The only new result presented in this version is a previously unpublished result $0.9409$ that
was proved by the author in Feb 2024.

\section{Introduction}
Let $N_e$ denote a sufficiently large even integer, $p$ and $q$, with or without subscript, denote prime numbers, and let $P_{r}$ denote an integer with at most $r$ prime factors counted with multiplicity. For each $N_e \geqslant 4$ and $r \geqslant 2$, we define 
\begin{equation}
D_{1,r}(N_e):= \left|\left\{p : p \leqslant N_e, N_e-p=P_{r}\right\}\right|.
\end{equation}

In 1966 Chen \cite{Chen1966} announced his remarkable Chen's theorem: let $N_e$ be a sufficiently large even integer, then
\begin{equation}
D_{1,2}(N_e) \geqslant 0.67\frac{C(N_e) N_e}{(\log N_e)^2}
\end{equation}
where
\begin{equation}
C(N_e):=\prod_{\substack{p \mid N_e \\ p>2}} \frac{p-1}{p-2} \prod_{p>2}\left(1-\frac{1}{(p-1)^{2}}\right).
\end{equation}
and the detail was published in \cite{Chen1973}. The original proof of Chen was simplified by Pan, Ding and Wang \cite{PDW1975}, Halberstam and Richert \cite{HR74}, Halberstam \cite{Halberstam1975} and Ross \cite{Ross1975}. As Halberstam and Richert indicated in \cite{HR74}, it would be interesting to know whether a more elaborate weighting procedure could be adapted to the purpose of (2). This might lead to numerical improvements and could be important. Chen's constant 0.67 was improved successively to
$$
0.689,\ 0.7544,\ 0.81,\ 0.8285,\ 0.836,\ 0.867,\ 0.899
$$
by Halberstam and Richert \cite{HR74} \cite{Halberstam1975}, Chen \cite{Chen1978_1} \cite{Chen1978_2}, Cai and Lu \cite{CL2002}, Wu \cite{Wu2004}, Cai \cite{CAI867} and Wu \cite{Wu2008} respectively.

In 1990, Wu \cite{WU90} generalized Chen's theorem and showed that
\begin{equation}
D_{1,r}(N_e) \geqslant 0.67 \frac{C(N_e) N_e}{(\log N_e)^2} (\log \log N_e)^{r-2}. 
\end{equation}
Kan \cite{Kan1} also proved the similar result in 1991:
\begin{equation}
D_{1,r}(N_e) \geqslant \frac{0.77}{(r-2)!} \frac{C(N_e) N_e}{(\log N_e)^2} (\log \log N_e)^{r-2},
\end{equation}
which is better than Wu's result when $r=3$. Kan \cite{Kan2} proved the more generalized theorem in 1992:
\begin{equation}
D_{s,r}(N_e) \geqslant \frac{0.77}{(s-1)!(r-2)!} \frac{C(N_e) N_e}{(\log N_e)^2} (\log \log N_e)^{s+r-3},
\end{equation}
where $s \geqslant 1$,
\begin{equation}
D_{s,r}(N_e):= \left|\left\{P_s : P_s \leqslant N_e, N_e-P_s=P_{r}\right\}\right|.
\end{equation}

Furthermore, for two relatively prime square-free positive integers $a$ and $b$, let $N$ denote a sufficiently large integer that is relatively prime to both $a$ and $b$, $a,b < N^{\varepsilon}$ and let $N$ be even if $a$ and $b$ are both odd. Let $R_{a,b}(N)$ denote the number of primes $p$ such that $ap$ and $N-ap$ are both square-free, $b \mid (N-ap)$, and $\frac{N-ap}{b}=P_2$. In 1976, Ross [\cite{RossPhD}, Chapter 3] got a similar result without the square-free restrictions on $ap$ and $N-ap$. In 2023, Li \cite{LIHUIXI} established that
\begin{equation}
R_{a, b}(N) \geqslant 0.68 \frac{C(abN) N}{a b(\log N)^{2}}.
\end{equation}

In this paper, we improve the result by using a delicate sieve process similar to that of \cite{CAI867} and prove that
\begin{theorem}\label{t1}
$$
R_{a, b}(N) \geqslant 0.8671\frac{C(abN) N}{a b(\log N)^{2}}.
$$
\end{theorem}
It is easy to see that when we take $a=1$ and $b=1$, Theorem~\ref{t1} implies Cai's result on Chen's theorem [\cite{CAI867}, Theorem 1]; when we take $a=1$ and $b=2$, Theorem~\ref{t1} improves Li's result related to the Lemoine's conjecture [\cite{LiHuixi2019}, Theorem 1].
When we take $a=q_1 q_2 \cdots q_s$ and $b=q^\prime_1 q^\prime_2 \cdots q^\prime_r$ where $q, q^\prime$ denote prime numbers satisfy
$$
s,r\geqslant 1, \quad q_i,q^\prime_j < N^{\varepsilon}, \quad 
(q_i, N)=(q^\prime_j, N)=1\ \operatorname{for\ every}\ 1 \leqslant i \leqslant s,1 \leqslant j \leqslant r,
$$
Theorem~\ref{t1} generalizes and improves the previous results of Kan [\cite{Kan1}, Theorem 2] [\cite{Kan2}, Theorem 2] and Wu [\cite{WU90}, Theorems 1 and 2]. Clearly one can modify our proof of Theorem~\ref{t1} to get a similar lower bound on the twin prime version. For this, we refer the interested readers to Ross's PhD thesis \cite{RossPhD} and [\cite{Opera}, Sect. 25.6], as well as \cite{KanJNT}, \cite{Kan1996} and \cite{KS1996} for some interesting applications.

Chen's theorem with small primes was first studied by Cai \cite{Cai2002}. For $0<\theta \leqslant 1$, we define 
\begin{equation}
D_{1,r}^{\theta}(N_e):= \left|\left\{p : p \leqslant N_e^{\theta}, N_e-p=P_{r}\right\}\right|.
\end{equation}
Then it is proved in \cite{Cai2002} that for $0.95 \leqslant \theta \leqslant 1$, we have
\begin{equation}
D_{1,2}^{\theta}(N_e) \gg \frac{C(N_e) N_e^{\theta}}{(\log N_e)^2}.
\end{equation}
Cai's range $0.95 \leqslant \theta \leqslant 1$ was extended successively to $0.945 \leqslant \theta \leqslant 1$ in \cite{CL2011} and to $0.941 \leqslant \theta \leqslant 1$ in \cite{Cai2015}.

In this paper, we generalize their results to integers of the form $ap+bP_2$. Let $R_{a,b}^{\theta}(N)$ denote the number of primes $p\leqslant N^{\theta}$ such that $ap$ and $N-ap$ are both square-free, $b \mid (N-ap)$, and $\frac{N-ap}{b}=P_2$. In 1976, Ross [\cite{RossPhD}, Chapter 5] got a similar result without the square-free restrictions on $ap$ and $N-ap$ and showed that $0.959 \leqslant \theta \leqslant 1$ is admissible. Now by using a delicate sieve process similar to that of \cite{Cai2015}, we prove that
\begin{theorem}\label{t2}
For $0.9409 \leqslant \theta \leqslant 1$ we have
$$
R_{a, b}^{\theta}(N) \gg \frac{C(abN) N^{\theta}}{a b(\log N)^{2}}.
$$
\end{theorem}
For similar results on the twin prime version with small primes, we refer the interested readers to \cite{LiuMingChit}, \cite{XieShenggang}, \cite{Coleman} and \cite{MengXianmeng}.

Chen's theorem in short intervals was first studied by Ross \cite{Ross1978}. For $0<\kappa \leqslant 1$, we define 
\begin{equation}
D_{1,r}(N_e,\kappa):= \left|\left\{p : N_e/2-N_e^{\kappa} \leqslant p,P_r \leqslant N_e/2+N_e^{\kappa}, N_e=p+P_{r}\right\}\right|.
\end{equation}
Then it is proved in \cite{Ross1978} that for $0.98 \leqslant \kappa \leqslant 1$, we have
\begin{equation}
D_{1,2}(N_e,\kappa) \gg \frac{C(N_e) N_e^{\kappa}}{(\log N_e)^2}.
\end{equation}
The constant 0.98 was improved successively to
$$
0.974,\ 0.973,\ 0.9729,\ 0.972,\ 0.971,\ 0.97
$$
by Wu \cite{Wu1993} \cite{Wu1994}, Salerno and Vitolo \cite{SV1993}, Cai and Lu \cite{CL1999}, Wu \cite{Wu2004} and Cai \cite{CAI867} respectively. 

In this paper, we generalize their results to integers of the form $ap+bP_2$. Let $R_{a,b}(N,\kappa)$ denote the number of primes $N/2-N^{\kappa} \leqslant p \leqslant N/2+N^{\kappa}$ such that $ap$ and $N-ap$ are both square-free, $b \mid (N-ap)$ and $\frac{N-ap}{b}=P_2$. In \cite{Ross1978}, Ross mentioned that his method can be used to prove similar results of $R_{a, b}(N,\kappa)$ with $0.98 \leqslant \kappa \leqslant 1$ and a detailed proof was given in [\cite{RossPhD}, Chapter 5]. Now by using a delicate sieve process similar to that of \cite{CAI867}, we prove that
\begin{theorem}\label{t3}
For $0.97 \leqslant \kappa \leqslant 1$ we have
$$
R_{a, b}(N,\kappa) \gg \frac{C(abN) N^{\kappa}}{a b(\log N)^{2}}.
$$
\end{theorem}

From our Theorems~\ref{t1}--\ref{t3}, it can be seen that the first aim of this paper is to improve the old results on the natural numbers of the form $ap+bP_2$ to be consistent with or better than the results on the even numbers of the form $p+P_2$. Before our work, all results on this topic are weaker than those of binary Goldbach problem. For Theorem~\ref{t1}, the constants 0.608 in \cite{RossPhD} and 0.68 in \cite{LIHUIXI} are smaller than 0.867 in \cite{CAI867}. For Theorems~\ref{t2}--\ref{t3}, Ross's exponent 0.959 and 0.98 are again weaker than those in \cite{Cai2015} and \cite{CAI867}.

Chen's theorem in arithmetical progressions was first studied by Kan and Shan \cite{KanShan}. If we define 
\begin{equation}
D_{1,r}(N_e,c,d):= \left|\left\{p : p \leqslant N_e, p \equiv d(\bmod c), (c,d)=1, (N_e-p,c)=1, N_e-p=P_{r}\right\}\right|,
\end{equation}
then it is proved in \cite{KanShan} that for $c \leqslant (\log N_e)^{C}$ where $C$ is a positive constant, we have
\begin{equation}
D_{1,r}(N_e,c,d) \geqslant \frac{0.77}{(r-2)!}\prod_{\substack{p \mid c \\ p \nmid N_e\\ p>2}} \left(\frac{p-1}{p-2}\right)\frac{C(N_e) N_e}{\varphi(c)(\log N_e)^2} (\log \log N_e)^{r-2},
\end{equation}
where $\varphi$ denote the Euler's totient function. Clearly their result (14) generalized the previous results (2) and (5). They also got the similar results on the twin prime version (or even the "safe prime" version, see \cite{Kan2004}) and Lewulis \cite{lewulis} considered the similar problem. However, their results are only valid when $c$ is "small". In 1999, Cai and Lu \cite{CL1999_2} considered this problem with "large" $c$ and proved that for $c \leqslant N_e^{\frac{1}{37}}$, except for $O\left(N_e^{\frac{1}{37}} (\log N_e)^{-A}\right)$ exceptional values, we have
\begin{equation}
D_{1,2}(N_e,c,d) \gg \prod_{\substack{p \mid c \\ p \nmid N_e\\ p>2}} \left(\frac{p-1}{p-2}\right) \frac{C(N_e)N_e}{\varphi(c)(\log N_e)^2}
\end{equation}
and they mentioned that the exponent $\frac{1}{37}$ can be improved to $0.028$. In this paper, we further generalize their results to integers of the form $ap+bP_2$. Let $R_{a,b}(N,c,d)$ denote the number of primes $p \equiv d(\bmod c)$ such that $ap$ and $N-ap$ are both square-free, $b \mid (N-ap)$, and $\frac{N-ap}{b}=P_2$. Then by using a delicate sieve process similar to that of \cite{CAI867}, we prove that
\begin{theorem}\label{t4}
For $c \leqslant (\log N)^{C}$, we have
$$
R_{a,b}(N,c,d) \geqslant 0.8671 \prod_{\substack{p \mid c \\ p \nmid N\\ p>2}} \left(\frac{p-1}{p-2}\right)\frac{C(abN) N}{\varphi(c)ab(\log N)^2}.
$$
\end{theorem}
\begin{theorem}\label{t41}
For $c \leqslant N^{0.028}$, except for $O\left(N^{0.028} (\log N)^{-A}\right)$ exceptional values, we have
$$
R_{a,b}(N,c,d) \gg \prod_{\substack{p \mid c \\ p \nmid N\\ p>2}} \left(\frac{p-1}{p-2}\right)\frac{C(abN) N}{\varphi(c)ab(\log N)^2}.
$$
\end{theorem}

Now we combine Theorem~\ref{t4} with Theorems~\ref{t2}--\ref{t3}. Let $R_{a,b}^{\theta}(N,c,d)$ denote the number of primes $p \equiv d(\bmod c)$ such that $p\leqslant N^{\theta}$, $ap$ and $N-ap$ are both square-free, $b \mid (N-ap)$, and $\frac{N-ap}{b}=P_2$. 
And let $R_{a,b}(N,c,d,\kappa)$ denote the number of primes $p \equiv d(\bmod c)$ such that $N/2-N^{\kappa} \leqslant p \leqslant N/2+N^{\kappa}$, $ap$ and $N-ap$ are both square-free, $b \mid (N-ap)$, $\frac{N-ap}{b}=P_2$, and $N/2-N^{\kappa} \leqslant \frac{N-ap}{b} \leqslant N/2+N^{\kappa}$. Then by using a delicate sieve process similar to that of \cite{CAI867} and \cite{Cai2015}, we prove that
\begin{theorem}\label{t42}
For $c \leqslant (\log N)^{C}$ and $0.9409 \leqslant \theta \leqslant 1$, we have
$$
R_{a,b}^{\theta}(N,c,d) \gg \prod_{\substack{p \mid c \\ p \nmid N\\ p>2}} \left(\frac{p-1}{p-2}\right)\frac{C(abN) N^{\theta}}{\varphi(c)ab(\log N)^2}.
$$
\end{theorem}
\begin{theorem}\label{t43}
For $c \leqslant (\log N)^{C}$ and $0.97 \leqslant \kappa \leqslant 1$, we have
$$
R_{a,b}(N,c,d,\kappa) \gg \prod_{\substack{p \mid c \\ p \nmid N\\ p>2}} \left(\frac{p-1}{p-2}\right)\frac{C(abN) N^{\kappa}}{\varphi(c)ab(\log N)^2}.
$$
\end{theorem}
Clearly our Theorems~\ref{t42}--\ref{t43} focus on the case when $c$ is "small". For "large" $c$, we need to control the size of both $\theta$ (or $\kappa$) and $c$, and it seems hard to say what is "optimal". For example, we can show that for some $0<\delta_1<0.028$, $0.9409<\delta_2<1$ and $c \leqslant N^{\delta_1}$, except for $O\left(N^{\delta_1} (\log N)^{-A}\right)$ exceptional values, we have
\begin{equation}
R_{a,b}^{\delta_2}(N,c,d) \gg \prod_{\substack{p \mid c \\ p \nmid N\\ p>2}} \left(\frac{p-1}{p-2}\right)\frac{C(abN) N^{\delta_2}}{\varphi(c)ab(\log N)^2},
\end{equation}
but we cannot say what choice of $\delta_1$ and $\delta_2$ are the optimal values.

From our Theorems~\ref{t4}--\ref{t43}, it can be seen that the second aim of this paper is to construct some new results on the natural numbers of the form $ap+bP_2$ that generalize the results on the even numbers of the form $p+P_2$, $p+P_r$ and $P_s+P_r$.

The last theorem in this paper is a Goldbach-type upper bound result. Similar to [\cite{LIHUIXI}, Theorem 1. (2)], we also improve the upper bound of the number of primes $p$ such that $ap$ and $N-ap$ are both square-free, $b \mid (N-ap)$, and $\frac{N-ap}{b}$ is also a prime number. By using a delicate sieve process similar to that of [\cite{ERPAN}, Chap. 9.2],  we prove that
\begin{theorem}\label{t5}
$$
\sum_{\substack{a p_{1}+b p_{2}=N \\ p_{1} \text { and } p_{2} \text { are primes }}} 1 \leqslant 7.928 \frac{C(abN) N}{a b(\log N)^{2}}.
$$
\end{theorem}

In fact, Lemmas~\ref{l31}--\ref{l35} are also valid for the sets $\mathcal{A}_3$--$\mathcal{A}_6$ in section 2 if we make some suitable modifications. Since the detail of the proof of Theorems~\ref{t3}--\ref{t5} is similar to those of \cite{CL1999}, \cite{KanShan}, \cite{CL1999_2}, \cite{ERPAN} and Theorems~\ref{t1}--\ref{t2} so we omit them in this paper.

In this paper, we do not focus on Chen's double sieve technique. Maybe this can be used to improve our Theorems~\ref{t1}--\ref{t5}. For this, we refer the interested readers to \cite{Chen1978_3}, \cite{Wu2004}, \cite{Wu2008} and Quarel's thesis \cite{quarel}.

It is worth to mention that if we relax the number of prime factors of $\frac{N-ap}{b}$ from two to three, we can extend the range of $\theta$ in Theorems~\ref{t2} and ~\ref{t42} and $\kappa$ in Theorems~\ref{t3} and ~\ref{t43} to $0.838 \leqslant \theta \leqslant 1$ and $0.919 \leqslant \kappa \leqslant 1$ respectively. This improvement partially relies on the cancellation of the use of Wu's mean value theorem (see \cite{Wu1993}, this is because we don't need Chen's switching principle to prove such results that involve integers of the form $a p+b P_{3}$).

\section{The sets we want to sieve}
We first list the sets that we will work with later. Let $\theta=0.9409$ in the following sections. Put
\begin{align}
\nonumber \mathcal{A}_1=&\left\{\frac{N-a p}{b}: p \leqslant \frac{N}{a},\ (p,a b N)=1,\right. \\
\nonumber & \quad \left. p \equiv N a_{b^{2}}^{-1}+k b \left(\bmod b^{2}\right),\ 0 \leqslant k \leqslant b-1,\ (k, b)=1\right\}, \\
\nonumber \mathcal{A}_2=&\left\{\frac{N-a p}{b}: p \leqslant \frac{N^{\theta}}{a},\ (p,a b N)=1,\right. \\
\nonumber & \quad \left. p \equiv N a_{b^{2}}^{-1}+k b \left(\bmod b^{2}\right),\ 0 \leqslant k \leqslant b-1,\ (k, b)=1\right\},\\
\nonumber \mathcal{A}_3=&\left\{\frac{N-a p}{b}: \frac{N/2-N^{0.97}}{a} \leqslant p \leqslant \frac{N/2+N^{0.97}}{a},\ (p,a b N)=1,\right. \\
\nonumber & \quad \left. p \equiv N a_{b^{2}}^{-1}+k b \left(\bmod b^{2}\right),\ 0 \leqslant k \leqslant b-1,\ (k, b)=1\right\},\\
\nonumber \mathcal{A}_4=&\left\{\frac{N-a p}{b}: p \leqslant \frac{N}{a},\ (p,a b N)=1,\ p \equiv d(\bmod c), (c,d)=1, \right. \\
\nonumber & \quad \left. \left(\frac{N-ad}{b},c\right)=1,\ p \equiv N a_{b^{2}}^{-1}+k b \left(\bmod b^{2}\right),\ 0 \leqslant k \leqslant b-1,\ (k, b)=1\right\}, \\
\nonumber \mathcal{A}_5=&\left\{\frac{N-a p}{b}: p \leqslant \frac{N^{\theta}}{a},\ (p,a b N)=1,\ p \equiv d(\bmod c),\ (c,d)=1, \right. \\
\nonumber & \quad \left. \left(\frac{N-ad}{b},c\right)=1,\ p \equiv N a_{b^{2}}^{-1}+k b \left(\bmod b^{2}\right),\ 0 \leqslant k \leqslant b-1,\ (k, b)=1\right\}, \\
\nonumber \mathcal{A}_6=&\left\{\frac{N-a p}{b}: \frac{N/2-N^{0.97}}{a} \leqslant p \leqslant \frac{N/2+N^{0.97}}{a},\ (p,a b N)=1,\ p \equiv d(\bmod c),  \right. \\
\nonumber & \quad \left. (c,d)=1,\ \left(\frac{N-ad}{b},c\right)=1,\ p \equiv N a_{b^{2}}^{-1}+k b \left(\bmod b^{2}\right),\ 0 \leqslant k \leqslant b-1,\ (k, b)=1\right\}, \\
\nonumber \mathcal{B}_1= & \left\{\frac{N-b p_{1} p_{2} p_{3}}{a}: \left(p_{1} p_{2} p_{3}, a b N\right)=1,\ p_{3} \leqslant \frac{N}{b p_{1} p_{2}},\right. \\
\nonumber & \quad \left(\frac{N}{b}\right)^{\frac{1}{13.2}} \leqslant p_1 <\left(\frac{N}{b}\right)^{\frac{1}{3}} \leqslant p_2 <\left(\frac{N}{b p_{1}}\right)^{\frac{1}{2}},\\ 
\nonumber & \quad \left. p_{3} \equiv N\left(b p_{1} p_{2}\right)_{a^{2}}^{-1}+j a \left(\bmod a^{2}\right),\ 0 \leqslant j \leqslant a-1,\ (j, a)=1\right\},\\
\nonumber \mathcal{B}_2= & \left\{\frac{N-b p_{1} p_{2} p_{3}}{a}: \left(p_{1} p_{2} p_{3}, a b N\right)=1,\ \frac{N-N^{\theta}}{b p_{1} p_{2}} \leqslant p_{3} \leqslant \frac{N}{b p_{1} p_{2}},\right. \\
\nonumber & \quad \left(\frac{N}{b}\right)^{\frac{1}{14}} \leqslant p_1 <\left(\frac{N}{b}\right)^{\frac{1}{3.1}} \leqslant p_2 <\left(\frac{N}{b p_{1}}\right)^{\frac{1}{2}},\\ 
\nonumber & \quad \left. p_{3} \equiv N\left(b p_{1} p_{2}\right)_{a^{2}}^{-1}+j a \left(\bmod a^{2}\right),\ 0 \leqslant j \leqslant a-1,\ (j, a)=1\right\},\\
\nonumber\mathcal{C}_1=&\left\{\frac{N-bmp_1p_2p_3p_4}{a}: \left(p_{1} p_{2} p_{4},\ abN\right)=1,\ \left(\frac{N}{b}\right)^{\frac{1}{13.2}} \leqslant p_{1}<p_{4}<p_{2}<\left(\frac{N}{b}\right)^{\frac{1}{8.4}}, \right.\\
\nonumber& \quad \left. 1 \leqslant m \leqslant \frac{N}{bp_{1} p_{2}^{2} p_{4}},\ \left(m, p_{1}^{-1} abN P\left(p_{4}\right)\right)=1,\right. \\ 
\nonumber& \quad \left. p_{3} \equiv N\left(bm p_{1} p_{2} p_4\right)_{a^{2}}^{-1}+j a \left(\bmod a^{2}\right),\ 0 \leqslant j \leqslant a-1,\ (j, a)=1,\right.\\
\nonumber& \quad \left. p_2<p_3<\min\left(\left(\frac{N}{b}\right)^\frac{1}{8.4},\frac{N}{bmp_1p_2p_4}\right)
\right\},\\
\nonumber\mathcal{C}_2=&\left\{\frac{N-bmp_1p_2p_3p_4}{a}: \left(p_{1} p_{2} p_3 p_{4}, abN\right)=1,\ \left(\frac{N}{b}\right)^{\frac{1}{14}} \leqslant p_{1}<p_{2}<p_3<p_{4}<\left(\frac{N}{b}\right)^{\frac{1}{8.8}},\right. \\
\nonumber& \quad \left.m p_{1} p_{2} p_3 p_{4} \equiv N b_{a^{2}}^{-1} +j a \left(\bmod a^{2}\right),\ 0 \leqslant j \leqslant a-1,\ (j, a)=1,\right.\\
\nonumber& \quad \left. \frac{N-N^{\theta}}{bp_{1} p_{2} p_3 p_{4}} \leqslant m \leqslant \frac{N}{bp_{1} p_{2} p_3 p_{4}},\ \left(m, p_{1}^{-1} abN P\left(p_{2}\right)\right)=1\right\},\\
\nonumber\mathcal{C}_3=&\left\{\frac{N-bmp_1p_2p_3p_4}{a}: \left(p_{1} p_{2} p_3 p_{4}, abN\right)=1,\ \left(\frac{N}{b}\right)^{\frac{1}{14}} \leqslant p_{1}<p_{2}<p_3<\left(\frac{N}{b}\right)^{\frac{1}{8.8}}\leqslant p_4<\left(\frac{N}{b}\right)^{\frac{4.5863}{14}}p_3^{-1},\right. \\
\nonumber& \quad \left.m p_{1} p_{2} p_3 p_{4} \equiv N b_{a^{2}}^{-1} +j a \left(\bmod a^{2}\right),\ 0 \leqslant j \leqslant a-1,\ (j, a)=1,\right.\\
\nonumber& \quad \left. \frac{N-N^{\theta}}{bp_{1} p_{2} p_3 p_{4}} \leqslant m \leqslant \frac{N}{bp_{1} p_{2} p_3 p_{4}},\ \left(m, p_{1}^{-1} abN P\left(p_{2}\right)\right)=1\right\},\\
\nonumber\mathcal{E}_1=&\left\{p_1 p_2:\left(p_{1} p_{2}, abN\right)=1,\ \left(\frac{N}{b}\right)^{\frac{1}{13.2}} \leqslant p_{1}<\left(\frac{N}{b}\right)^{\frac{1}{3}}\leqslant p_{2}<\left(\frac{N}{bp_1}\right)^{\frac{1}{2}}
\right\},\\
\nonumber\mathcal{E}_2=&\left\{p_1 p_2:\left(p_{1} p_{2}, abN\right)=1,\ \left(\frac{N}{b}\right)^{\frac{1}{14}} \leqslant p_{1}<\left(\frac{N}{b}\right)^{\frac{1}{3.1}}\leqslant p_{2}<\left(\frac{N}{bp_1}\right)^{\frac{1}{2}}
\right\},\\
\nonumber\mathcal{F}_1=&\left\{mp_1 p_2 p_4: \left(p_{1} p_{2} p_{4}, abN\right)=1,\ \left(\frac{N}{b}\right)^{\frac{1}{13.2}} \leqslant p_{1}<p_{4}<p_{2}<\left(\frac{N}{b}\right)^{\frac{1}{8.4}},\right.\\
\nonumber& \quad \left. 1 \leqslant m \leqslant \frac{N}{bp_{1} p_{2}^{2} p_{4}},\ \left(m, p_{1}^{-1} abN P\left(p_{4}\right)\right)=1
\right\},\\
\nonumber\mathcal{F}_2=&\left\{mp_1 p_2 p_3 p_4:\left(p_{1} p_{2} p_3 p_{4}, abN\right)=1,\ \left(\frac{N}{b}\right)^{\frac{1}{14}} \leqslant p_{1}<p_{2}<p_3<p_{4}<\left(\frac{N}{b}\right)^{\frac{1}{8.8}},\right. \\ 
\nonumber& \quad \left. \frac{N-N^{\theta}}{bp_{1} p_{2} p_3 p_{4}} \leqslant m \leqslant \frac{N}{bp_{1} p_{2} p_3 p_{4}},\ \left(m, p_{1}^{-1} abN P\left(p_{2}\right)\right)=1
\right\},\\
\nonumber\mathcal{F}_3=&\left\{mp_1 p_2 p_3 p_4:\left(p_{1} p_{2} p_3 p_{4}, abN\right)=1,\ \left(\frac{N}{b}\right)^{\frac{1}{14}} \leqslant p_{1}<p_{2}<p_3<\left(\frac{N}{b}\right)^{\frac{1}{8.8}}\leqslant p_4<\left(\frac{N}{b}\right)^{\frac{4.5863}{14}}p_3^{-1},\right. \\
\nonumber& \quad \left. \frac{N-N^{\theta}}{bp_{1} p_{2} p_3 p_{4}} \leqslant m \leqslant \frac{N}{bp_{1} p_{2} p_3 p_{4}},\ \left(m, p_{1}^{-1} abN P\left(p_{2}\right)\right)=1\right\},
\end{align}
where $a_{b^{2}}^{-1}$ is the multiplicative inverse of $a \bmod b^{2}$, which exists by our assumption $(a, b)=1$.

\section{Preliminary lemmas}
Let $\mathcal{A}$ denote a finite set of positive integers, $\mathcal{P}$ denote an infinite set of primes and $z \geqslant 2$. Suppose that $|\mathcal{A}| \sim X_{\mathcal{A}}$ and for square-free $d$, put
$$
\mathcal{P}=\{p : (p, N)=1\},\quad
\mathcal{P}(r)=\{p : p \in \mathcal{P},(p, r)=1\},
$$
$$
P(z)=\prod_{\substack{p\in \mathcal{P}\\p<z}} p,\quad
\mathcal{A}_{d}=\{a : a d \in \mathcal{A} \},\quad
S(\mathcal{A}; \mathcal{P},z)=\sum_{\substack{a \in \mathcal{A} \\ (a, P(z))=1}} 1.
$$

\begin{lemma}\label{l1} ([\cite{Kan2}, Lemma 1]). If
$$
\sum_{z_{1} \leqslant p<z_{2}} \frac{\omega(p)}{p}=\log \frac{\log z_{2}}{\log z_{1}}+O\left(\frac{1}{\log z_{1}}\right), \quad z_{2}>z_{1} \geqslant 2,
$$
where $\omega(d)$ is a multiplicative function, $0 \leqslant \omega(p)<p, X>1$ is independent of $d$. Then
$$
S(\mathcal{A}; \mathcal{P}, z) \geqslant X_{\mathcal{A}} W(z)\left\{f\left(\frac{\log D}{\log z}\right)+O\left(\frac{1}{\log ^{\frac{1}{3}} D}\right)\right\}-\sum_{\substack{n\leqslant D \\ n \mid P(z)}}\left|\eta\left(X_{\mathcal{A}}, n\right)\right|
$$
$$
S(\mathcal{A}; \mathcal{P}, z) \leqslant X_{\mathcal{A}} W(z)\left\{F\left(\frac{\log D}{\log z}\right)+O\left(\frac{1}{\log ^{\frac{1}{3}} D}\right)\right\}+\sum_{\substack{n\leqslant D \\ n \mid P(z)}}\left|\eta\left(X_{\mathcal{A}}, n\right)\right|
$$
where
$$
W(z)=\prod_{\substack{p<z \\ (p,N)=1}}\left(1-\frac{\omega(p)}{p}\right),\quad \eta(X_{\mathcal{A}}, n)=|\mathcal{A}_n|-\frac{\omega(n)}{n} X_\mathcal{A}=\sum_{\substack{a \in \mathcal{A} \\ a \equiv 0(\bmod n)}} 1-\frac{\omega(n)}{n} X_\mathcal{A},
$$
$\gamma$ denote the Euler's constant, $f(s)$ and $F(s)$ are determined by the following differential-difference equation
\begin{align*}
\begin{cases}
F(s)=\frac{2 e^{\gamma}}{s}, \quad f(s)=0, \quad &0<s \leqslant 2,\\
(s F(s))^{\prime}=f(s-1), \quad(s f(s))^{\prime}=F(s-1), \quad &s \geqslant 2 .
\end{cases}
\end{align*}
\end{lemma}

\begin{lemma}\label{l2} ([\cite{CAI867}, Lemma 2]).
\begin{align*}
F(s)=&\frac{2 e^{\gamma}}{s}, \quad 0<s \leqslant 3;\\
F(s)=&\frac{2 e^{\gamma}}{s}\left(1+\int_{2}^{s-1} \frac{\log (t-1)}{t} d t\right), \quad 3 \leqslant s \leqslant 5 ;\\
F(s)=&\frac{2 e^{\gamma}}{s}\left(1+\int_{2}^{s-1} \frac{\log (t-1)}{t} d t+\int_{2}^{s-3} \frac{\log (t-1)}{t} d t \int_{t+2}^{s-1} \frac{1}{u} \log \frac{u-1}{t+1} d u\right), \quad 5 \leqslant s \leqslant 7;\\
f(s)=&\frac{2 e^{\gamma} \log (s-1)}{s}, \quad 2 \leqslant s \leqslant 4 ;\\
f(s)=&\frac{2 e^{\gamma}}{s}\left(\log (s-1)+\int_{3}^{s-1} \frac{d t}{t} \int_{2}^{t-1} \frac{\log (u-1)}{u} d u\right), \quad 4 \leqslant s \leqslant 6 ;\\
f(s)=& \frac{2 e^{\gamma}}{s}\left(\log (s-1)+\int_{3}^{s-1} \frac{d t}{t} \int_{2}^{t-1}\frac{\log (u-1)}{u} d u\right.\\ & \left.+\int_{2}^{s-4} \frac{\log (t-1)}{t} d t \int_{t+2}^{s-2} \frac{1}{u} \log \frac{u-1}{t+1} \log \frac{s}{u+2} d u\right), \quad 6 \leqslant s \leqslant 8.
\end{align*}
\end{lemma}

\begin{lemma}\label{l4} ([\cite{CAI867}, Lemma 4]). Let
$$
x>1, \quad z=x^{\frac{1}{u}}, \quad Q(z)=\prod_{p<z} p.
$$
Then for $u \geqslant 1$, we have
$$
\sum_{\substack{n \leqslant x \\(n, Q(z))=1}} 1=w(u) \frac{x}{\log z}+O\left(\frac{x}{\log ^{2} z}\right),
$$
where $w(u)$ is determined by the following differential-difference equation
\begin{align*}
\begin{cases}
w(u)=\frac{1}{u}, & \quad 1 \leqslant u \leqslant 2, \\
(u w(u))^{\prime}=w(u-1), & \quad u \geqslant 2 .
\end{cases}
\end{align*}
Moreover, we have
$$
\begin{cases}w(u) \leqslant \frac{1}{1.763}, & u \geqslant 2, \\ w(u)<0.5644, & u \geqslant 3, \\ w(u)<0.5617, & u \geqslant 4.\end{cases}
$$
\end{lemma}

\begin{lemma}\label{buchstabshort} ([\cite{CL2011}, Lemma 2.6], [\cite{CL1999}, Lemma 4]). Let
$$
x>1, \quad x^{\frac{19}{24}+\varepsilon} \leqslant y_1 \leqslant \frac{x}{\log x}, \quad x^{\frac{3}{5}} \leqslant y_2 <x, \quad z=x^{\frac{1}{u}}, \quad Q(z)=\prod_{p<z} p.
$$
Then for $u > 1$, we have
$$
\sum_{\substack{x-y_1 \leqslant n \leqslant x \\(n, Q(z))=1}} 1=w(u) \frac{y_1}{\log z}+O\left(\frac{y_1}{\log ^{2} z}\right),
$$
$$
\sum_{\substack{x \leqslant n < x+y_2 \\(n, Q(z))=1}} 1=w(u) \frac{y_2}{\log z}+O\left(\frac{y_2}{\log ^{2} z}\right),
$$
where $w(u)$ is defined in Lemma~\ref{l4}.
\end{lemma}

\begin{lemma}\label{Wfunction}
If we define the function $\omega$ as $\omega(p)=0$ for primes $p \mid a b N$ and $\omega(p)=\frac{p}{p-1}$ for other primes and $N^{\frac{1}{\alpha}-\varepsilon}<z \leqslant N^{\frac{1}{\alpha}}$, then we have
$$
W(z)=\frac{2\alpha e^{-\gamma} C(abN)(1+o(1))}{\log N}.
$$
\end{lemma}
\begin{proof} By  [\cite{LIHUIXI}, Lemma 2] we have
$$
W(z)=\frac{N}{\varphi(N)} \prod_{(p,N)=1}\left(1-\frac{\omega(p)}{p}\right)\left(1-\frac{1}{p}\right)^{-1} \frac{e^{-\gamma}}{\log z}\left(1+O\left(\frac{1}{\log z}\right)\right).
$$
Since $2 \mid a b N$, we have
\begin{align}
\nonumber W(z)&=\frac{N}{\varphi(N)} \prod_{(p,N)=1}\left(1-\frac{\omega(p)}{p}\right)\left(1-\frac{1}{p}\right)^{-1} \frac{e^{-\gamma}}{\log z}\left(1+O\left(\frac{1}{\log z}\right)\right) \\
\nonumber & =\prod_{p \mid N} \frac{p}{p-1} \prod_{(p,N)=1} \left(1-\frac{\omega(p)}{p}\right)\left(1-\frac{1}{p}\right)^{-1} \frac{\alpha e^{-\gamma}(1+o(1))}{\log N} \\
\nonumber & =\prod_{p \mid N} \frac{p}{p-1} \prod_{p \mid a b} \left(1-\frac{1}{p}\right)^{-1} \prod_{(p,abN)=1} \left(1-\frac{1}{p-1}\right)\left(1-\frac{1}{p}\right)^{-1} \frac{\alpha e^{-\gamma}(1+o(1))}{\log N} \\
\nonumber & =\prod_{\substack{p \mid N \\
p>2}} \frac{p}{p-1} \prod_{\substack{p \mid a b \\
p>2}} \frac{p}{p-1} \frac{\prod_{p>2} \frac{p(p-2)}{(p-1)^{2}}}{\prod_{\substack{p \mid a b N \\
p>2}} \frac{p(p-2)}{(p-1)^{2}}} \frac{2\alpha e^{-\gamma}(1+o(1))}{\log N} \\
\nonumber & =\frac{2\alpha e^{-\gamma} C(abN)(1+o(1))}{\log N}.
\end{align}
\end{proof}

\section{Mean value theorems}
Now we provide some mean value theorems which will be used in bounding various sieve error terms later. The first two lemmas come from Pan and Pan's book \cite{ERPAN} and they were first proven by Pan, Ding and Wang.
\begin{lemma}\label{l3} ([\cite{ERPAN}, Corollary 8.2]). Let
$$
\pi(x ; k, d, l)=\sum_{\substack{k p \leq x \\ k p \equiv l(\bmod d)}} 1
$$
and let $g(k)$ be a real function, $g(k) \ll 1$. Then, for any given constant $A>0$, there exists a constant $B=B(A)>0$ such that
$$
\sum_{d \leqslant x^{1/2} (\log x)^{-B}}\max _{y \leqslant x}\max _{(l, d)=1} \left|\sum_{\substack{k \leqslant E(x) \\ (k, d)=1}} g(k) H(y ; k, d, l)\right| \ll \frac{x}{\log ^{A} x},
$$
where
$$
H(y ; k, d, l)=\pi(y ; k, d, l)-\frac{1}{\varphi(d)}\pi(y ; k, 1, 1)=\sum_{\substack{k p \leqslant y \\ k p \equiv l (\bmod d)}} 1-\frac{1}{\varphi(d)} \sum_{k p \leqslant y} 1, 
$$
$$
\frac{1}{2} \leqslant E(x) \ll x^{1-\alpha}, \quad 0<\alpha \leqslant 1,\quad B(A)=\frac{3}{2}A+17.
$$
\end{lemma}
\begin{lemma}\label{remark1} ([\cite{ERPAN}, Corollary 8.3 and 8.4]). Let $r_{1}(y)$ be a positive function depending on $x$ and satisfying $r_{1}(y) \ll x^{\alpha}$ for $y \leqslant x$. Then under the conditions in Lemma~\ref{l3}, we have
$$
\sum_{d \leqslant x^{1/2} (\log x)^{-B}} \max _{y \leqslant x}\max _{(l, d)=1} \left|\sum_{\substack{k \leqslant E(x) \\ (k, d)=1}} g(k) H\left(kr_{1}(y) ; k, d, l\right)\right| \ll \frac{x}{\log ^{A} x} .
$$
Let $r_{2}(k)$ be a positive function depending on $x$ and $y$ such that $k r_{2}(k) \ll x$ for $k \leqslant E(x)$ , $y \leqslant x$. Then under the conditions in Lemma~\ref{l3}, we have
$$
\sum_{d \leqslant x^{1/2} (\log x)^{-B}} \max _{y \leqslant x}\max _{(l, d)=1} \left|\sum_{\substack{k \leqslant E(x) \\ (k, d)=1}} g(k) H\left(kr_{2}(k) ; k, d, l\right)\right| \ll \frac{x}{\log ^{A} x} .
$$
\end{lemma}

The next two lemmas were first proven by Wu \cite{Wu1993}, and they are the "short interval" version of Lemmas~\ref{l3}--\ref{remark1}. These will help us deal with the sieve error terms involved in evaluation of $S_{4}^{\prime}$ and $S_{7}^{\prime}$.
\begin{lemma}\label{BVshort}
([\cite{Wu1993}, Theorem 2]). Let $g(k)$ be a real function such that
$$
\sum_{k \leqslant x} \frac{g^{2}(k)}{k} \ll \log ^{C} x
$$
for some $C>0$. Then, for any given constant $A>0$, there exists a constant $B=B(A,C)>0$ such that
$$
\sum_{d \leqslant x^{t-1/2} (\log x)^{-B}} \max _{x/2 \leqslant y \leqslant x}\max _{(l, d)=1}\max _{h \leqslant x^t} \left|\sum_{\substack{k \leqslant x^\beta \\ (k, d)=1}} g(k) H_1(y ,h, k, d, l)\right| \ll \frac{x^t}{\log ^{A} x},
$$
where 
$$
\begin{aligned}
H_1(y ,h, k, d, l)=&\left(\pi(y+h ; k, d, l)-\pi(y ; k, d, l)\right)\\
&-\frac{1}{\varphi(d)}\left(\pi(y+h ; k, 1, 1)-(\pi(y ; k, 1, 1)\right)\\
=&\sum_{\substack{y<k p \leqslant y+h \\ k p \equiv l (\bmod d)}} 1-\frac{1}{\varphi(d)} \sum_{y<k p \leqslant y+h} 1,
\end{aligned}
$$
$$
\frac{3}{5} < t \leqslant 1, \quad 0 \leqslant \beta<\frac{5 t-3}{2},\quad B(A,C)=3A+C+34.
$$
\end{lemma}
\begin{lemma}\label{remark2}
([\cite{Cai2015}, Lemma 7], [\cite{CL1999}, Remark]). Let $g(k)$ be a real function such that
$$
\sum_{k \leqslant x} \frac{g^{2}(k)}{k} \ll \log ^{C} x
$$
for some $C>0$. Let $r_1(k,h)$ and $r_2(k,h)$ be positive function such that
$$
y \leqslant kr_1(k,h), kr_2(k,h) \leqslant y+h.
$$
Then, for any given constant $A>0$, there exists a constant $B=B(A,C)>0$ such that
$$
\sum_{d \leqslant x^{t-1/2} (\log x)^{-B}}\max _{x/2 \leqslant y \leqslant x}\max _{(l, d)=1}\max _{h \leqslant x^t} \left|\sum_{\substack{k \leqslant x^\beta \\ (k, d)=1}} g(k) H_2(y ,h, k, d, l)\right| \ll \frac{x^t}{\log ^{A} x},
$$
where
$$
\begin{aligned}
H_2(y ,h, k, d, l)=&\left(\pi(kr_2(k,h) ; k, d, l)-\pi(kr_1(k,h) ; k, d, l)\right)\\
&-\frac{1}{\varphi(d)}\left(\pi(kr_2(k,h); k, 1, 1)-(\pi(kr_1(k,h) ; k, 1, 1)\right)\\
=&\sum_{\substack{kr_1(k,h)<k p \leqslant kr_2(k,h) \\ k p \equiv l (\bmod d)}} 1-\frac{1}{\varphi(d)} \sum_{kr_1(k,h)<k p \leqslant kr_2(k,h)} 1,
\end{aligned}
$$
$$
\frac{3}{5} < t \leqslant 1, \quad 0 \leqslant \beta<\frac{5 t-3}{2},\quad B(A,C)=3A+C+34.
$$
\end{lemma}

In \cite{Cai2015}, Cai said that we faced the difficulty that cannot be overcome by our Lemmas~\ref{BVshort}--\ref{remark2} which are not sufficient to deal with some of the sieve error terms involved. Actually, the function $g(k)$ cannot be well defined to control the sieve error terms occurred in the evaluation of $S_{6}^{\prime}$. (i.e. $\frac{5\theta-3}{2}<\frac{13}{14}$). So we need a new mean value theorem to overcome that. The next lemma is a new mean value theorem for products of large primes over short intervals and it was first proven by Cai \cite{Cai2015}. This lemma will help us deal with the sieve error terms involved in evaluation of $S_{6}^{\prime}$.
\begin{lemma}\label{newmeanvalue}For $j=2,3$ and any given constant $A>0$, there exists a constant $B=B(A)>0$ such that 
$$
\sum_{d \leqslant x^{\theta-1/2} (\log x)^{-B}}\max _{(l, d)=1}\left|\sum_{\substack{mp_1 p_2 p_3 p_4 \in \mathcal{F}_j \\ m p_1 p_2 p_3 p_4 \equiv l(\bmod d)}} 1-\frac{1}{\varphi(d)} \sum_{\substack{m p_1 p_2 p_3 p_4 \in \mathcal{F}_j \\(m p_1 p_2 p_3 p_4, d)=1}} 1\right| \ll \frac{x^\theta}{\log ^{A} x} .
$$
\end{lemma}
\begin{proof}
This result can be proved in the same way as [\cite{Cai2015}, Lemma 8] by showing that for $j=2,3$ and $5 \leqslant r \leqslant 14$, the bounds
$$
\sum_{d \leqslant x^{\theta-1/2} (\log x)^{-B}}\max _{(l, d)=1}\left|\sum_{\substack{p_1 p_2 \cdots  p_r \in \mathcal{F}_j \\p_1 p_2 \cdots  p_r \equiv l(\bmod d)}} 1-\frac{1}{\varphi(d)} \sum_{\substack{p_1 p_2 \cdots  p_r \in \mathcal{F}_j \\(p_1 p_2 \cdots  p_r, d)=1}} 1\right| \ll \frac{x^\theta}{\log ^{A} x}
$$
hold.
\end{proof}

The following lemmas are the "arithmetical progression with almost all large $c$" versions of the above lemmas, and they will help us prove Theorem~\ref{t41}. We can also obtain variants of Theorems~\ref{t42}--\ref{t43} with "large" $c$ by using the following lemmas.
\begin{lemma}\label{almostBV}
([\cite{CL1999_2}, Lemma 4]). For any given constant $A>0$,  under the conditions in Lemmas~\ref{l3}--\ref{remark1}, there exists a constant $B=B(A)>0$ such that for $c \leqslant x^{0.028}$, except for $O\left(x^{0.028} (\log x)^{-A}\right)$ exceptional values, we have
$$
R_1=\sum_{d \leqslant \left(x^{1/2} (\log x)^{-B}\right)/c}\max _{y \leqslant x}\max _{(l, dc)=1} \left|\sum_{\substack{k \leqslant E(x) \\ (k, d)=1}} g(k) H(y ; k, dc, l)\right| \ll \frac{x^{1-0.028}}{\log ^{A} x},
$$
$$
R_2=\sum_{d \leqslant \left(x^{1/2} (\log x)^{-B}\right)/c}\max _{y \leqslant x}\max _{(l, dc)=1} \left|\sum_{\substack{k \leqslant E(x) \\ (k, d)=1}} g(k) H(kr_{1}(y) ; k, dc, l)\right| \ll \frac{x^{1-0.028}}{\log ^{A} x},
$$
$$
R_3=\sum_{d \leqslant \left(x^{1/2} (\log x)^{-B}\right)/c}\max _{y \leqslant x}\max _{(l, dc)=1} \left|\sum_{\substack{k \leqslant E(x) \\ (k, d)=1}} g(k) H(kr_{2}(k) ; k, dc, l)\right| \ll \frac{x^{1-0.028}}{\log ^{A} x}.
$$
\end{lemma}
\begin{proof}
We prove Lemma~\ref{almostBV} in the case $R_1$ only, the same argument can be
applied to the cases $R_2$ and $R_3$. Let $\tau(d)$ denote the divisor function, By Lemma~\ref{l3} and similar arguments as in [\cite{MengXianmeng}, Lemma 3], we have
\begin{align}
\nonumber \sum_{c \leqslant N^{0.028}} R_1 =& \sum_{c \leqslant N^{0.028}} \sum_{d \leqslant \left(x^{1/2} (\log x)^{-B}\right)/c}\max _{y \leqslant x}\max _{(l, dc)=1} \left|\sum_{\substack{k \leqslant E(x) \\ (k, d)=1}} g(k) H(y ; k, dc, l)\right| \\
\nonumber \leqslant & \sum_{d \leqslant x^{1/2} (\log x)^{-B}} \tau(d) \max _{y \leqslant x}\max _{(l, d)=1} \left|\sum_{\substack{k \leqslant E(x) \\ (k, d)=1}} g(k) H(y ; k, d, l)\right| \ll \frac{x}{\log ^{2A} x},\\
\nonumber \sum_{\substack{c \leqslant N^{0.028} \\ R_1> \frac{x^{1-0.028}}{\log ^{A} x}}} 1 \ll& \frac{\log ^{A} x}{x^{1-0.028}} \sum_{c \leqslant N^{0.028}} R_1 \ll \frac{x^{0.028}}{\log ^{A} x}.
\end{align}
Now the proof of Lemma~\ref{almostBV} is completed.
\end{proof}
\begin{lemma}\label{almostBVshort}
For any given constant $A>0$,  under the conditions in Lemmas~\ref{BVshort}--\ref{remark2}, there exists a constant $B=B(A,C)>0$ such that for $c \leqslant x^{0.028}$, except for $O\left(x^{0.028} (\log x)^{-A}\right)$ exceptional values, we have
$$
R_4=\sum_{d \leqslant \left(x^{t-1/2} (\log x)^{-B}\right)/c} \max _{x/2 \leqslant y \leqslant x}\max _{(l, dc)=1}\max _{h \leqslant x^t} \left|\sum_{\substack{k \leqslant x^\beta \\ (k, d)=1}} g(k) H_1(y ,h, k, dc, l)\right| \ll \frac{x^{t-0.028}}{\log ^{A} x},
$$
$$
R_5=\sum_{d \leqslant \left(x^{t-1/2} (\log x)^{-B}\right)/c}\max _{x/2 \leqslant y \leqslant x}\max _{(l, dc)=1}\max _{h \leqslant x^t} \left|\sum_{\substack{k \leqslant x^\beta \\ (k, d)=1}} g(k) H_2(y ,h, k, dc, l)\right| \ll \frac{x^{t-0.028}}{\log ^{A} x}.
$$
\end{lemma}
\begin{proof}
We prove Lemma~\ref{almostBVshort} in the case $R_4$ only, the same argument can be
applied to the case $R_5$. By Lemma~\ref{BVshort} and similar arguments as in [\cite{MengXianmeng}, Lemma 3], we have
\begin{align}
\nonumber \sum_{c \leqslant N^{0.028}} R_4 =& \sum_{c \leqslant N^{0.028}} \sum_{d \leqslant \left(x^{t-1/2} (\log x)^{-B}\right)/c} \max _{x/2 \leqslant y \leqslant x}\max _{(l, dc)=1}\max _{h \leqslant x^t} \left|\sum_{\substack{k \leqslant x^\beta \\ (k, d)=1}} g(k) H_1(y ,h, k, dc, l)\right| \\
\nonumber \leqslant & \sum_{d \leqslant x^{t-1/2} (\log x)^{-B}} \tau(d) \max _{x/2 \leqslant y \leqslant x}\max _{(l, d)=1}\max _{h \leqslant x^t} \left|\sum_{\substack{k \leqslant x^\beta \\ (k, d)=1}} g(k) H_1(y ,h, k, d, l)\right| \ll \frac{x^t}{\log ^{2A} x},\\
\nonumber \sum_{\substack{c \leqslant N^{0.028} \\ R_4> \frac{x^{t-0.028}}{\log ^{A} x}}} 1 \ll& \frac{\log ^{A} x}{x^{t-0.028}} \sum_{c \leqslant N^{0.028}} R_4 \ll \frac{x^{0.028}}{\log ^{A} x}.
\end{align}
Now the proof of Lemma~\ref{almostBVshort} is completed.
\end{proof}
\begin{lemma}\label{almostnewmean} For $j=2,3$, let
$$
\mathcal{F}^{\prime}_j=\left\{mp_1 p_2 p_3 p_4: mp_1 p_2 p_3 p_4 \in \mathcal{F}_j, (p_1 p_2 p_3 p_4, c)=1\right\},
$$ 
then for any given constant $A>0$, there exists a constant $B=B(A)>0$ such that for $c \leqslant x^{0.028}$, except for $O\left(x^{0.028} (\log x)^{-A}\right)$ exceptional values, we have
$$
R^{\prime}_j=\sum_{d \leqslant \left(x^{\theta-1/2} (\log x)^{-B}\right)/c}\max _{(l, dc)=1}\left|\sum_{\substack{mp_1 p_2 p_3 p_4 \in \mathcal{F}^{\prime}_j \\ m p_1 p_2 p_3 p_4 \equiv l(\bmod dc)}} 1-\frac{1}{\varphi(dc)} \sum_{\substack{m p_1 p_2 p_3 p_4 \in \mathcal{F}^{\prime}_j \\(m p_1 p_2 p_3 p_4, dc)=1}} 1\right| \ll \frac{x^{\theta-0.028}}{\log ^{A} x} .
$$
\end{lemma}
\begin{proof}
We prove Lemma~\ref{almostnewmean} in the case $R^{\prime}_2$ only, the same argument can be
applied to the case $R^{\prime}_3$. By Lemma~\ref{newmeanvalue} and similar arguments as in [\cite{MengXianmeng}, Lemma 3], we have
\begin{align}
\nonumber \sum_{c \leqslant N^{0.028}} R^{\prime}_2 =& \sum_{c \leqslant N^{0.028}} \sum_{d \leqslant \left(x^{\theta-1/2} (\log x)^{-B}\right)/c}\max _{(l, dc)=1}\left|\sum_{\substack{mp_1 p_2 p_3 p_4 \in \mathcal{F}^{\prime}_j \\ m p_1 p_2 p_3 p_4 \equiv l(\bmod dc)}} 1-\frac{1}{\varphi(dc)} \sum_{\substack{m p_1 p_2 p_3 p_4 \in \mathcal{F}^{\prime}_j \\(m p_1 p_2 p_3 p_4, dc)=1}} 1\right| \\
\nonumber \leqslant & \sum_{d \leqslant x^{\theta-1/2} (\log x)^{-B}} \tau(d) \max _{(l, d)=1}\left|\sum_{\substack{mp_1 p_2 p_3 p_4 \in \mathcal{F}_j \\ m p_1 p_2 p_3 p_4 \equiv l(\bmod d)}} 1-\frac{1}{\varphi(d)} \sum_{\substack{m p_1 p_2 p_3 p_4 \in \mathcal{F}_j \\(m p_1 p_2 p_3 p_4, d)=1}} 1\right| \ll \frac{x^\theta}{\log ^{2A} x},\\
\nonumber \sum_{\substack{c \leqslant N^{0.028} \\ R^{\prime}_2> \frac{x^{\theta-0.028}}{\log ^{A} x}}} 1 \ll& \frac{\log ^{A} x}{x^{\theta-0.028}} \sum_{c \leqslant N^{0.028}} R^{\prime}_2 \ll \frac{x^{0.028}}{\log ^{A} x}.
\end{align}
Now the proof of Lemma~\ref{almostnewmean} is completed.
\end{proof}

\section{Weighted sieve method}
Now we provide the delicate weighted sieves in order to prove our Theorems~\ref{t1}--\ref{t5}.
\begin{lemma}\label{l31} Let $\mathcal{A}=\mathcal{A}_1$ in section 2 and $0<\alpha<\beta \leqslant \frac{1}{3}$. Then we have
\begin{align*}
2R_{a, b}(N) \geqslant& 2S\left(\mathcal{A};\mathcal{P}, \left(\frac{N}{b}\right)^{\alpha}\right)-\sum_{\substack{(\frac{N}{b})^{\alpha} \leqslant p<(\frac{N}{b})^{\beta} \\ (p, N)=1}} S\left(\mathcal{A}_{p};\mathcal{P}, \left(\frac{N}{b}\right)^{\alpha}\right)\\
&-\sum_{\substack{(\frac{N}{b})^{\alpha} \leqslant p_1<(\frac{N}{b})^{\beta} \leqslant p_2 <(\frac{N}{bp_1})^{\frac{1}{2}} \\ (p_1 p_2, N)=1}} S\left(\mathcal{A}_{p_1 p_2};\mathcal{P}(p_1),p_2\right)
- 2\sum_{\substack{(\frac{N}{b})^{\beta} \leqslant p_1 < p_2 <(\frac{N}{bp_1})^{\frac{1}{2}} \\ (p_1 p_2, N)=1} }S\left(\mathcal{A}_{p_1 p_2};\mathcal{P}(p_1),p_2\right)\\
&+\sum_{\substack{(\frac{N}{b})^{\alpha} \leqslant p_1 < p_2 < p_3<(\frac{N}{b})^{\beta} \\ (p_1 p_2 p_3, N)=1} }S\left(\mathcal{A}_{p_1 p_2 p_3};\mathcal{P}(p_1),p_2\right) +O\left(N^{1-\alpha}\right).
\end{align*}
\end{lemma}
\begin{proof} It is similar to that of [\cite{CAI867}, Lemma 5]. By the trivial inequality
$$
R_{a,b}(N)\geqslant S\left(\mathcal{A};\mathcal{P}, \left(\frac{N}{b}\right)^{\beta}\right)-\sum_{\substack{(\frac{N}{b})^{\beta} \leqslant p_{1}<p_{2}<(\frac{N}{bp_1})^{\frac{1}{2}} \\ (p_{1} p_{2}, N)=1}} S\left(\mathcal{A}_{p_{1} p_{2}} ; \mathcal{P}(p_{1}), p_{2}\right)
$$
and Buchstab's identity we have 
\begin{align}
\nonumber R_{a,b}(N)  \geqslant& S\left(\mathcal{A};\mathcal{P}, \left(\frac{N}{b}\right)^{\beta}\right)-\sum_{\substack{(\frac{N}{b})^{\beta} \leqslant p_{1}<p_{2}<(\frac{N}{bp_1})^{\frac{1}{2}} \\(p_{1} p_{2}, N)=1}} S\left(\mathcal{A}_{p_{1} p_{2}} ; \mathcal{P}(p_{1}), p_{2}\right) \\ 
\nonumber =&S\left(\mathcal{A};\mathcal{P}, \left(\frac{N}{b}\right)^{\alpha}\right)-\sum_{\substack{(\frac{N}{b})^{\alpha} \leqslant p<(\frac{N}{b})^{\beta} \\ (p, N)=1} }S\left(\mathcal{A}_{p};\mathcal{P},\left(\frac{N}{b}\right)^{\alpha}\right) \\ 
&+\sum_{\substack{(\frac{N}{b})^{\alpha} \leqslant p_1<p_2<(\frac{N}{b})^{\beta} \\ (p_1 p_2, N)=1}} S\left(\mathcal{A}_{p_1 p_2};\mathcal{P},p_1\right)
-\sum_{\substack{(\frac{N}{b})^{\beta} \leqslant p_{1}<p_{2}<(\frac{N}{bp_1})^{\frac{1}{2}} \\ (p_{1} p_{2}, N)=1}} S\left(\mathcal{A}_{p_{1} p_{2}} ; \mathcal{P}(p_{1}), p_{2}\right).
\end{align}
On the other hand, we have the trivial inequality
\begin{align}
\nonumber R_{a,b}(N)\geqslant& S\left(\mathcal{A};\mathcal{P}, \left(\frac{N}{b}\right)^{\alpha}\right)-\sum_{\substack{(\frac{N}{b})^{\alpha} \leqslant p_{1}<p_{2}<(\frac{N}{bp_1})^{\frac{1}{2}} \\ (p_{1} p_{2}, N)=1}} S\left(\mathcal{A}_{p_{1} p_{2}} ; \mathcal{P}(p_{1}), p_{2}\right)\\
\nonumber =&S\left(\mathcal{A};\mathcal{P}, \left(\frac{N}{b}\right)^{\alpha}\right)-\sum_{\substack{(\frac{N}{b})^{\alpha} \leqslant p_{1}<p_{2}<(\frac{N}{b})^{\beta} \\ (p_{1} p_{2}, N)=1}} S\left(\mathcal{A}_{p_{1} p_{2}} ; \mathcal{P}(p_{1}), p_{2}\right)\\
\nonumber &-\sum_{\substack{(\frac{N}{b})^{\alpha} \leqslant p_{1}<(\frac{N}{b})^{\beta} \leqslant p_2 <(\frac{N}{bp_1})^{\frac{1}{2}} \\ (p_{1} p_{2}, N)=1}} S\left(\mathcal{A}_{p_{1} p_{2}} ; \mathcal{P}(p_{1}), p_{2}\right)\\
&-\sum_{\substack{(\frac{N}{b})^{\beta} \leqslant p_{1}<p_{2}<(\frac{N}{bp_1})^{\frac{1}{2}} \\ (p_{1} p_{2}, N)=1}} S\left(\mathcal{A}_{p_{1} p_{2}} ; \mathcal{P}(p_{1}), p_{2}\right).
\end{align}
Now by Buchstab's identity we have
$$
\sum_{\substack{(\frac{N}{b})^{\alpha} \leqslant p_1<p_2<(\frac{N}{b})^{\beta} \\ (p_1 p_2, N)=1}} S\left(\mathcal{A}_{p_1 p_2};\mathcal{P},p_1\right)-\sum_{\substack{(\frac{N}{b})^{\alpha} \leqslant p_{1}<p_{2}<(\frac{N}{b})^{\beta} \\ (p_{1} p_{2}, N)=1}} S\left(\mathcal{A}_{p_{1} p_{2}} ; \mathcal{P}(p_{1}), p_{2}\right)
$$
\begin{equation}
=\sum_{\substack{(\frac{N}{b})^{\alpha} \leqslant p_1 < p_2 < p_3<(\frac{N}{b})^{\beta} \\ (p_1 p_2 p_3, N)=1} }S\left(\mathcal{A}_{p_1 p_2 p_3};\mathcal{P}(p_1),p_2\right) +O\left(N^{1-\alpha}\right),
\end{equation}
where the trivial bound
\begin{equation}
\sum_{\substack{(\frac{N}{b})^{\alpha} \leqslant p_1<p_2<(\frac{N}{b})^{\beta} \\ (p_1 p_2, N)=1} }S\left(\mathcal{A}_{p^2_1 p_2};\mathcal{P},p_1\right) \ll N^{1-\alpha}
\end{equation}
is used. Now we add (17) and (18) and by (19), Lemma~\ref{l31} follows. 
\end{proof}
\begin{lemma}\label{l31short} Let $\mathcal{A}=\mathcal{A}_2$ in section 2 and $0<\alpha<\beta \leqslant \frac{1}{3}$. Then we have
\begin{align*}
2R_{a, b}^{\theta}(N) \geqslant& 2S\left(\mathcal{A};\mathcal{P}, \left(\frac{N}{b}\right)^{\alpha}\right)-\sum_{\substack{(\frac{N}{b})^{\alpha} \leqslant p<(\frac{N}{b})^{\beta} \\ (p, N)=1}} S\left(\mathcal{A}_{p};\mathcal{P}, \left(\frac{N}{b}\right)^{\alpha}\right)\\
&-\sum_{\substack{(\frac{N}{b})^{\alpha} \leqslant p_1<(\frac{N}{b})^{\beta} \leqslant p_2 <(\frac{N}{bp_1})^{\frac{1}{2}} \\ (p_1 p_2, N)=1}} S\left(\mathcal{A}_{p_1 p_2};\mathcal{P}(p_1),p_2\right)
- 2\sum_{\substack{(\frac{N}{b})^{\beta} \leqslant p_1 < p_2 <(\frac{N}{bp_1})^{\frac{1}{2}} \\ (p_1 p_2, N)=1} }S\left(\mathcal{A}_{p_1 p_2};\mathcal{P}(p_1),p_2\right)\\
&+\sum_{\substack{(\frac{N}{b})^{\alpha} \leqslant p_1 < p_2 < p_3<(\frac{N}{b})^{\beta} \\ (p_1 p_2 p_3, N)=1} }S\left(\mathcal{A}_{p_1 p_2 p_3};\mathcal{P}(p_1),p_2\right) +O\left(N^{1-\alpha}\right).
\end{align*}
\end{lemma}
\begin{proof} It is similar to that of Lemma~\ref{l31} so we omit it here.
\end{proof}

\begin{lemma}\label{l32} Let $\mathcal{A}=\mathcal{A}_1$ in section 2, then we have
\begin{align*}
4R_{a,b}(N) \geqslant &3S\left(\mathcal{A};\mathcal{P}, \left(\frac{N}{b}\right)^{\frac{1}{13.2}}\right)+S\left(\mathcal{A};\mathcal{P}, \left(\frac{N}{b}\right)^{\frac{1}{8.4}}\right)\\
&+\sum_{\substack{(\frac{N}{b})^{\frac{1}{13.2}} \leqslant p_1<p_2<(\frac{N}{b})^{\frac{1}{8.4}} \\ (p_1 p_2, N)=1} }S\left(\mathcal{A}_{p_1 p_2};\mathcal{P},\left(\frac{N}{b}\right)^{\frac{1}{13.2}}\right)\\
&+\sum_{\substack{(\frac{N}{b})^{\frac{1}{13.2}} \leqslant p_1<(\frac{N}{b})^{\frac{1}{8.4}} \leqslant p_2<(\frac{N}{b})^{\frac{4.6}{13.2}}p^{-1}_1 \\ (p_1 p_2, N)=1} }S\left(\mathcal{A}_{p_1 p_2};\mathcal{P},\left(\frac{N}{b}\right)^{\frac{1}{13.2}}\right)\\
&-\sum_{\substack{(\frac{N}{b})^{\frac{1}{13.2}} \leqslant p<(\frac{N}{b})^{\frac{4.1001}{13.2}} \\ (p, N)=1} }S\left(\mathcal{A}_{p};\mathcal{P},\left(\frac{N}{b}\right)^{\frac{1}{13.2}}\right)\\
&-\sum_{\substack{(\frac{N}{b})^{\frac{1}{13.2}} \leqslant p<(\frac{N}{b})^{\frac{3.6}{13.2}} \\ (p, N)=1}} S\left(\mathcal{A}_{p};\mathcal{P},\left(\frac{N}{b}\right)^{\frac{1}{13.2}}\right)\\
&-\sum_{\substack{(\frac{N}{b})^{\frac{1}{13.2}} \leqslant p_1<(\frac{N}{b})^{\frac{1}{3}} \leqslant p_2 <(\frac{N}{bp_1})^{\frac{1}{2}} \\ (p_1 p_2, N)=1} }S\left(\mathcal{A}_{p_1 p_2};\mathcal{P}(p_1),p_2\right)\\
&-\sum_{\substack{(\frac{N}{b})^{\frac{1}{8.4}} \leqslant p_1<(\frac{N}{b})^{\frac{1}{3.604}} \leqslant p_2 <(\frac{N}{bp_1})^{\frac{1}{2}} \\ (p_1 p_2, N)=1} }S\left(\mathcal{A}_{p_1 p_2};\mathcal{P}(p_1),\left(\frac{N}{bp_1 p_2}\right)^{\frac{1}{2}}\right)\\
&-\sum_{\substack{(\frac{N}{b})^{\frac{4.1001}{13.2}} \leqslant p<(\frac{N}{b})^{\frac{1}{3}} \\ (p, N)=1}} S\left(\mathcal{A}_{p};\mathcal{P},\left(\frac{N}{b}\right)^{\frac{1}{13.2}}\right)\\
&-\sum_{\substack{(\frac{N}{b})^{\frac{3.6}{13.2}} \leqslant p<(\frac{N}{b})^{\frac{1}{3.604}} \\ (p, N)=1} }S\left(\mathcal{A}_{p};\mathcal{P},\left(\frac{N}{b}\right)^{\frac{1}{8.4}}\right)\\
&-\sum_{\substack{(\frac{N}{b})^{\frac{1}{13.2}} \leqslant p_1 < p_2 < p_3< p_4<(\frac{N}{b})^{\frac{1}{8.4}} \\ (p_1 p_2 p_3 p_4, N)=1} }S\left(\mathcal{A}_{p_1 p_2 p_3 p_4};\mathcal{P}(p_1),p_2\right) \\
&-\sum_{\substack{(\frac{N}{b})^{\frac{1}{13.2}} \leqslant p_1 < p_2 < p_3<(\frac{N}{b})^{\frac{1}{8.4}} \leqslant p_4< (\frac{N}{b})^{\frac{4.6}{13.2}}p^{-1}_3 \\ (p_1 p_2 p_3 p_4, N)=1} }S\left(\mathcal{A}_{p_1 p_2 p_3 p_4};\mathcal{P}(p_1),p_2\right) \\
&-2\sum_{\substack{(\frac{N}{b})^{\frac{1}{3.604}} \leqslant p_1 < p_2 <(\frac{N}{bp_1})^{\frac{1}{2}} \\ (p_1 p_2, N)=1} }S\left(\mathcal{A}_{p_1 p_2};\mathcal{P}(p_1),p_2\right)
+O\left(N^{\frac{12.2}{13.2}}\right)\\
=&\left(3 S_{11}+S_{12}\right)+\left(S_{21}+S_{22}\right)-\left(S_{31}+S_{32}\right)-\left(S_{41}+S_{42}\right)\\
&-\left(S_{51}+S_{52}\right)-\left(S_{61}+S_{62}\right)-2S_{7}+O\left(N^{\frac{12.2}{13.2}}\right)\\
=&S_1+S_2-S_3-S_4-S_5-S_6-2S_7+O\left(N^{\frac{12.2}{13.2}}\right).
\end{align*}
\end{lemma}
\begin{proof} It is similar to that of [\cite{CAI867}, Lemma 6]. By Buchstab's identity, we have
\begin{align}
\nonumber S\left(\mathcal{A};\mathcal{P}, \left(\frac{N}{b}\right)^{\frac{1}{8.4}}\right)=&S\left(\mathcal{A};\mathcal{P}, \left(\frac{N}{b}\right)^{\frac{1}{13.2}}\right)\\
\nonumber &-\sum_{\substack{(\frac{N}{b})^{\frac{1}{13.2}} \leqslant p<(\frac{N}{b})^{\frac{1}{8.4}} \\ (p, N)=1} }S\left(\mathcal{A}_{p};\mathcal{P},\left(\frac{N}{b}\right)^{\frac{1}{13.2}}\right)\\
\nonumber &+\sum_{\substack{(\frac{N}{b})^{\frac{1}{13.2}} \leqslant p_1<p_2<(\frac{N}{b})^{\frac{1}{8.4}} \\ (p_1 p_2, N)=1}} S\left(\mathcal{A}_{p_1 p_2};\mathcal{P},\left(\frac{N}{b}\right)^{\frac{1}{13.2}}\right)\\
&-\sum_{\substack{(\frac{N}{b})^{\frac{1}{13.2}} \leqslant p_1<p_2<p_3<(\frac{N}{b})^{\frac{1}{8.4}} \\ (p_1 p_2 p_3, N)=1} }S\left(\mathcal{A}_{p_1 p_2 p_3};\mathcal{P},p_1\right),
\end{align}
\begin{align}
\nonumber &\sum_{\substack{(\frac{N}{b})^{\frac{1}{8.4}} \leqslant p<(\frac{N}{b})^{\frac{3.6}{13.2}} \\ (p, N)=1} }S\left(\mathcal{A}_{p};\mathcal{P},\left(\frac{N}{b}\right)^{\frac{1}{8.4}}\right)\\ \nonumber \leqslant &\sum_{\substack{(\frac{N}{b})^{\frac{1}{8.4}} \leqslant p<(\frac{N}{b})^{\frac{3.6}{13.2}} \\ (p, N)=1} }S\left(\mathcal{A}_{p};\mathcal{P},\left(\frac{N}{b}\right)^{\frac{1}{13.2}}\right)\\
\nonumber &-\sum_{\substack{(\frac{N}{b})^{\frac{1}{13.2}} \leqslant p_1<(\frac{N}{b})^{\frac{1}{8.4}} \leqslant p_2<(\frac{N}{b})^{\frac{4.6}{13.2}}p^{-1}_1 \\ (p_1 p_2, N)=1} }S\left(\mathcal{A}_{p_1 p_2};\mathcal{P},\left(\frac{N}{b}\right)^{\frac{1}{13.2}}\right)\\
&+\sum_{\substack{(\frac{N}{b})^{\frac{1}{13.2}} \leqslant p_1<p_2<(\frac{N}{b})^{\frac{1}{8.4}} \leqslant p_3<(\frac{N}{b})^{\frac{4.6}{13.2}}p^{-1}_2 \\ (p_1 p_2 p_3, N)=1} }S\left(\mathcal{A}_{p_1 p_2 p_3};\mathcal{P},p_1\right),
\end{align}
\begin{align}
\nonumber &\sum_{\substack{(\frac{N}{b})^{\frac{1}{8.4}} \leqslant p_1<(\frac{N}{b})^{\frac{1}{3.604}} \leqslant p_2 <(\frac{N}{bp_1})^{\frac{1}{2}} \\ (p_1 p_2, N)=1} }S\left(\mathcal{A}_{p_1 p_2};\mathcal{P}(p_1),p_2\right)\\
\nonumber =&\sum_{\substack{(\frac{N}{b})^{\frac{1}{8.4}} \leqslant p_1<(\frac{N}{b})^{\frac{1}{3.604}} \leqslant p_2 <(\frac{N}{bp_1})^{\frac{1}{3}} \\ (p_1 p_2, N)=1} }S\left(\mathcal{A}_{p_1 p_2};\mathcal{P}(p_1),p_2\right)\\
&+\sum_{\substack{(\frac{N}{b})^{\frac{1}{8.4}} \leqslant p_1<(\frac{N}{b})^{\frac{1}{3.604}},\ (\frac{N}{bp_1})^{\frac{1}{3}}\leqslant p_2 <(\frac{N}{bp_1})^{\frac{1}{2}} \\ (p_1 p_2, N)=1} }S\left(\mathcal{A}_{p_1 p_2};\mathcal{P}(p_1),p_2\right).
\end{align}
If $p_{2} \leqslant \left(\frac{N}{bp_{1}}\right)^{\frac{1}{3}}$, then $p_{2} \leqslant \left(\frac{N}{bp_{1} p_{2}}\right)^{\frac{1}{2}}$ and by Buchstab's identity we have
\begin{align}
\nonumber &\sum_{\substack{(\frac{N}{b})^{\frac{1}{8.4}} \leqslant p_1<(\frac{N}{b})^{\frac{1}{3.604}} \leqslant p_2 <(\frac{N}{bp_1})^{\frac{1}{3}} \\ (p_1 p_2, N)=1}} S\left(\mathcal{A}_{p_1 p_2};\mathcal{P}(p_1),p_2\right)\\
\nonumber=&\sum_{\substack{(\frac{N}{b})^{\frac{1}{8.4}} \leqslant p_1<(\frac{N}{b})^{\frac{1}{3.604}} \leqslant p_2 <(\frac{N}{bp_1})^{\frac{1}{3}} \\ (p_1 p_2, N)=1} }S\left(\mathcal{A}_{p_1 p_2};\mathcal{P}(p_1),\left(\frac{N}{bp_1 p_2}\right)^{\frac{1}{2}}\right)\\
&+\sum_{\substack{(\frac{N}{b})^{\frac{1}{8.4}} \leqslant p_1<(\frac{N}{b})^{\frac{1}{3.604}} \leqslant p_2\leqslant p_3 <(\frac{N}{bp_1 p_2})^{\frac{1}{2}} \\ (p_1 p_2 p_3, N)=1} }S\left(\mathcal{A}_{p_1 p_2 p_3};\mathcal{P}(p_1 p_2),p_3\right).
\end{align}
On the other hand, if $p_{2} \geqslant \left(\frac{N}{bp_{1}}\right)^{\frac{1}{3}}$, then $p_{2} \geqslant \left(\frac{N}{bp_{1} p_{2}}\right)^{\frac{1}{2}}$ and we have
\begin{align}
\nonumber&\sum_{\substack{(\frac{N}{b})^{\frac{1}{8.4}} \leqslant p_1<(\frac{N}{b})^{\frac{1}{3.604}},\ (\frac{N}{bp_1})^{\frac{1}{3}}\leqslant p_2 <(\frac{N}{bp_1})^{\frac{1}{2}} \\ (p_1 p_2, N)=1} }S\left(\mathcal{A}_{p_1 p_2};\mathcal{P}(p_1),p_2\right)\\
\leqslant &\sum_{\substack{(\frac{N}{b})^{\frac{1}{8.4}} \leqslant p_1<(\frac{N}{b})^{\frac{1}{3.604}},\ (\frac{N}{bp_1})^{\frac{1}{3}}\leqslant p_2 <(\frac{N}{bp_1})^{\frac{1}{2}} \\ (p_1 p_2, N)=1} }S\left(\mathcal{A}_{p_1 p_2};\mathcal{P}(p_1),\left(\frac{N}{bp_1 p_2}\right)^{\frac{1}{2}}\right).
\end{align}
By (23)--(25) we get
\begin{align}
\nonumber&\sum_{\substack{(\frac{N}{b})^{\frac{1}{8.4}} \leqslant p_1<(\frac{N}{b})^{\frac{1}{3.604}} \leqslant p_2 <(\frac{N}{bp_1})^{\frac{1}{2}} \\ (p_1 p_2, N)=1} }S\left(\mathcal{A}_{p_1 p_2};\mathcal{P}(p_1),p_2\right)\\
\nonumber\leqslant &\sum_{\substack{(\frac{N}{b})^{\frac{1}{8.4}} \leqslant p_1<(\frac{N}{b})^{\frac{1}{3.604}} \leqslant p_2 <(\frac{N}{bp_1})^{\frac{1}{2}} \\ (p_1 p_2, N)=1} }S\left(\mathcal{A}_{p_1 p_2};\mathcal{P}(p_1),\left(\frac{N}{bp_1 p_2}\right)^{\frac{1}{2}}\right)\\
&+\sum_{\substack{(\frac{N}{b})^{\frac{1}{8.4}} \leqslant p_1<(\frac{N}{b})^{\frac{1}{3.604}} \leqslant p_2\leqslant p_3 <(\frac{N}{bp_1 p_2})^{\frac{1}{2}} \\ (p_1 p_2 p_3, N)=1}} S\left(\mathcal{A}_{p_1 p_2 p_3};\mathcal{P}(p_1 p_2),p_3\right).
\end{align}
By Buchstab's identity we have
\begin{align}
\nonumber&\sum_{\substack{(\frac{N}{b})^{\frac{1}{13.2}} \leqslant p_1<p_2<p_3<(\frac{N}{b})^{\frac{1}{3}} \\ (p_1 p_2 p_3, N)=1} }S\left(\mathcal{A}_{p_1 p_2 p_3};\mathcal{P}(p_1),p_2\right)\\
\nonumber&-\sum_{\substack{(\frac{N}{b})^{\frac{1}{13.2}} \leqslant p_1<p_2<p_3<(\frac{N}{b})^{\frac{1}{8.4}} \\ (p_1 p_2 p_3, N)=1} }S\left(\mathcal{A}_{p_1 p_2 p_3};\mathcal{P},p_1\right)\\
\nonumber&-\sum_{\substack{(\frac{N}{b})^{\frac{1}{13.2}} \leqslant p_1<p_2<(\frac{N}{b})^{\frac{1}{8.4}} \leqslant p_3<(\frac{N}{b})^{\frac{4.6}{13.2}}p^{-1}_2 \\ (p_1 p_2 p_3, N)=1} }S\left(\mathcal{A}_{p_1 p_2 p_3};\mathcal{P},p_1\right)\\
\nonumber&-\sum_{\substack{(\frac{N}{b})^{\frac{1}{8.4}} \leqslant p_1<(\frac{N}{b})^{\frac{1}{3.604}} \leqslant p_2\leqslant p_3 <(\frac{N}{bp_1 p_2})^{\frac{1}{2}} \\ (p_1 p_2 p_3, N)=1} }S\left(\mathcal{A}_{p_1 p_2 p_3};\mathcal{P}(p_1 p_2),p_3\right)\\
\nonumber\geqslant &-\sum_{\substack{(\frac{N}{b})^{\frac{1}{13.2}} \leqslant p_1 < p_2 < p_3< p_4<(\frac{N}{b})^{\frac{1}{8.4}} \\ (p_1 p_2 p_3 p_4, N)=1}} S\left(\mathcal{A}_{p_1 p_2 p_3 p_4};\mathcal{P}(p_1),p_2\right) \\
&-\sum_{\substack{(\frac{N}{b})^{\frac{1}{13.2}} \leqslant p_1 < p_2 < p_3<(\frac{N}{b})^{\frac{1}{8.4}} \leqslant p_4< (\frac{N}{b})^{\frac{4.6}{13.2}}p^{-1}_3 \\ (p_1 p_2 p_3 p_4, N)=1} }S\left(\mathcal{A}_{p_1 p_2 p_3 p_4};\mathcal{P}(p_1),p_2\right) 
+O\left(N^{\frac{12.2}{13.2}}\right),
\end{align}
where an argument similar to (20) is used. By Lemma~\ref{l31} with $(\alpha, \beta)=\left(\frac{1}{13.2}, \frac{1}{3}\right)$ and $(\alpha, \beta)=$ $\left(\frac{1}{8.4}, \frac{1}{3.604}\right)$ and (21)--(22), (26)--(27)  we complete the proof of Lemma~\ref{l32}.
\end{proof}
\begin{lemma}\label{l33} Let $\mathcal{A}=\mathcal{A}_2$ in section 2, then we have
\begin{align*}
4R_{a,b}^{\theta}(N) \geqslant &3S\left(\mathcal{A};\mathcal{P}, \left(\frac{N}{b}\right)^{\frac{1}{14}}\right)+S\left(\mathcal{A};\mathcal{P}, \left(\frac{N}{b}\right)^{\frac{1}{8.8}}\right)\\
&+\sum_{\substack{(\frac{N}{b})^{\frac{1}{14}} \leqslant p_1<p_2<(\frac{N}{b})^{\frac{1}{8.8}} \\ (p_1 p_2, N)=1} }S\left(\mathcal{A}_{p_1 p_2};\mathcal{P},\left(\frac{N}{b}\right)^{\frac{1}{14}}\right)\\
&+\sum_{\substack{(\frac{N}{b})^{\frac{1}{14}} \leqslant p_1<(\frac{N}{b})^{\frac{1}{8.8}} \leqslant p_2<(\frac{N}{b})^{\frac{4.5863}{14}}p^{-1}_1 \\ (p_1 p_2, N)=1} }S\left(\mathcal{A}_{p_1 p_2};\mathcal{P},\left(\frac{N}{b}\right)^{\frac{1}{14}}\right)\\
&-\sum_{\substack{(\frac{N}{b})^{\frac{1}{14}} \leqslant p<(\frac{N}{b})^{\frac{4.08631}{14}} \\ (p, N)=1} }S\left(\mathcal{A}_{p};\mathcal{P},\left(\frac{N}{b}\right)^{\frac{1}{14}}\right)\\
&-\sum_{\substack{(\frac{N}{b})^{\frac{1}{14}} \leqslant p<(\frac{N}{b})^{\frac{3.5863}{14}} \\ (p, N)=1}} S\left(\mathcal{A}_{p};\mathcal{P},\left(\frac{N}{b}\right)^{\frac{1}{14}}\right)\\
&-\sum_{\substack{(\frac{N}{b})^{\frac{1}{14}} \leqslant p_1<(\frac{N}{b})^{\frac{1}{3.1}} \leqslant p_2 <(\frac{N}{bp_1})^{\frac{1}{2}} \\ (p_1 p_2, N)=1} }S\left(\mathcal{A}_{p_1 p_2};\mathcal{P}(p_1),p_2\right)\\
&-\sum_{\substack{(\frac{N}{b})^{\frac{1}{8.8}} \leqslant p_1<(\frac{N}{b})^{\frac{1}{3.7}} \leqslant p_2 <(\frac{N}{bp_1})^{\frac{1}{2}} \\ (p_1 p_2, N)=1} }S\left(\mathcal{A}_{p_1 p_2};\mathcal{P}(p_1),\left(\frac{N}{bp_1 p_2}\right)^{\frac{1}{2}}\right)\\
&-\sum_{\substack{(\frac{N}{b})^{\frac{4.08631}{14}} \leqslant p<(\frac{N}{b})^{\frac{1}{3.1}} \\ (p, N)=1}} S\left(\mathcal{A}_{p};\mathcal{P},\left(\frac{N}{b}\right)^{\frac{1}{14}}\right)\\
&-\sum_{\substack{(\frac{N}{b})^{\frac{3.5863}{14}} \leqslant p<(\frac{N}{b})^{\frac{1}{3.7}} \\ (p, N)=1} }S\left(\mathcal{A}_{p};\mathcal{P},\left(\frac{N}{b}\right)^{\frac{1}{8.8}}\right)\\
&-\sum_{\substack{(\frac{N}{b})^{\frac{1}{14}} \leqslant p_1 < p_2 < p_3< p_4<(\frac{N}{b})^{\frac{1}{8.8}} \\ (p_1 p_2 p_3 p_4, N)=1} }S\left(\mathcal{A}_{p_1 p_2 p_3 p_4};\mathcal{P}(p_1),p_2\right) \\
&-\sum_{\substack{(\frac{N}{b})^{\frac{1}{14}} \leqslant p_1 < p_2 < p_3<(\frac{N}{b})^{\frac{1}{8.8}} \leqslant p_4< (\frac{N}{b})^{\frac{4.5863}{14}}p_3^{-1} \\ (p_1 p_2 p_3 p_4, N)=1} }S\left(\mathcal{A}_{p_1 p_2 p_3 p_4};\mathcal{P}(p_1),p_2\right) \\
&-2\sum_{\substack{(\frac{N}{b})^{\frac{1}{3.1}} \leqslant p_1 < p_2 <(\frac{N}{bp_1})^{\frac{1}{2}} \\ (p_1 p_2, N)=1} }S\left(\mathcal{A}_{p_1 p_2};\mathcal{P}(p_1),p_2\right)\\
&-2\sum_{\substack{(\frac{N}{b})^{\frac{1}{3.7}} \leqslant p_1 < p_2 <(\frac{N}{bp_1})^{\frac{1}{2}} \\ (p_1 p_2, N)=1} }S\left(\mathcal{A}_{p_1 p_2};\mathcal{P}(p_1),p_2\right)+O\left(N^{\frac{13}{14}}\right)\\
=&\left(3 S_{11}^{\prime}+S_{12}^{\prime}\right)+\left(S_{21}^{\prime}+S_{22}^{\prime}\right)-\left(S_{31}^{\prime}+S_{32}^{\prime}\right)-\left(S_{41}^{\prime}+S_{42}^{\prime}\right)\\
&-\left(S_{51}^{\prime}+S_{52}^{\prime}\right)-\left(S_{61}^{\prime}+S_{62}^{\prime}\right)-2\left(S_{71}^{\prime}+S_{72}^{\prime}\right)+O\left(N^{\frac{13}{14}}\right)\\
=&S_1^{\prime}+S_2^{\prime}-S_3^{\prime}-S_4^{\prime}-S_5^{\prime}-S_6^{\prime}-2S_7^{\prime}+O\left(N^{\frac{13}{14}}\right).
\end{align*}
\end{lemma}
\begin{proof} It is similar to that of Lemma~\ref{l32} and [\cite{Cai2015}, Lemma 9] so we omit it here.
\end{proof}
\begin{lemma}\label{l34} See \cite{CAI867}. Let $\mathcal{A}=\mathcal{A}_1$ in section 2, $D_{1}=\left(\frac{N}{b}\right)^{1/2}\left(\log\left(\frac{N}{b}\right)\right)^{-B}$ with $B=B(A)>0$ in Lemma~\ref{l3}, and $\underline{p}=\frac{D_{1}}{p}$. Then we have
\begin{align*}
& \sum_{\substack{(\frac{N}{b})^{\frac{4.1001}{13.2}} \leqslant p<(\frac{N}{b})^{\frac{1}{3}} \\
(p, N)=1}} S\left(\mathcal{A}_{p};\mathcal{P}, \underline{p}^{\frac{1}{2.5}}\right) \\
\leqslant& \sum_{\substack{(\frac{N}{b})^{\frac{4.1001}{13.2}} \leqslant p<(\frac{N}{b})^{\frac{1}{3}} \\
(p, N)=1}} S\left(\mathcal{A}_{p};\mathcal{P}, \underline{p}^{\frac{1}{3.675}}\right)
\\
&-\frac{1}{2}\sum_{\substack{(\frac{N}{b})^{\frac{4.1001}{13.2}} \leqslant p<(\frac{N}{b})^{\frac{1}{3}} \\
(p, N)=1}}\sum_{\substack{\underline{p}^{\frac{1}{3.675}} \leqslant p_1<\underline{p}^{\frac{1}{2.5}} \\ (p_1, N)=1}}S\left(\mathcal{A}_{p p_1};\mathcal{P}, \underline{p}^{\frac{1}{3.675}}\right) 
\\
&+\frac{1}{2}
\sum
_{
\substack{(\frac{N}{b})^{\frac{4.1001}{13.2}} \leqslant p<(\frac{N}{b})^{\frac{1}{3}} \\
(p, N)=1}
}
\sum
_{
\substack{\underline{p}^{\frac{1}{3.675}} \leqslant p_1<p_2<p_3<\underline{p}^{\frac{1}{2.5}} \\ (p_1 p_2 p_3, N)=1}
}
S\left(\mathcal{A}_{p p_1 p_2 p_3};\mathcal{P}(p_1), p_2\right)
+O\left(N^{\frac{19}{20}}\right).
\end{align*}
\end{lemma}
\begin{proof} It is similar to that of [\cite{CAI867}, Lemma 7]. By Buchstab's identity, we have
\begin{align}
\nonumber S\left(\mathcal{A}_{p};\mathcal{P}, \underline{p}^{\frac{1}{2.5}}\right)= & S\left(\mathcal{A}_{p};\mathcal{P}, \underline{p}^{\frac{1}{3.675}}\right)-\sum_{\substack{\underline{p}^{\frac{1}{3.675}} \leqslant p_1<\underline{p}^{\frac{1}{2.5}} \\ (p_1, N)=1}}S\left(\mathcal{A}_{p p_1};\mathcal{P}, \underline{p}^{\frac{1}{3.675}}\right)\\
&+\sum_{\substack{\underline{p}^{\frac{1}{3.675}} \leqslant p_1<p_2<\underline{p}^{\frac{1}{2.5}} \\ (p_1 p_2, N)=1}}S\left(\mathcal{A}_{p p_1 p_2};\mathcal{P}, p_1\right),\\
\nonumber S\left(\mathcal{A}_{p};\mathcal{P}, \underline{p}^{\frac{1}{2.5}}\right)=&S\left(\mathcal{A}_{p};\mathcal{P}, \underline{p}^{\frac{1}{3.675}}\right)-\sum_{\substack{\underline{p}^{\frac{1}{3.675}} \leqslant p_1<\underline{p}^{\frac{1}{2.5}} \\ (p_1, N)=1}}S\left(\mathcal{A}_{p p_1};\mathcal{P}, \underline{p}^{\frac{1}{2.5}}\right)\\
&-\sum_{\substack{\underline{p}^{\frac{1}{3.675}} \leqslant p_1<p_2<\underline{p}^{\frac{1}{2.5}} \\ (p_1 p_2, N)=1}}S\left(\mathcal{A}_{p p_1 p_2};\mathcal{P}(p_1), p_2\right),
\end{align}
\begin{align}
\nonumber &\sum_{\substack{\underline{p}^{\frac{1}{3.675}} \leqslant p_1<p_2<\underline{p}^{\frac{1}{2.5}} \\ (p_1 p_2, N)=1}}S\left(\mathcal{A}_{p p_1 p_2};\mathcal{P}, p_1\right)-\sum_{\substack{\underline{p}^{\frac{1}{3.675}} \leqslant p_1<p_2<\underline{p}^{\frac{1}{2.5}} \\ (p_1 p_2, N)=1}}S\left(\mathcal{A}_{p p_1 p_2};\mathcal{P}(p_1), p_2\right)\\
=&\sum_{\substack{\underline{p}^{\frac{1}{3.675}} \leqslant p_1<p_2<p_3<\underline{p}^{\frac{1}{2.5}} \\ (p_1 p_2 p_3, N)=1}}S\left(\mathcal{A}_{p p_1 p_2 p_3};\mathcal{P}(p_1), p_2\right)+\sum_{\substack{\underline{p}^{\frac{1}{3.675}} \leqslant p_1<p_2<\underline{p}^{\frac{1}{2.5}} \\ (p_1 p_2, N)=1}}S\left(\mathcal{A}_{p p^2_1 p_2};\mathcal{P}, p_1\right).
\end{align}
Now we add (28) and (29), sum over $p$ in the interval $\left[\left(\frac{N}{b}\right)^{\frac{4.1001}{13.2}}, \left(\frac{N}{b}\right)^{\frac{1}{3}}\right)$ and by (30), we get Lemma~\ref{l34}, where the trivial inequality
$$
\sum_{\substack{(\frac{N}{b})^{\frac{4.1001}{13.2}} \leqslant p<(\frac{N}{b})^{\frac{1}{3}} \\
(p, N)=1}}\sum_{\substack{\underline{p}^{\frac{1}{3.675}} \leqslant p_1<p_2<\underline{p}^{\frac{1}{2.5}} \\ (p_1 p_2, N)=1}}S\left(\mathcal{A}_{p p^2_1 p_2};\mathcal{P}, p_1\right) \ll N^{\frac{19}{20}}
$$
is used.
\end{proof}
\begin{lemma}\label{l35} See \cite{CL2011}. Let $\mathcal{A}=\mathcal{A}_2$ in section 2, $D_{2}=\left(\frac{N}{b}\right)^{\theta/2}\left(\log\left(\frac{N}{b}\right)\right)^{-B}$ with $B=B(A)>0$ in Lemma~\ref{l3}, and $\underline{p}^{\prime}=\frac{D_{2}}{p}$. Then we have
\begin{align*}
& \sum_{\substack{(\frac{N}{b})^{\frac{4.08631}{14}} \leqslant p<(\frac{N}{b})^{\frac{1}{3.1}} \\
(p, N)=1}} S\left(\mathcal{A}_{p};\mathcal{P}, \underline{p}^{\prime\frac{1}{2.5}}\right) \\
\leqslant& \sum_{\substack{(\frac{N}{b})^{\frac{4.08631}{14}} \leqslant p<(\frac{N}{b})^{\frac{1}{3.1}} \\
(p, N)=1}} S\left(\mathcal{A}_{p};\mathcal{P}, \underline{p}^{\prime\frac{1}{3.675}}\right)
\\
&-\frac{1}{2}\sum_{\substack{(\frac{N}{b})^{\frac{4.08631}{14}} \leqslant p<(\frac{N}{b})^{\frac{1}{3.1}} \\
(p, N)=1}}\sum_{\substack{\underline{p}^{\prime\frac{1}{3.675}} \leqslant p_1<\underline{p}^{\prime\frac{1}{2.5}} \\ (p_1, N)=1}}S\left(\mathcal{A}_{p p_1};\mathcal{P}, \underline{p}^{\prime\frac{1}{3.675}}\right) 
\\
&+\frac{1}{2}
\sum
_{
\substack{(\frac{N}{b})^{\frac{4.08631}{14}} \leqslant p<(\frac{N}{b})^{\frac{1}{3.1}} \\
(p, N)=1}
}
\sum
_{
\substack{\underline{p}^{\prime\frac{1}{3.675}} \leqslant p_1<p_2<p_3<\underline{p}^{\prime\frac{1}{2.5}} \\ (p_1 p_2 p_3, N)=1}
}
S\left(\mathcal{A}_{p p_1 p_2 p_3};\mathcal{P}(p_1), p_2\right)
+O\left(N^{\theta-\frac{1}{20}}\right).
\end{align*}
\end{lemma}
\begin{proof} It is similar to that of Lemma~\ref{l34} and [\cite{Cai2015}, Lemma 10] so we omit it here.
\end{proof}
\begin{lemma}\label{upperboundsieve} See \cite{ERPAN}. Let $\mathcal{A}=\mathcal{A}_1$ in section 2, then we have
\begin{align*}
\sum_{\substack{a p_{1}+b p_{2}=N \\ p_{1} \text { and } p_{2} \text { are primes }}} 1 \leqslant &  S\left(\mathcal{A};\mathcal{P}, \left(\frac{N}{b}\right)^{\frac{1}{5}}\right)+O\left(N^{\frac{1}{5}}\right)\\
\leqslant & S\left(\mathcal{A};\mathcal{P}, \left(\frac{N}{b}\right)^{\frac{1}{7}}\right) \\ 
&-\frac{1}{2} \sum_{\substack{(\frac{N}{b})^{\frac{1}{7}} \leqslant p<(\frac{N}{b})^{\frac{1}{5}} \\ (p, N)=1} }S\left(\mathcal{A}_{p};\mathcal{P},\left(\frac{N}{b}\right)^{\frac{1}{7}}\right) \\
&+\frac{1}{2} \sum_{\substack{(\frac{N}{b})^{\frac{1}{7}} \leqslant p_1<p_2<p_3<(\frac{N}{b})^{\frac{1}{5}} \\ (p_1 p_2 p_3, N)=1} }S\left(\mathcal{A}_{p_1 p_2 p_3};\mathcal{P}(p_1), p_2\right)+O\left(N^{\frac{6}{7}}\right)\\
=& \Upsilon_1-\frac{1}{2}\Upsilon_2+\frac{1}{2}\Upsilon_3+O\left(N^{\frac{6}{7}}\right).
\end{align*}
\end{lemma}
\begin{proof} It is similar to that of Lemma~\ref{l34} and [\cite{ERPAN}, p. 211, Lemma 5] so we omit it here.
\end{proof}

\section{Proof of Theorem 1.1}
In this section, sets $\mathcal{A}_1$, $\mathcal{B}_1$, $\mathcal{C}_1$, $\mathcal{E}_1$ and $\mathcal{F}_1$ are defined respectively. We define the function $\omega$ as $\omega(p)=0$ for primes $p \mid a b N$ and $\omega(p)=\frac{p}{p-1}$ for other primes.
\subsection{Evaluation of $S_{1}, S_{2}, S_{3}$}
Let $D_{\mathcal{A}_{1}}=\left(\frac{N}{b}\right)^{1 / 2}\left(\log \left(\frac{N}{b}\right)\right)^{-B}$ for some positive constant $B$. We can take 
\begin{equation}
X_{\mathcal{A}_1}=\sum_{\substack{0 \leqslant k \leqslant b-1 \\(k, b)=1}}\pi\left(\frac{N}{a} ; b^{2}, N a_{b^{2}}^{-1}+k b\right) \sim \frac{\varphi(b) N}{a \varphi\left(b^{2}\right) \log N} \sim \frac{N}{a b \log N}
\end{equation}
so that $|\mathcal{A}_1| \sim X_{\mathcal{A}_1}$. By Lemma~\ref{Wfunction} for $z_{\mathcal{A}_1}=\left(\frac{N}{b}\right)^{\frac{1}{\alpha}}$ we have
\begin{equation}
W(z_{\mathcal{A}_1})=\frac{2\alpha e^{-\gamma} C(abN)(1+o(1))}{\log N}.
\end{equation}

To deal with the error terms, any $\frac{N-a p}{b}$ in $\mathcal{A}_1$ is relatively prime to $b$, so $\eta\left(X_{\mathcal{A}_1}, n\right)=0$ for any integer $n$ that shares a common prime divisor with $b$. If $n$ and $a$ share a common prime divisor $r$, say $n=r n^{\prime}$ and $a=r a^{\prime}$, then $\frac{N-a p}{b n}=\frac{N-r a^{\prime} p}{b r n^{\prime}} \in \mathbb{Z}$ implies $r \mid N$, which is a contradiction to $(a, N)=1$. Similarly, we have $\eta\left(X_{\mathcal{A}_1}, n\right)=0$ if $(n, N)>1$. We conclude that $\eta\left(X_{\mathcal{A}_1}, n\right)=0$ if $(n, a b N)>1$. 
For a square-free integer $n \leqslant D_{\mathcal{A}_1}$ such that $(n, abN)=1$, to make $n \mid \frac{N-a p}{b}$ for some $\frac{N-a p}{b} \in \mathcal{A}_1$, we need $a p \equiv N(\bmod b n)$, which implies $a p \equiv N+k b n$ $\left(\bmod b^{2} n\right)$ for some $0 \leqslant k \leqslant b-1$. Since $\left(\frac{N-a p}{b n}, b\right)=1$, we can further require $(k, b)=1$. When $k$ runs through the reduced residues modulo $b$, we know $k a_{b^{2} n}^{-1}$ also runs through the reduced residues modulo $b$. Therefore, we have $p \equiv N a_{b^{2} n}^{-1} +k b n\left(\bmod b^{2} n\right)$ for some $0 \leqslant k \leqslant b-1$ such that $(k, b)=1$. Conversely, if $p=N a_{b^{2} n}^{-1} +k b n+m b^{2} n$ for some integer $m$ and some $0 \leqslant k \leqslant b-1$ such that $(k, b)=1$, then $\left(\frac{N-a p}{b n}, b\right)=$ $\left(\frac{N-a a_{b^{2} n}^{-1} N-a k b n-a m b^{2} n}{b n}, b\right)=(-a k, b)=1$. Therefore, for square-free integers $n$ such that $(n, abN)=1$, we have
\begin{align}
\nonumber \left|\eta\left(X_{\mathcal{A}_1}, n\right)\right|  =&\left|\sum_{\substack{a \in \mathcal{A}_1 \\
a \equiv 0(\bmod n)}} 1-\frac{\omega(n)}{n} X_{\mathcal{A}_1}\right| \\
\nonumber =&\left|\sum_{\substack{0 \leqslant k \leqslant b-1 \\
(k, \bar{b})=1}} \pi\left(\frac{N}{a} ; b^{2} n, N a_{b^{2} n}^{-1} +k b n\right)-\frac{X_{\mathcal{A}_1}}{\varphi(n)}\right| \\
\nonumber =&\left|\sum_{\substack{0 \leqslant k \leqslant b-1 \\
(k, b)=1}}\left( \pi\left(\frac{N}{a} ; b^{2} n, N a_{b^{2} n}^{-1} +k b n\right)- \frac{\pi\left(\frac{N}{a} ; b^{2}, N a_{b^{2}}^{-1}+k b\right)}{\varphi(n)}\right)\right|\\
\nonumber \ll&\left|\sum_{\substack{0 \leqslant k \leqslant b-1 \\
(k, b)=1}}\left( \pi\left(\frac{N}{a} ; b^{2} n, N a_{b^{2} n}^{-1} +k b n\right)-\frac{\pi\left(\frac{N}{a} ; 1,1\right)}{\varphi\left(b^2 n\right)}\right)\right|\\
\nonumber& +\left|\sum_{\substack{0 \leqslant k \leqslant b-1 \\
(k, b)=1}}\left( \frac{\pi\left(\frac{N}{a} ; b^{2}, N a_{b^{2}}^{-1} +k b\right)}{\varphi(n)}- \frac{\pi\left(\frac{N}{a} ; 1,1\right)}{\varphi\left(b^{2} n\right)}\right)\right| \\
\nonumber \ll& \sum_{\substack{0 \leqslant k \leqslant b-1 \\
(k, \bar{b})=1}}\left|\pi\left(\frac{N}{a} ; b^{2} n, N a_{b^{2} n}^{-1} +k b n\right)-\frac{\pi\left(\frac{N}{a} ; 1,1\right)}{\varphi\left(b^{2} n\right)}\right| \\
& +\frac{1}{\varphi(n)} \sum_{\substack{0 \leqslant k \leqslant b-1 \\
(k, \bar{b})=1}}\left|\pi\left(\frac{N}{a} ; b^{2}, N a_{b^{2}}^{-1} +k b\right)-\frac{\pi\left(\frac{N}{a} ; 1,1\right)}{\varphi\left(b^{2}\right)}\right| .
\end{align}
By Lemma~\ref{l3} with $g(k)=1$ for $k=1$ and $g(k)=0$ for $k>1$, we have
\begin{equation}
\sum_{\substack{n \leqslant D_{\mathcal{A}_1} \\ n \mid P(z_{\mathcal{A}_1})}}  \left|\eta\left(X_{\mathcal{A}_1}, n\right)\right| \ll N(\log N)^{-5}
\end{equation}
and
\begin{equation}
\sum_{p}\sum_{\substack{n \leqslant \frac{D_{\mathcal{A}_{1}}}{p} \\ n \mid P(z_{\mathcal{A}_1})}} \left|\eta\left(X_{\mathcal{A}_1}, pn\right)\right| \ll N(\log N)^{-5}.
\end{equation}

Then by (31)--(35), Lemma~\ref{l1}, Lemma~\ref{l2} and some routine arguments we have
\begin{align}
\nonumber S_{11} \geqslant& X_{\mathcal{A}_1} W\left(z_{\mathcal{A}_1}\right)\left\{f\left(\frac{1/2}{1/13.2}\right)+O\left(\frac{1}{\log ^{\frac{1}{3}} D_{\mathcal{A}_1}}\right)\right\}-\sum_{\substack{n<D_{\mathcal{A}_1} \\ n \mid P(z_{\mathcal{A}_1})}}\left|\eta\left(X_{\mathcal{A}_1}, n\right)\right| \\
\nonumber \geqslant& \frac{N}{a b \log N}\frac{2\times 13.2 e^{-\gamma} C(abN)(1+o(1))}{\log N}\left(\frac{2 e^{\gamma}}{\frac{13.2}{2}}\left(\log 5.6+\int_{2}^{4.6} \frac{\log (s-1)}{s} \log \frac{5.6}{s+1} d s\right)\right)\\
\nonumber \geqslant& (1+o(1)) \frac{8C(abN) N}{ab(\log N)^2}\left(\log 5.6+\int_{2}^{4.6} \frac{\log (s-1)}{s} \log \frac{5.6}{s+1} d s\right)\\
\nonumber \geqslant& 14.82216 \frac{C(abN) N}{ab(\log N)^2},\\
\nonumber S_{12} \geqslant& X_{\mathcal{A}_1} W\left(z_{\mathcal{A}_1}\right)\left\{f\left(\frac{1/2}{1/8.4}\right)+O\left(\frac{1}{\log ^{\frac{1}{3}} D_{\mathcal{A}_1}}\right)\right\}-\sum_{\substack{n<D_{\mathcal{A}_1} \\ n \mid P(z_{\mathcal{A}_1})}}\left|\eta\left(X_{\mathcal{A}_1}, n\right)\right| \\
\nonumber \geqslant& \frac{N}{a b \log N}\frac{2\times 8.4 e^{-\gamma} C(abN)(1+o(1))}{\log N}\left(\frac{2 e^{\gamma}}{\frac{8.4}{2}}\left(\log 3.2+\int_{2}^{2.2} \frac{\log (s-1)}{s} \log \frac{3.2}{s+1} d s\right)\right)\\
\nonumber \geqslant& (1+o(1)) \frac{8C(abN) N}{ab(\log N)^2}\left(\log 3.2+\int_{2}^{2.2} \frac{\log (s-1)}{s} \log \frac{3.2}{s+1} d s\right) \\
\nonumber \geqslant& 9.30664 \frac{C(abN) N}{ab(\log N)^2},\\
S_{1}=&3 S_{11}+S_{12} \geqslant 53.77312\frac{C(abN) N}{ab(\log N)^2} .
\end{align}
Similarly, we have
\begin{align}
\nonumber S_{21} \geqslant& \frac{N}{a b \log N}\frac{2\times 13.2 e^{-\gamma} C(abN)(1+o(1))}{\log N} \times \\
\nonumber &\sum_{\substack{(\frac{N}{b})^{\frac{1}{13.2}} \leqslant p_1<p_2<(\frac{N}{b})^{\frac{1}{8.4}} \\ (p_1 p_2, N)=1} }\frac{1}{p_1 p_2} f\left(13.2\left(\frac{1}{2}-\frac{\log p_1 p_2}{\log \frac{N}{b}}\right)\right)\\
\nonumber \geqslant& \frac{N}{a b \log N}\frac{2\times 13.2 e^{-\gamma} C(abN)(1+o(1))}{\log N} \times \\
\nonumber &\sum_{\substack{(\frac{N}{b})^{\frac{1}{13.2}} \leqslant p_1<p_2<(\frac{N}{b})^{\frac{1}{8.4}} \\ (p_1 p_2, N)=1} }\frac{1}{p_1 p_2} \frac{2 e^{\gamma} \log \left(13.2\left(\frac{1}{2}-\frac{\log p_1 p_2}{\log \frac{N}{b}}\right)-1\right)}{13.2\left(\frac{1}{2}-\frac{\log p_1 p_2}{\log \frac{N}{b}}\right)}\\
\nonumber \geqslant& (1+o(1)) \frac{4 C(abN) N}{ab(\log N)^{2} } \left(\int_{\frac{1}{13.2}}^{\frac{1}{8.4}} \int_{t_{1}}^{\frac{1}{8.4}} \frac{\log \left(5.6-13.2\left(t_{1}+t_{2}\right)\right)}{t_{1} t_{2}\left(\frac{1}{2}-\left(t_{1}+t_{2}\right)\right)} d t_{1} d t_{2}\right)\\
\nonumber \geqslant& (1+o(1)) \frac{8 C(abN) N}{ab(\log N)^{2} } \left(\int_{\frac{1}{13.2}}^{\frac{1}{8.4}} \int_{t_{1}}^{\frac{1}{8.4}} \frac{\log \left(5.6-13.2\left(t_{1}+t_{2}\right)\right)}{t_{1} t_{2}\left(1-2\left(t_{1}+t_{2}\right)\right)} d t_{1} d t_{2}\right),\\
\nonumber S_{22} \geqslant& \frac{N}{a b \log N}\frac{2\times 13.2 e^{-\gamma} C(abN)(1+o(1))}{\log N} \times \\
\nonumber &\sum_{\substack{(\frac{N}{b})^{\frac{1}{13.2}} \leqslant p_1<(\frac{N}{b})^{\frac{1}{8.4}} \leqslant p_2<(\frac{N}{b})^{\frac{4.6}{13.2}}p^{-1}_1 \\ (p_1 p_2, N)=1} }\frac{1}{p_1 p_2} f\left(13.2\left(\frac{1}{2}-\frac{\log p_1 p_2}{\log \frac{N}{b}}\right)\right)\\
\nonumber \geqslant& \frac{N}{a b \log N}\frac{2\times 13.2 e^{-\gamma} C(abN)(1+o(1))}{\log N} \times \\
\nonumber &\sum_{\substack{(\frac{N}{b})^{\frac{1}{13.2}} \leqslant p_1<(\frac{N}{b})^{\frac{1}{8.4}} \leqslant p_2<(\frac{N}{b})^{\frac{4.6}{13.2}}p^{-1}_1 \\ (p_1 p_2, N)=1} }\frac{1}{p_1 p_2} \frac{2 e^{\gamma} \log \left(13.2\left(\frac{1}{2}-\frac{\log p_1 p_2}{\log \frac{N}{b}}\right)-1\right)}{13.2\left(\frac{1}{2}-\frac{\log p_1 p_2}{\log \frac{N}{b}}\right)}\\
\nonumber \geqslant& (1+o(1)) \frac{4 C(abN) N}{ab(\log N)^{2}} \left(\int_{\frac{1}{13.2}}^{\frac{1}{8.4}} \int_{\frac{1}{8.4}}^{\frac{4.6}{13.2}-t_{1}} \frac{\log \left(5.6-13.2\left(t_{1}+t_{2}\right)\right)}{t_{1} t_{2}\left(\frac{1}{2}-\left(t_{1}+t_{2}\right)\right)} d t_{1} d t_{2}\right)\\
\nonumber \geqslant& (1+o(1)) \frac{8 C(abN) N}{ab(\log N)^{2}} \left(\int_{\frac{1}{13.2}}^{\frac{1}{8.4}} \int_{\frac{1}{8.4}}^{\frac{4.6}{13.2}-t_{1}} \frac{\log \left(5.6-13.2\left(t_{1}+t_{2}\right)\right)}{t_{1} t_{2}\left(1-2\left(t_{1}+t_{2}\right)\right)} d t_{1} d t_{2}\right),\\
\nonumber S_{2}=&S_{21}+S_{22}\\
\nonumber \geqslant&(1+o(1)) \frac{8 C(abN) N}{ab(\log N)^{2}} \left(\int_{\frac{1}{13.2}}^{\frac{1}{8.4}} \int_{t_{1}}^{\frac{4.6}{13.2}-t_{1}} \frac{\log \left(5.6-13.2\left(t_{1}+t_{2}\right)\right)}{t_{1} t_{2}\left(1-2\left(t_{1}+t_{2}\right)\right)} d t_{1} d t_{2}\right)\\
\geqslant& 5.201296 \frac{C(abN) N}{ab(\log N)^{2}},
\end{align}
\begin{align}
\nonumber S_{31} \leqslant& \frac{N}{a b \log N}\frac{2\times 13.2 e^{-\gamma} C(abN)(1+o(1))}{\log N} \sum_{\substack{(\frac{N}{b})^{\frac{1}{13.2}} \leqslant p<(\frac{N}{b})^{\frac{4.1001}{13.2}} \\ (p, N)=1} }\frac{1}{p} F\left(13.2\left(\frac{1}{2}-\frac{\log p}{\log \frac{N}{b}}\right)\right)\\
\nonumber \leqslant& \frac{N}{a b \log N}\frac{2\times 13.2 e^{-\gamma} C(abN)(1+o(1))}{\log N} \int_{(\frac{N}{b})^{\frac{1}{13.2}}}^{(\frac{N}{b})^{\frac{4.1001}{13.2}}} \frac{1}{u \log u} F\left(13.2\left(\frac{1}{2}-\frac{\log u}{\log \frac{N}{b}}\right)\right) d u \\
\nonumber \leqslant&(1+o(1)) \frac{8 C(abN) N}{ab(\log N)^{2}}\left(\log \frac{4.1001(13.2-2)}{13.2-8.2002}+\int_{2}^{4.6} \frac{\log (s-1)}{s} \log \frac{5.6(5.6-s)}{s+1} d s\right.\\
\nonumber&\left.+\int_{2}^{2.6} \frac{\log (s-1)}{s} d s \int_{s+2}^{4.6} \frac{1}{t} \log \frac{t-1}{s+1} \log \frac{5.6(5.6-t)}{t+1} d t\right)\leqslant 21.9016 \frac{C(abN) N}{ab(\log N)^{2}},\\
\nonumber S_{32} \leqslant& \frac{N}{a b \log N}\frac{2\times 13.2 e^{-\gamma} C(abN)(1+o(1))}{\log N} \sum_{\substack{(\frac{N}{b})^{\frac{1}{13.2}} \leqslant p<(\frac{N}{b})^{\frac{3.6}{13.2}} \\ (p, N)=1} }\frac{1}{p} F\left(13.2\left(\frac{1}{2}-\frac{\log p}{\log \frac{N}{b}}\right)\right)\\
\nonumber \leqslant& \frac{N}{a b \log N}\frac{2\times 13.2 e^{-\gamma} C(abN)(1+o(1))}{\log N} \int_{(\frac{N}{b})^{\frac{1}{13.2}}}^{(\frac{N}{b})^{\frac{3.6}{13.2}}} \frac{1}{u \log u} F\left(13.2\left(\frac{1}{2}-\frac{\log u}{\log \frac{N}{b}}\right)\right) d u \\
\nonumber \leqslant&(1+o(1)) \frac{8 C(abN) N}{ab(\log N)^{2}}\left(\log \frac{3.6(13.2-2)}{13.2-7.2}+\int_{2}^{4.6} \frac{\log (s-1)}{s} \log \frac{5.6(5.6-s)}{s+1} d s\right.\\
\nonumber&\left.+\int_{2}^{2.6} \frac{\log (s-1)}{s} d s \int_{s+2}^{4.6} \frac{1}{t} \log \frac{t-1}{s+1} \log \frac{5.6(5.6-t)}{t+1} d t\right) \leqslant 19.40136 \frac{C(abN) N}{ab(\log N)^{2}},\\
S_{3}=& S_{31}+S_{32} \leqslant 41.30296 \frac{C(abN) N}{ab(\log N)^{2}}.
\end{align}

\subsection{Evaluation of $S_{4}, S_{7}$}
Let $D_{\mathcal{B}_1}=N^{1 / 2}(\log N)^{-B}$. By Chen's switching principle and similar arguments as in \cite{CL2002}, we know that
\begin{equation}
|\mathcal{E}_1|<\left(\frac{N}{b}\right)^{\frac{2}{3}}, \quad
\left(\frac{N}{b}\right)^{\frac{1}{3}}<e \leqslant \left(\frac{N}{b}\right)^{\frac{2}{3}} \ \operatorname{for}\ e \in \mathcal{E}_1, \quad
S_{41} \leqslant S\left(\mathcal{B}_1;\mathcal{P},D_{\mathcal{B}_1}^{\frac{1}{2}}\right)+O\left(N^{\frac{2}{3}}\right).
\end{equation}
Then we can take
\begin{equation}
X_{\mathcal{B}_1}=\sum_{\substack{(\frac{N}{b})^{\frac{1}{13.2}} \leqslant p_{1} <(\frac{N}{b})^{\frac{1}{3}}\leqslant p_{2} <(\frac{N}{b p_{1}})^{\frac{1}{2}}  \\ 0 \leqslant j \leqslant a-1,(j, a)=1}} \pi\left(\frac{N}{b p_{1} p_{2}} ; a^{2}, N\left(b p_{1} p_{2}\right)_{a^{2}}^{-1}+j a\right)
\end{equation}
so that $|\mathcal{B}_1| \sim X_{\mathcal{B}_1}$. By Lemma~\ref{Wfunction} for $z_{\mathcal{B}_1}=D_{\mathcal{B}_1}^{\frac{1}{2}}=N^{\frac{1}{4}}(\log N)^{-B/2}$ we have
\begin{equation}
W(z_{\mathcal{B}_1})=\frac{8 e^{-\gamma} C(abN)(1+o(1))}{\log N}, \quad F(2)=e^{\gamma}.
\end{equation}
By the prime number theorem and integration by parts we get that
\begin{align}
\nonumber X_{\mathcal{B}_1} &=(1+o(1)) \sum_{(\frac{N}{b})^{\frac{1}{13.2}} \leqslant p_{1} <(\frac{N}{b})^{\frac{1}{3}}\leqslant p_{2} <(\frac{N}{b p_{1}})^{\frac{1}{2}}} \frac{\varphi(a) \frac{N}{b p_{1} p_{2}}}{\varphi\left(a^{2}\right) \log \left(\frac{N}{b p_{1} p_{2}}\right)} \\
\nonumber & =(1+o(1)) \frac{N}{a b} \sum_{(\frac{N}{b})^{\frac{1}{13.2}} \leqslant p_{1} <(\frac{N}{b})^{\frac{1}{3}}\leqslant p_{2} <(\frac{N}{b p_{1}})^{\frac{1}{2}}} \frac{1}{p_{1} p_{2} \log \left(\frac{N}{p_{1} p_{2}}\right)} \\
\nonumber& =(1+o(1)) \frac{N}{a b} \int_{(\frac{N}{b})^{\frac{1}{13.2}}}^{(\frac{N}{b})^{\frac{1}{3}}} \frac{d t}{t \log t} \int_{(\frac{N}{b})^{\frac{1}{3}}}^{(\frac{N}{bt})^{\frac{1}{2}}} \frac{d u}{u \log u \log \left(\frac{N}{u t}\right)}\\
& =(1+o(1)) \frac{N}{ab\log N} \int_{2}^{12.2} \frac{\log \left(2-\frac{3}{s+1}\right)}{s} d s .
\end{align}

To deal with the error terms, for an integer $n$ such that $(n, a b N)>1$, similarly to the discussion for $\eta\left(X_{\mathcal{A}_1}, n\right)$, we have $\eta\left(X_{\mathcal{B}_1}, n\right)=0$.
For a square-free integer $n$ such that $(n, a b N)=1$, if $n \mid \frac{N-b p_{1} p_{2} p_{3}}{a}$, then $(p_{1} ,n)=1$ and $(p_{2} , n)=1$. Moreover, if $\left(\frac{N-b p_{1} p_{2} p_{3}}{a n}, a\right)=1$, then we have $b p_{1} p_{2} p_{3} \equiv$ $N+j a n\left(\bmod a^{2} n\right)$ for some $j$ such that $0 \leqslant j \leqslant a-1$ and $(j, a)=1$. Conversely, if $b p_{1} p_{2} p_{3}=N+j a n+s a^{2} n$ for some integer $j$ such that $0 \leqslant j \leqslant a$ and $(j, a)=1$, some integer $n$ relatively prime to $p_{1} p_{2}$ such that $an \mid\left(N-b p_{1} p_{2} p_{3}\right)$, and some integer $s$, then $\left(\frac{N-b p_{1} p_{2} p_{3}}{a n}, a\right)=(-j, a)=1$. Since $j b p_{1} p_{2}$ runs through the reduced residues modulo $a$ when $j$ runs through the reduced residues modulo $a$ and $\pi\left(x ; k, 1,1\right)=\pi\left(\frac{x}{k} ; 1,1\right)$, for square-free integers $n$ such that $(n, a b N)=1$, we have
\begin{align}
\nonumber \left|\eta\left(X_{\mathcal{B}_1}, n\right)\right| =&\left|\sum_{\substack{a \in \mathcal{B}_1 \\
a \equiv 0(\bmod n)}} 1-\frac{\omega(n)}{n} X_{\mathcal{B}_1}\right|=\left|\sum_{\substack{a \in \mathcal{B}_1 \\
a \equiv 0(\bmod n)}} 1-\frac{X_{\mathcal{B}_1}}{\varphi(n)}\right| \\
\nonumber =&\left|\sum_{\substack{(\frac{N}{b})^{\frac{1}{13.2}} \leqslant p_{1} <(\frac{N}{b})^{\frac{1}{3}}\leqslant p_{2} <(\frac{N}{b p_{1}})^{\frac{1}{2}},(p_1 p_2,N)=1 \\
\left(p_{1} p_{2}, n\right)=1,
0 \leqslant j \leqslant a-1,(j, a)=1}} \pi\left(N ; b p_{1} p_{2}, a^{2} n, N+j a n\right)\right.\\
\nonumber&\left. -\sum_{\substack{(\frac{N}{b})^{\frac{1}{13.2}} \leqslant p_{1} <(\frac{N}{b})^{\frac{1}{3}}\leqslant p_{2} <(\frac{N}{b p_{1}})^{\frac{1}{2}} \\
\left(p_{1} p_{2}, n\right)=1,
0 \leqslant j \leqslant a-1,(j, a)=1}} \frac{\pi\left(\frac{N}{b p_{1} p_{2}} ; a^{2}, N\left(b p_{1} p_{2}\right)_{a^{2}}^{-1}+j a\right)}{\varphi(n)}\right| \\
\nonumber \ll& \left|\sum_{\substack{(\frac{N}{b})^{\frac{1}{13.2}} \leqslant p_{1} <(\frac{N}{b})^{\frac{1}{3}}\leqslant p_{2} <(\frac{N}{b p_{1}})^{\frac{1}{2}},(p_1 p_2,N)=1 \\
\left(p_{1} p_{2}, n\right)=1,
0 \leqslant j \leqslant a-1,(j, a)=1}}\left(\pi\left(N ; b p_{1} p_{2}, a^{2} n, N+j a n\right)\right.\right. \\
\nonumber& \left.\left.-\frac{\pi\left(\frac{N}{b p_{1} p_{2}} ; a^{2}, N\left(b p_{1} p_{2}\right)_{a^{2}}^{-1}+j a\right)}{\varphi(n)}\right)\right| \\
\nonumber& +\sum_{\substack{(\frac{N}{b})^{\frac{1}{13.2}} \leqslant p_{1} <(\frac{N}{b})^{\frac{1}{3}}\leqslant p_{2} <(\frac{N}{b p_{1}})^{\frac{1}{2}} \\
\left(p_{1} p_{2}, nN\right)>1,
0 \leqslant j \leqslant a-1,(j, a)=1}} \frac{\pi\left(\frac{N}{b p_{1} p_{2}} ; a^{2}, N\left(b p_{1} p_{2}\right)_{a^{2}}^{-1}+j a\right)}{\varphi(n)}\\
\nonumber \ll&\left|\sum_{\substack{(\frac{N}{b})^{\frac{1}{13.2}} \leqslant p_{1} <(\frac{N}{b})^{\frac{1}{3}}\leqslant p_{2} <(\frac{N}{b p_{1}})^{\frac{1}{2}},(p_1 p_2,N)=1 \\
\left(p_{1} p_{2}, n\right)=1,
0 \leqslant j \leqslant a-1,(j, a)=1}}\left(\pi\left(N ; b p_{1} p_{2}, a^{2} n, N+j a n\right)-\frac{\pi\left(N ; b p_{1} p_{2}, 1,1\right)}{\varphi\left(a^{2} n\right)}\right)\right| \\
\nonumber& +\left| \sum_{\substack{(\frac{N}{b})^{\frac{1}{13.2}} \leqslant p_{1} <(\frac{N}{b})^{\frac{1}{3}}\leqslant p_{2} <(\frac{N}{b p_{1}})^{\frac{1}{2}},(p_1 p_2,N)=1 \\
\left(p_{1} p_{2}, n\right)=1,
0 \leqslant j \leqslant a-1,(j, a)=1}}\left(\frac{\pi\left(\frac{N}{b p_{1} p_{2}} ; a^{2}, N\left(b p_{1} p_{2}\right)_{a^{2}}^{-1}+j a\right)}{\varphi(n)}-\frac{\pi\left(\frac{N}{b p_{1} p_{2}} ; 1,1\right)}{\varphi\left(a^{2} n\right)}\right)\right| \\
\nonumber&+N^{\frac{12.2}{13.2}}(\log N)^{2}\\
\nonumber \ll&\left|\sum_{\substack{(\frac{N}{b})^{\frac{1}{13.2}} \leqslant p_{1} <(\frac{N}{b})^{\frac{1}{3}}\leqslant p_{2} <(\frac{N}{b p_{1}})^{\frac{1}{2}},(p_1 p_2,N)=1 \\
\left(p_{1} p_{2}, n\right)=1,
0 \leqslant j \leqslant a-1,(j, a)=1}}\left(\pi\left(N ; b p_{1} p_{2}, a^{2} n, N+j a n\right)-\frac{\pi\left(N ; b p_{1} p_{2}, 1,1\right)}{\varphi\left(a^{2} n\right)}\right)\right| \\
\nonumber& +\frac{1}{\varphi(n)}\left| \sum_{\substack{(\frac{N}{b})^{\frac{1}{13.2}} \leqslant p_{1} <(\frac{N}{b})^{\frac{1}{3}}\leqslant p_{2} <(\frac{N}{b p_{1}})^{\frac{1}{2}},(p_1 p_2,N)=1 \\
\left(p_{1} p_{2}, n\right)=1,
0 \leqslant j \leqslant a-1,(j, a)=1}}\left(\pi\left(\frac{N}{b p_{1} p_{2}} ; a^{2}, N\left(b p_{1} p_{2}\right)_{a^{2}}^{-1}+j a\right)-\frac{\pi\left(\frac{N}{b p_{1} p_{2}} ; 1,1\right)}{\varphi\left(a^{2}\right)}\right)\right| \\
&+N^{\frac{12.2}{13.2}}(\log N)^{2}.
\end{align}
By Lemma~\ref{l3} with
$$
g(k)= 
\begin{cases}
1, & \text { if } k \in \mathcal{E}_1 \\ 
0, & \text { otherwise }
\end{cases},
$$
we have
\begin{equation}
\sum_{\substack{n \leqslant D_{\mathcal{B}_1} \\ n \mid P(z_{\mathcal{B}_1})}} \left|\eta\left(X_{\mathcal{B}_1}, n\right)\right| \ll N(\log N)^{-5}.
\end{equation}

Then by (39)--(44) and some routine arguments we have
$$
S_{41} \leqslant (1+o(1)) \frac{8C(abN) N}{ab(\log N)^{2}} \int_{2}^{12.2} \frac{\log \left(2-\frac{3}{s+1}\right)}{s} d s.
$$
Similarly, we have
$$
S_{42} \leqslant(1+o(1)) \frac{8C(abN) N}{ab(\log N)^{2}}\int_{2.604}^{7.4} \frac{\log \left(2.604-\frac{3.604}{s+1}\right)}{s} d s,
$$
$$
S_{4}=S_{41}+S_{42} \leqslant (1+o(1)) \frac{8C(abN) N}{ab(\log N)^{2}} \left(\int_{2}^{12.2} \frac{\log \left(2-\frac{3}{s+1}\right)}{s} d s+\int_{2.604}^{7.4} \frac{\log \left(2.604-\frac{3.604}{s+1}\right)}{s} d s\right)
$$
\begin{equation}
\leqslant 10.69152 \frac{C(abN) N}{ab(\log N)^{2}},
\end{equation}
\begin{equation}
S_7 \leqslant(1+o(1)) \frac{8C(abN) N}{ab(\log N)^{2}}\int_{2}^{2.604} \frac{\log (s-1)}{s} d s
\leqslant 0.5160672 \frac{C(abN) N}{ab(\log N)^{2}}.
\end{equation}

\subsection{Evaluation of $S_{6}$}
Let $D_{\mathcal{C}_1}=N^{1 / 2}(\log N)^{-B}$. By Chen's switching principle and similar arguments as in \cite{CAI867}, we know that
\begin{equation}
|\mathcal{F}_1|<\left(\frac{N}{b}\right)^{\frac{12.2}{13.2}}, \quad
\left(\frac{N}{b}\right)^{\frac{1}{4.4}}<e< \left(\frac{N}{b}\right)^{\frac{12.2}{13.2}} \ \operatorname{for}\ e \in \mathcal{F}_1, \quad
S_{61} \leqslant S\left(\mathcal{C}_1;\mathcal{P},D_{\mathcal{C}_1}^{\frac{1}{2}}\right)+O\left(N^{\frac{12.2}{13.2}}\right).
\end{equation}
By Lemma~\ref{Wfunction} for $z_{\mathcal{C}_1}=D_{\mathcal{C}_1}^{\frac{1}{2}}=N^{\frac{1}{4}}(\log N)^{-B/2}$ we have
\begin{equation}
W(z_{\mathcal{C}_1})=\frac{8 e^{-\gamma} C(abN)(1+o(1))}{\log N}, \quad F(2)=e^{\gamma}.
\end{equation}
By Lemma~\ref{l4} we have
\begin{align}
\nonumber |\mathcal{C}_1|&=
\sum_{mp_1 p_2 p_4 \in \mathcal{F}_1}\sum_{\substack{p_2<p_3<\min((\frac{N}{b})^\frac{1}{8.4},(\frac{N}{bmp_1p_2p_4}))
\\p_{3} \equiv N(bm p_{1} p_{2} p_4)_{a^{2}}^{-1}+j a (\bmod a^{2}) \\ 0 \leqslant j \leqslant a-1,(j, a)=1 }}1 \\
\nonumber &=\sum_{\substack{(\frac{N}{b})^{\frac{1}{13.2}} \leqslant p_1 < p_4 < p_2< p_3<(\frac{N}{b})^{\frac{1}{8.4}} \\ (p_1 p_2 p_3 p_4, N)=1} }\sum_{\substack{1\leqslant m\leqslant \frac{N}{bp_1 p_2 p_3 p_4}\\\left(m, p_{1}^{-1} abN P\left(p_{4}\right)\right)=1 }}\frac{\varphi(a)}{\varphi(a^2)}+O\left(N^{\frac{12.2}{13.2}}\right)\\
\nonumber &<(1+o(1))\frac{N}{ab}\sum_{(\frac{N}{b})^{\frac{1}{13.2}} \leqslant p_{1}<p_{4}<p_{2}<p_{3}<(\frac{N}{b})^{\frac{1}{8.4}}} \frac{0.5617}{p_{1} p_{2} p_{3} p_{4} \log p_{4}}+O\left(N^{\frac{12.2}{13.2}}\right)\\
&=(1+o(1))\frac{0.5617N}{ab\log N}\int_{\frac{1}{13.2}}^{\frac{1}{8.4}} \frac{d t_{1}}{t_{1}} \int_{t_{1}}^{\frac{1}{8.4}} \frac{1}{t_{2}}\left(\frac{1}{t_{1}}-\frac{1}{t_{2}}\right) \log \frac{1}{8.4 t_{2}} d t_{2} .
\end{align}

To deal with the error terms, for an integer $n$ such that $(n, a b N)>1$, similarly to the discussion for $\eta\left(X_{\mathcal{B}_1}, n\right)$, we have $\eta\left(|{\mathcal{C}_1}|, n\right)=0$. 
For a square-free integer $n$ that is relatively prime to $a b N$, if $n \mid \frac{N-bm p_{1} p_{2} p_{3} p_4}{a}$, then $(p_{1} ,n)=1, (p_{2} ,n)=1$ and $(p_{4} , n)=1$. Moreover, if $\left(\frac{N-bm p_{1} p_{2} p_{3} p_4}{a n}, a\right)=1$, then we have $bm p_{1} p_{2} p_{3} p_4 \equiv$ $N+j a n\left(\bmod a^{2} n\right)$ for some $j$ such that $0 \leqslant j \leqslant a-1$ and $(j, a)=1$. Conversely, if $bm p_{1} p_{2} p_{3} p_4=N+j a n+s a^{2} n$ for some integer $j$ such that $0 \leqslant j \leqslant a$ and $(j, a)=1$, some integer $n$ relatively prime to $p_{1} p_{2} p_4$ such that $an \mid\left(N-bm p_{1} p_{2} p_{3} p_4\right)$, and some integer $s$, then $\left(\frac{N-bm p_{1} p_{2} p_{3} p_4}{a n}, a\right)=(-j, a)=1$. Since $j bm p_{1} p_{2} p_4$ runs through the reduced residues modulo $a$ when $j$ runs through the reduced residues modulo $a$ and $\pi\left(x ; k, 1,1\right)=\pi\left(\frac{x}{k} ; 1,1\right)$, for a square-free integer $n$ relatively prime to $a b N$, we have
\begin{align}
\nonumber \left|\eta\left(|{\mathcal{C}_1}|, n\right)\right| =&\left|\sum_{\substack{a \in \mathcal{C}_1 \\
a \equiv 0(\bmod n)}} 1-\frac{\omega(n)}{n} |\mathcal{C}_1|\right|=\left|\sum_{\substack{a \in \mathcal{C}_1 \\
a \equiv 0(\bmod n)}} 1-\frac{|\mathcal{C}_1|}{\varphi(n)}\right| \\
\nonumber =&\left|\sum_{\substack{e \in \mathcal{F}_1 \\ (e,n)=1}}\left(\sum_{\substack{p_2<p_3<\min((\frac{N}{b})^\frac{1}{8.4},(\frac{N}{be}))
\\be p_3 \equiv N+jan (\bmod a^{2}n) \\ 0 \leqslant j \leqslant a-1,(j, a)=1 }}1-\frac{1}{\varphi(n)}\sum_{\substack{p_2<p_3<\min((\frac{N}{b})^\frac{1}{8.4},(\frac{N}{be}))
\\p_{3} \equiv N(bm p_{1} p_{2} p_4)_{a^{2}}^{-1}+j a (\bmod a^{2}) \\ 0 \leqslant j \leqslant a-1,(j, a)=1 }}1\right)\right|\\
&+\frac{1}{\varphi(n)}\sum_{\substack{e \in \mathcal{F}_1 \\ (e,n)>1}}\sum_{\substack{p_2<p_3<\min((\frac{N}{b})^\frac{1}{8.4},(\frac{N}{be}))
\\p_{3} \equiv N(bm p_{1} p_{2} p_4)_{a^{2}}^{-1}+j a (\bmod a^{2}) \\ 0 \leqslant j \leqslant a-1,(j, a)=1 }}1.
\end{align}
Let
$$
g(k)=\sum_{\substack{e=k \\ e\in \mathcal{F}_1 \\ 0 \leqslant j \leqslant a-1,(j, a)=1}}1,
$$
then
\begin{align}
\nonumber \left|\eta\left(|{\mathcal{C}_1}|, n\right)\right|  \ll& \left|\sum_{\substack{(\frac{N}{b})^{\frac{1}{4.4}}<k< (\frac{N}{b})^{\frac{12.2}{13.2}}\\ (k,n)=1}}g(k)\left(\sum_{\substack{p_2<p_3<\min((\frac{N}{b})^\frac{1}{8.4},(\frac{N}{bk}))
\\bk p_3 \equiv N+jan (\bmod a^{2}n) }}1-\frac{1}{\varphi(n)}\sum_{\substack{p_2<p_3<\min((\frac{N}{b})^\frac{1}{8.4},(\frac{N}{bk}))
\\p_{3} \equiv N(bm p_{1} p_{2} p_4)_{a^{2}}^{-1}+j a (\bmod a^{2}) }}1\right)\right|\\
\nonumber &+\frac{1}{\varphi(n)}\sum_{\substack{(\frac{N}{b})^{\frac{1}{4.4}}<k< (\frac{N}{b})^{\frac{12.2}{13.2}}\\ (k,n) \geqslant (\frac{N}{b})^{\frac{1}{13.2}}}} \sum_{\substack{p_2<p_3<\min((\frac{N}{b})^\frac{1}{8.4},(\frac{N}{bk}))
\\p_{3} \equiv N(bm p_{1} p_{2} p_4)_{a^{2}}^{-1}+j a (\bmod a^{2}) }}1\\
\nonumber \ll&\left|\sum_{\substack{(\frac{N}{b})^{\frac{1}{4.4}}<k< (\frac{N}{b})^{\frac{7.4}{8.4}}\\ (k,n)=1}}g(k)\left(\pi\left(bk\left(\frac{N}{b}\right)^\frac{1}{8.4}; b k, a^2 n, N + j a n\right)-\frac{\pi\left(\left(\frac{N}{b}\right)^\frac{1}{8.4}; a^2, N(bm p_{1} p_{2} p_4)_{a^{2}}^{-1}+j a\right)}{\varphi(n)}\right)\right|\\
\nonumber &+\left|\sum_{\substack{(\frac{N}{b})^{\frac{7.4}{8.4}}<k< (\frac{N}{b})^{\frac{12.2}{13.2}}\\ (k,n)=1}}g(k)\left(\pi\left(N; b k, a^2 n, N + j a n\right)-\frac{\pi\left(\frac{N}{bk}; a^2, N(bm p_{1} p_{2} p_4)_{a^{2}}^{-1}+j a\right)}{\varphi(n)}\right)\right|\\
\nonumber &+\left|\sum_{\substack{(\frac{N}{b})^{\frac{1}{4.4}}<k< (\frac{N}{b})^{\frac{12.2}{13.2}}\\ (k,n)=1}}g(k)\left(\pi\left(b k p_2; b k, a^2 n, N + j a n\right)-\frac{\pi\left(p_2; a^2, N(bm p_{1} p_{2} p_4)_{a^{2}}^{-1}+j a\right)}{\varphi(n)}\right)\right|\\
\nonumber &+N^{\frac{12.2}{13.2}}(\log N)^2\\
\nonumber \ll&\left|\sum_{\substack{(\frac{N}{b})^{\frac{1}{4.4}}<k< (\frac{N}{b})^{\frac{7.4}{8.4}}\\ (k,n)=1}}g(k)\left(\pi\left(bk\left(\frac{N}{b}\right)^\frac{1}{8.4}; b k, a^2 n, N + j a n\right)-\frac{\pi\left(bk\left(\frac{N}{b}\right)^\frac{1}{8.4}; b k, 1, 1\right)}{\varphi(a^2 n)}\right)\right|\\
\nonumber &+\left|\sum_{\substack{(\frac{N}{b})^{\frac{1}{4.4}}<k< (\frac{N}{b})^{\frac{7.4}{8.4}}\\ (k,n)=1}}g(k)\left(\frac{\pi\left(\left(\frac{N}{b}\right)^\frac{1}{8.4}; a^2, N(bm p_{1} p_{2} p_4)_{a^{2}}^{-1}+j a\right)}{\varphi(n)}-\frac{\pi\left(\left(\frac{N}{b}\right)^\frac{1}{8.4};  1, 1\right)}{\varphi(a^2 n)}\right)\right|\\
\nonumber &+\left|\sum_{\substack{(\frac{N}{b})^{\frac{7.4}{8.4}}<k< (\frac{N}{b})^{\frac{12.2}{13.2}}\\ (k,n)=1}}g(k)\left(\pi\left(N; b k, a^2 n, N + j a n\right)-\frac{\pi\left(N; b k, 1, 1\right)}{\varphi(a^2 n)}\right)\right|\\
\nonumber &+\left|\sum_{\substack{(\frac{N}{b})^{\frac{7.4}{8.4}}<k< (\frac{N}{b})^{\frac{12.2}{13.2}}\\ (k,n)=1}}g(k)\left(\frac{\pi\left(\frac{N}{bk}; a^2, N(bm p_{1} p_{2} p_4)_{a^{2}}^{-1}+j a\right)}{\varphi(n)}-\frac{\pi\left(\frac{N}{bk}; 1,1\right)}{\varphi(a^2 n)}\right)\right|\\
\nonumber &+\left|\sum_{\substack{(\frac{N}{b})^{\frac{1}{4.4}}<k< (\frac{N}{b})^{\frac{12.2}{13.2}}\\ (k,n)=1}}g(k)\left(\pi\left(b k p_2; b k, a^2 n, N + j a n\right)-\frac{\pi\left(b k p_2; b k, 1,1\right)}{\varphi(a^2 n)}\right)\right|\\
\nonumber &+\left|\sum_{\substack{(\frac{N}{b})^{\frac{1}{4.4}}<k< (\frac{N}{b})^{\frac{12.2}{13.2}}\\ (k,n)=1}}g(k)\left(\frac{\pi\left(p_2; a^2, N(bm p_{1} p_{2} p_4)_{a^{2}}^{-1}+j a\right)}{\varphi(n)}-\frac{\pi\left(p_2;1,1\right)}{\varphi(a^2 n)}\right)\right|\\
\nonumber &+N^{\frac{12.2}{13.2}}(\log N)^2\\
\nonumber
\ll&\left|\sum_{\substack{(\frac{N}{b})^{\frac{1}{4.4}}<k< (\frac{N}{b})^{\frac{7.4}{8.4}}\\ (k,n)=1}}g(k)\left(\pi\left(bk\left(\frac{N}{b}\right)^\frac{1}{8.4}; b k, a^2 n, N + j a n\right)-\frac{\pi\left(bk\left(\frac{N}{b}\right)^\frac{1}{8.4}; b k, 1, 1\right)}{\varphi(a^2 n)}\right)\right|\\
\nonumber &+\frac{1}{\varphi(n)}\left|\sum_{\substack{(\frac{N}{b})^{\frac{1}{4.4}}<k< (\frac{N}{b})^{\frac{7.4}{8.4}}\\ (k,n)=1}}g(k)\left(\pi\left(\left(\frac{N}{b}\right)^\frac{1}{8.4}; a^2, N(bm p_{1} p_{2} p_4)_{a^{2}}^{-1}+j a\right)-\frac{\pi\left(\left(\frac{N}{b}\right)^\frac{1}{8.4};  1, 1\right)}{\varphi(a^2 )}\right)\right|\\
\nonumber &+\left|\sum_{\substack{(\frac{N}{b})^{\frac{7.4}{8.4}}<k< (\frac{N}{b})^{\frac{12.2}{13.2}}\\ (k,n)=1}}g(k)\left(\pi\left(N; b k, a^2 n, N + j a n\right)-\frac{\pi\left(N; b k, 1, 1\right)}{\varphi(a^2 n)}\right)\right|\\
\nonumber &+\frac{1}{\varphi(n)}\left|\sum_{\substack{(\frac{N}{b})^{\frac{7.4}{8.4}}<k< (\frac{N}{b})^{\frac{12.2}{13.2}}\\ (k,n)=1}}g(k)\left(\pi\left(\frac{N}{bk}; a^2, N(bm p_{1} p_{2} p_4)_{a^{2}}^{-1}+j a\right)-\frac{\pi\left(\frac{N}{bk}; 1,1\right)}{\varphi(a^2 )}\right)\right|\\
\nonumber &+\left|\sum_{\substack{(\frac{N}{b})^{\frac{1}{4.4}}<k< (\frac{N}{b})^{\frac{12.2}{13.2}}\\ (k,n)=1}}g(k)\left(\pi\left(b k p_2; b k, a^2 n, N + j a n\right)-\frac{\pi\left(b k p_2; b k, 1,1\right)}{\varphi(a^2 n)}\right)\right|\\
\nonumber &+\frac{1}{\varphi(n)}\left|\sum_{\substack{(\frac{N}{b})^{\frac{1}{4.4}}<k< (\frac{N}{b})^{\frac{12.2}{13.2}}\\ (k,n)=1}}g(k)\left(\pi\left(p_2; a^2, N(bm p_{1} p_{2} p_4)_{a^{2}}^{-1}+j a\right)-\frac{\pi\left(p_2;1,1\right)}{\varphi(a^2 )}\right)\right|\\
&+N^{\frac{12.2}{13.2}}(\log N)^2.
\end{align}
By Lemmas~\ref{l3}--\ref{remark1}, we have
\begin{equation}
\sum_{\substack{n \leqslant D_{\mathcal{C}_1} \\ n \mid P(z_{\mathcal{C}_1})}} \left|\eta\left(|{\mathcal{C}_1}|, n\right)\right| \ll N(\log N)^{-5}.
\end{equation}

By (47)--(52) we have
\begin{align}
\nonumber S_{61} &\leqslant(1+o(1)) \frac{0.5617 \times 8 C(abN) N}{ab(\log N)^2} \int_{\frac{1}{13.2}}^{\frac{1}{8.4}} \frac{d t_{1}}{t_{1}} \int_{t_{1}}^{\frac{1}{8.4}} \frac{1}{t_{2}}\left(\frac{1}{t_{1}}-\frac{1}{t_{2}}\right) \log \frac{1}{8.4 t_{2}} d t_{2} \\
&\leqslant 0.0864362 \frac{C(abN) N}{ab(\log N)^2}.
\end{align}
Similarly, we have
\begin{align}
\nonumber S_{62}=& \sum_{(\frac{N}{b})^{\frac{1}{13.2}} \leqslant p_{1}<p_{2}<p_{3}<(\frac{N}{b})^{ \frac{1}{8.4}} \leqslant p_{4}<(\frac{N}{b})^{\frac{1.4}{8.4}} } S\left(\mathcal{A}_{p_{1} p_{2} p_{3} p_{4}} ; \mathcal{P}\left(p_{1}\right), p_{2}\right)\\
\nonumber &+\sum_{(\frac{N}{b})^{\frac{1}{13.2}} \leqslant p_{1}<p_{2}<p_{3}<(\frac{N}{b})^{ \frac{1}{8.4}}<(\frac{N}{b})^{ \frac{1.4}{8.4}} \leqslant p_{4}<(\frac{N}{b})^{\frac{4.6}{13.2}}p^{-1}_3 } S\left(\mathcal{A}_{p_{1} p_{2} p_{3} p_{4}} ; \mathcal{P}\left(p_{1}\right), p_{2}\right)\\
\nonumber \leqslant& (1+o(1)) \frac{0.5617 \times 8 C(abN) N}{ab(\log N)^2} \left(21.6 \log \frac{13.2}{8.4}-9.6\right) \log 1.4 \\
\nonumber &+(1+o(1)) \frac{0.5644 \times 8 C(abN) N}{ab(\log N)^2}\int_{\frac{1}{13.2}}^{\frac{1}{8.4}} \frac{d t_{1}}{t_{1}} \int_{t_{1}}^{\frac{1}{8.4}} \frac{1}{t_{2}}\left(\frac{1}{t_{1}}-\frac{1}{t_{2}}\right) \log \left(\frac{8.4}{1.4}\left(\frac{4.6}{13.2}-t_{2}\right)\right) d t_{2}\\
\leqslant& 0.5208761\frac{ C(abN) N}{ab(\log N)^2}.
\end{align}
By (53) and (54) we have
\begin{align}
\nonumber S_{6}=S_{61}+S_{62} &\leqslant 0.0864362 \frac{C(abN) N}{ab(\log N)^2} +0.5208761 \frac{C(abN) N}{ab(\log N)^2}\\
&\leqslant 0.6073123 \frac{C(abN) N}{ab(\log N)^2}.
\end{align}

\subsection{Evaluation of $S_{5}$} For $p \geqslant \left(\frac{N}{b}\right)^{\frac{4.1001}{13.2}}$ we have
$$
\underline{p}^{\frac{1}{2.5}} \leqslant \left(\frac{N}{b}\right)^{\frac{1}{13.2}}, \quad S\left(\mathcal{A}_{p};\mathcal{P}, \left(\frac{N}{b}\right)^{\frac{1}{13.2}}\right) \leqslant S\left(\mathcal{A}_{p};\mathcal{P}, \underline{p}^{\frac{1}{2.5}}\right).
$$
By Lemma~\ref{l34} we have
\begin{align}
\nonumber S_{51}&=\sum_{\substack{(\frac{N}{b})^{\frac{4.1001}{13.2}} \leqslant p<(\frac{N}{b})^{\frac{1}{3}} \\ (p, N)=1}} S\left(\mathcal{A}_{p};\mathcal{P},\left(\frac{N}{b}\right)^{\frac{1}{13.2}}\right)\\
&\leqslant \sum_{\substack{(\frac{N}{b})^{\frac{4.1001}{13.2}} \leqslant p<(\frac{N}{b})^{\frac{1}{3}} \\ (p, N)=1}} S\left(\mathcal{A}_{p};\mathcal{P},\underline{p}^{\frac{1}{2.5}}\right)
\leqslant \Gamma_{1}-\frac{1}{2} \Gamma_{2}+\frac{1}{2} \Gamma_{3}+O\left(N^{\frac{19}{20}}\right) .
\end{align}
By Lemmas~\ref{l1}, ~\ref{l2}, ~\ref{Wfunction}, ~\ref{l3} and some routine arguments we get
\begin{align}
\nonumber\Gamma_{1}&=\sum_{\substack{(\frac{N}{b})^{\frac{4.1001}{13.2}} \leqslant p<(\frac{N}{b})^{\frac{1}{3}} \\ (p, N)=1}} S\left(\mathcal{A}_{p};\mathcal{P},\underline{p}^{\frac{1}{3.675}}\right)\\
&\leqslant(1+o(1)) \frac{8C(abN) N}{ab(\log N)^{2}}\left(\int_{\frac{4.1001}{13.2}}^{\frac{1}{3}} \frac{d t}{t(1-2 t)}\right)\left(1+\int_{2}^{2.675} \frac{\log (t-1)}{t} d t\right),\\
\nonumber\Gamma_{2}&=\sum_{\substack{(\frac{N}{b})^{\frac{4.1001}{13.2}} \leqslant p<(\frac{N}{b})^{\frac{1}{3}} \\ (p, N)=1}} \sum_{\substack{\underline{p}^{\frac{1}{3.675}} \leqslant p_1<\underline{p}^{\frac{1}{2.5}} \\ (p_1, N)=1}}S\left(\mathcal{A}_{p p_1};\mathcal{P}, \underline{p}^{\frac{1}{3.675}}\right)\\
&\geqslant(1+o(1)) \frac{8C(abN) N}{ab(\log N)^{2}}\left(\int_{\frac{4.1001}{13.2}}^{\frac{1}{3}} \frac{d t}{t(1-2 t)}\right)\left(\int_{1.5}^{2.675} \frac{\log \left(2.675-\frac{3.675}{t+1}\right)}{t} d t\right).
\end{align}
By arguments similar to the evaluation of $S_{61}$ we get that
\begin{align}
\nonumber \Gamma_{3}&=\sum_{\substack{(\frac{N}{b})^{\frac{4.1001}{13.2}} \leqslant p<(\frac{N}{b})^{\frac{1}{3}} \\ (p, N)=1}} 
\sum
_{
\substack{\underline{p}^{\frac{1}{3.675}} \leqslant p_1<p_2<p_3<\underline{p}^{\frac{1}{2.5}} \\ (p_1 p_2 p_3, N)=1}
}
S\left(\mathcal{A}_{p p_1 p_2 p_3};\mathcal{P}(p_1), p_2\right)\\
\nonumber &\leqslant (1+o(1))\frac{8C(abN)}{\log N} \sum_{\substack{(\frac{N}{b})^{\frac{4.1001}{13.2}} \leqslant p<(\frac{N}{b})^{\frac{1}{3}} \\ (p, N)=1}} 
\sum_{\substack{\underline{p}^{\frac{1}{3.675}} \leqslant p_1<p_2<p_3<\underline{p}^{\frac{1}{2.5}} \\ (p_1 p_2 p_3, N)=1}}\sum_{\substack{ m\leqslant \frac{N}{bp p_1 p_2 p_3}\\\left(m, p_{1}^{-1} abN P\left(p_{2}\right)\right)=1 }}\frac{\varphi(a)}{\varphi(a^2)}\\
\nonumber &\leqslant (1+o(1))\frac{8C(abN) N}{1.763 ab \log N}\sum_{\substack{(\frac{N}{b})^{\frac{4.1001}{13.2}} \leqslant p<(\frac{N}{b})^{\frac{1}{3}} \\ (p, N)=1}} \frac{1}{p \log \underline{p}} \int_{\frac{1}{3.675}}^{\frac{1}{2.5}} \int_{t_{1}}^{\frac{1}{2.5}} \int_{t_{2}}^{\frac{1}{2.5}} \frac{d t_{1} d t_{2} d t_{3}}{t_{1} t_{2}^{2} t_{3}} \\
&\leqslant(1+o(1)) \frac{16C(abN) N}{1.763ab(\log N)^{2}}\left(\int_{\frac{4.1001}{13.2}}^{\frac{1}{3}} \frac{d t}{t(1-2 t)}\right)\left(6.175 \log \frac{3.675}{2.5}-2.35\right).
\end{align}
By (56)--(59) we have
$$
S_{51}=\sum_{\substack{(\frac{N}{b})^{\frac{4.1001}{13.2}} \leqslant p<(\frac{N}{b})^{\frac{1}{3}} \\ (p, N)=1}} S\left(\mathcal{A}_{p};\mathcal{P},\left(\frac{N}{b}\right)^{\frac{1}{13.2}}\right)
$$
$$
\begin{gathered}
\leqslant(1+o(1)) \frac{8C(abN) N}{ab(\log N)^{2}}\left(\int_{\frac{4.1001}{13.2}}^{\frac{1}{3}} \frac{d t}{t(1-2 t)}\right)\times\\ 
\left(1+\int_{2}^{2.675} \frac{\log (t-1)}{t} d t-\frac{1}{2} \int_{1.5}^{2.675} \frac{\log \left(2.675-\frac{3.675}{t+1}\right)}{t} d t
+\frac{1}{1.763}\left(6.175 \log \frac{3.675}{2.5}-2.35\right)\right).
\end{gathered}
$$
Similarly, we have
$$
S_{52}=\sum_{\substack{(\frac{N}{b})^{\frac{3.6}{13.2}} \leqslant p<(\frac{N}{b})^{\frac{1}{3.604}} \\ (p, N)=1} }S\left(\mathcal{A}_{p};\mathcal{P},\left(\frac{N}{b}\right)^{\frac{1}{8.4}}\right)
$$
$$
\begin{gathered}
\leqslant(1+o(1)) \frac{8C(abN) N}{ab(\log N)^{2}}\left(\int_{\frac{3.6}{13.2}}^{\frac{1}{3.604}} \frac{d t}{t(1-2 t)}\right)\times\\ 
\left(1+\int_{2}^{2.675} \frac{\log (t-1)}{t} d t-\frac{1}{2} \int_{1.5}^{2.675} \frac{\log \left(2.675-\frac{3.675}{t+1}\right)}{t} d t
+\frac{1}{1.763}\left(6.175 \log \frac{3.675}{2.5}-2.35\right)\right)
\end{gathered}
$$
$$
S_{5}=S_{51}+S_{52}
$$
$$
\begin{gathered}
\leqslant(1+o(1)) \frac{8C(abN) N}{ab(\log N)^{2}}\left(\int_{\frac{4.1001}{13.2}}^{\frac{1}{3}} \frac{d t}{t(1-2 t)}+\int_{\frac{3.6}{13.2}}^{\frac{1}{3.604}} \frac{d t}{t(1-2 t)}\right)\times\\ 
\left(1+\int_{2}^{2.675} \frac{\log (t-1)}{t} d t-\frac{1}{2} \int_{1.5}^{2.675} \frac{\log \left(2.675-\frac{3.675}{t+1}\right)}{t} d t
+\frac{1}{1.763}\left(6.175 \log \frac{3.675}{2.5}-2.35\right)\right)
\end{gathered}
$$
\begin{equation}
\leqslant 1.87206 \frac{C(abN) N}{ab(\log N)^{2}}.
\end{equation}

\subsection{Proof of theorem 1.1}
By Lemma~\ref{l32}, (36)--(38), (45)--(46), (55) and (60) we get
$$
S_{1}+S_{2} \geqslant 58.974416 \frac{C(abN) N}{ab(\log N)^{2}},
$$
$$
S_{3}+S_{4}+S_{5}+S_{6}+2S_{7} \leqslant 55.505987 \frac{C(abN) N}{ab(\log N)^{2}},
$$
$$
4R_{a,b}(N) \geqslant (S_{1}+S_{2})-(S_{3}+S_{4}+S_{5}+S_{6}+2S_{7}) \geqslant 3.468429 \frac{C(abN) N}{ab(\log N)^{2}},
$$
$$
R_{a,b}(N) \geqslant 0.8671 \frac{C(abN) N}{ab(\log N)^{2}}.
$$
Theorem~\ref{t1} is proved.

\section{Proof of Theorem 1.2}
In this section, sets $\mathcal{A}_2$, $\mathcal{B}_2$, $\mathcal{C}_2$,
$\mathcal{C}_3$, $\mathcal{E}_2$,
$\mathcal{F}_2$ and $\mathcal{F}_3$ are defined respectively. We define the function $\omega$ as $\omega(p)=0$ for primes $p \mid a b N$ and $\omega(p)=\frac{p}{p-1}$ for other primes.
\subsection{Evaluation of $S_{1}^{\prime}, S_{2}^{\prime}, S_{3}^{\prime}$}
Let $D_{\mathcal{A}_{2}}=\left(\frac{N}{b}\right)^{\theta / 2}\left(\log \left(\frac{N}{b}\right)\right)^{-B}$ for some positive constant $B$. We can take 
\begin{equation}
X_{\mathcal{A}_2}=\sum_{\substack{0 \leqslant k \leqslant b-1 \\(k, b)=1}}\pi\left(\frac{N^\theta}{a} ; b^{2}, N a_{b^{2}}^{-1}+k b\right) \sim \frac{\varphi(b) N^\theta}{a \varphi\left(b^{2}\right) \log N^\theta} \sim \frac{N^\theta}{a b\theta \log N}.
\end{equation}
so that $|\mathcal{A}_2| \sim X_{\mathcal{A}_2}$. By Lemma~\ref{Wfunction} for $z_{\mathcal{A}_2}=\left(\frac{N}{b}\right)^{\frac{1}{\alpha}}$ we have
\begin{equation}
W(z_{\mathcal{A}_2})=\frac{2\alpha e^{-\gamma} C(abN)(1+o(1))}{\log N}.
\end{equation}

To deal with the error terms, for an integer $n$ such that $(n, a b N)>1$, similarly to the discussion for $\eta\left(X_{\mathcal{A}_1}, n\right)$, we have $\eta\left(X_{\mathcal{A}_2}, n\right)=0$.
For a square-free integer $n \leqslant D_{\mathcal{A}_2}$ such that $(n, abN)=1$, to make $n \mid \frac{N-a p}{b}$ for some $\frac{N-a p}{b} \in \mathcal{A}_2$, we need $a p \equiv N(\bmod b n)$, which implies $a p \equiv N+k b n$ $\left(\bmod b^{2} n\right)$ for some $0 \leqslant k \leqslant b-1$. Since $\left(\frac{N-a p}{b n}, b\right)=1$, we can further require $(k, b)=1$. When $k$ runs through the reduced residues modulo $b$, we know $k a_{b^{2} n}^{-1}$ also runs through the reduced residues modulo $b$. Therefore, we have $p \equiv N a_{b^{2} n}^{-1} +k b n\left(\bmod b^{2} n\right)$ for some $0 \leqslant k \leqslant b-1$ such that $(k, b)=1$. Conversely, if $p=N a_{b^{2} n}^{-1} +k b n+m b^{2} n$ for some integer $m$ and some $0 \leqslant k \leqslant b-1$ such that $(k, b)=1$, then $\left(\frac{N-a p}{b n}, b\right)=$ $\left(\frac{N-a a_{b^{2} n}^{-1} N-a k b n-a m b^{2} n}{b n}, b\right)=(-a k, b)=1$. Therefore, for square-free integers $n$ such that $(n, abN)=1$, we have
\begin{align}
\nonumber \left|\eta\left(X_{\mathcal{A}_2}, n\right)\right|  =&\left|\sum_{\substack{a \in \mathcal{A}_2 \\
a \equiv 0(\bmod n)}} 1-\frac{\omega(n)}{n} X_{\mathcal{A}_2}\right| \\
\nonumber=&\left|\sum_{\substack{0 \leqslant k \leqslant b-1 \\
(k, \bar{b})=1}} \pi\left(\frac{N^\theta}{a} ; b^{2} n, N a_{b^{2} n}^{-1} +k b n\right)-\frac{X_{\mathcal{A}_2}}{\varphi(n)}\right| \\
\nonumber =&\left|\sum_{\substack{0 \leqslant k \leqslant b-1 \\
(k, b)=1}}\left( \pi\left(\frac{N^\theta}{a} ; b^{2} n, N a_{b^{2} n}^{-1} +k b n\right)- \frac{\pi\left(\frac{N^\theta}{a} ; b^{2}, N a_{b^{2}}^{-1}+k b\right)}{\varphi(n)}\right)\right|\\
\nonumber \ll&\left|\sum_{\substack{0 \leqslant k \leqslant b-1 \\
(k, b)=1}}\left( \pi\left(\frac{N^\theta}{a} ; b^{2} n, N a_{b^{2} n}^{-1} +k b n\right)-\frac{\pi\left(\frac{N^\theta}{a} ; 1,1\right)}{\varphi\left(b^2 n\right)}\right)\right|\\
\nonumber& +\left|\sum_{\substack{0 \leqslant k \leqslant b-1 \\
(k, b)=1}}\left( \frac{\pi\left(\frac{N^\theta}{a} ; b^{2}, N a_{b^{2}}^{-1} +k b\right)}{\varphi(n)}- \frac{\pi\left(\frac{N^\theta}{a} ; 1,1\right)}{\varphi\left(b^{2} n\right)}\right)\right| \\
\nonumber \ll& \sum_{\substack{0 \leqslant k \leqslant b-1 \\
(k, \bar{b})=1}}\left|\pi\left(\frac{N^\theta}{a} ; b^{2} n, N a_{b^{2} n}^{-1} +k b n\right)-\frac{\pi\left(\frac{N^\theta}{a} ; 1,1\right)}{\varphi\left(b^{2} n\right)}\right| \\
& +\frac{1}{\varphi(n)} \sum_{\substack{0 \leqslant k \leqslant b-1 \\
(k, \bar{b})=1}}\left|\pi\left(\frac{N^\theta}{a} ; b^{2}, N a_{b^{2}}^{-1} +k b\right)-\frac{\pi\left(\frac{N^\theta}{a} ; 1,1\right)}{\varphi\left(b^{2}\right)}\right| .
\end{align}
By Lemma~\ref{l3} with $g(k)=1$ for $k=1$ and $g(k)=0$ for $k>1$, we have
\begin{equation}
\sum_{\substack{n \leqslant D_{\mathcal{A}_2} \\ n \mid P(z_{\mathcal{A}_2})}} \left|\eta\left(X_{\mathcal{A}_2}, n\right)\right| \ll N^\theta(\log N)^{-5}
\end{equation}
and
\begin{equation}
\sum_{p}\sum_{\substack{n \leqslant \frac{D_{\mathcal{A}_{2}}}{p} \\ n \mid P(z_{\mathcal{A}_2})}} \left|\eta\left(X_{\mathcal{A}_2}, pn\right)\right| \ll N^\theta(\log N)^{-5}.
\end{equation}

Then by (61)--(65), Lemma~\ref{l1}, Lemma~\ref{l2} and some routine arguments we have
\begin{align}
\nonumber S_{11}^{\prime} \geqslant& X_{\mathcal{A}_2} W\left(z_{\mathcal{A}_2}\right)\left\{f\left(\frac{\theta/2}{1/14}\right)+O\left(\frac{1}{\log ^{\frac{1}{3}} D_{\mathcal{A}_2}}\right)\right\}-\sum_{\substack{n<D_{\mathcal{A}_2} \\ n \mid P(z_{\mathcal{A}_2})}} \left|\eta\left(X_{\mathcal{A}_2}, n\right)\right|\\
\nonumber \geqslant& \frac{N^\theta}{a b\theta \log N}\frac{2\times 14 e^{-\gamma} C(abN)(1+o(1))}{\log N}\times\\
\nonumber & \left(\frac{2 e^{\gamma}}{\frac{14\theta}{2}}\left(\log (7\theta-1)+\int_{2}^{7\theta-2} \frac{\log (s-1)}{s} \log \frac{7\theta-1}{s+1} d s\right)\right)\\
\nonumber \geqslant& (1+o(1)) \frac{8C(abN) N^\theta}{ab\theta^2 (\log N)^2}\left(\log (7\theta-1)+\int_{2}^{7\theta-2} \frac{\log (s-1)}{s} \log \frac{7\theta-1}{s+1} d s\right)\\
\nonumber \geqslant& 16.70802 \frac{C(abN) N^\theta}{ab (\log N)^2},\\
\nonumber S_{12}^{\prime} \geqslant& X_{\mathcal{A}_2} W\left(z_{\mathcal{A}_2}\right)\left\{f\left(\frac{\theta/2}{1/8.8}\right)+O\left(\frac{1}{\log ^{\frac{1}{3}} D_{\mathcal{A}_2}}\right)\right\}-\sum_{\substack{n<D_{\mathcal{A}_2} \\ n \mid P(z_{\mathcal{A}_2})}} \left|\eta\left(X_{\mathcal{A}_2}, n\right)\right|\\
\nonumber \geqslant& \frac{N^\theta}{a b\theta \log N}\frac{2\times 8.8 e^{-\gamma} C(abN)(1+o(1))}{\log N}\times\\
\nonumber &\left(\frac{2 e^{\gamma}}{\frac{8.8\theta}{2}}\left(\log (4.4\theta-1)+\int_{2}^{4.4\theta-2} \frac{\log (s-1)}{s} \log \frac{4.4\theta-1}{s+1} d s\right)\right)\\
\nonumber \geqslant& (1+o(1)) \frac{8C(abN) N^\theta}{ab\theta^2(\log N)^2}\left(\log (4.4\theta-1)+\int_{2}^{4.4\theta-2} \frac{\log (s-1)}{s} \log \frac{4.4\theta-1}{s+1} d s\right) \\
\nonumber \geqslant& 10.340342 \frac{C(abN) N^\theta}{ab(\log N)^2},\\
S_{1}^{\prime}=&3 S_{11}^{\prime}+S_{12}^{\prime} \geqslant 60.464402 \frac{C(abN) N^\theta}{ab(\log N)^2} .
\end{align}
Similarly, we have
\begin{align}
\nonumber S_{21}^{\prime} \geqslant& \frac{N^\theta}{a b\theta \log N}\frac{2\times 14 e^{-\gamma} C(abN)(1+o(1))}{\log N} \times \\
\nonumber &\sum_{\substack{(\frac{N}{b})^{\frac{1}{14}} \leqslant p_1<p_2<(\frac{N}{b})^{\frac{1}{8.8}} \\ (p_1 p_2, N)=1} }\frac{1}{p_1 p_2} f\left(14\left(\frac{\theta}{2}-\frac{\log p_1 p_2}{\log \frac{N}{b}}\right)\right)\\
\nonumber \geqslant& \frac{N^\theta}{a b\theta \log N}\frac{2\times 14 e^{-\gamma} C(abN)(1+o(1))}{\log N} \times \\
\nonumber &\sum_{\substack{(\frac{N}{b})^{\frac{1}{14}} \leqslant p_1<p_2<(\frac{N}{b})^{\frac{1}{8.8}} \\ (p_1 p_2, N)=1} }\frac{1}{p_1 p_2} \frac{2 e^{\gamma} \log \left(14\left(\frac{\theta}{2}-\frac{\log p_1 p_2}{\log \frac{N}{b}}\right)-1\right)}{14\left(\frac{\theta}{2}-\frac{\log p_1 p_2}{\log \frac{N}{b}}\right)}\\
\nonumber \geqslant&(1+o(1)) \frac{4 C(abN) N^\theta}{ab\theta(\log N)^{2} } \left(\int_{\frac{1}{14}}^{\frac{1}{8.8}} \int_{t_{1}}^{\frac{1}{8.8}} \frac{\log \left((7 \theta-1)-14\left(t_{1}+t_{2}\right)\right)}{t_{1} t_{2}\left(\frac{\theta}{2}-\left(t_{1}+t_{2}\right)\right)} d t_{1} d t_{2}\right)\\
\nonumber \geqslant&(1+o(1)) \frac{8 C(abN) N^\theta}{ab\theta(\log N)^{2} } \left(\int_{\frac{1}{14}}^{\frac{1}{8.8}} \int_{t_{1}}^{\frac{1}{8.8}} \frac{\log \left((7 \theta-1)-14\left(t_{1}+t_{2}\right)\right)}{t_{1} t_{2}\left(\theta-2\left(t_{1}+t_{2}\right)\right)} d t_{1} d t_{2}\right),\\
\nonumber S_{22}^{\prime} \geqslant& \frac{N^\theta}{a b\theta \log N}\frac{2\times 14 e^{-\gamma} C(abN)(1+o(1))}{\log N} \times \\
\nonumber &\sum_{\substack{(\frac{N}{b})^{\frac{1}{14}} \leqslant p_1<(\frac{N}{b})^{\frac{1}{8.8}} \leqslant p_2<(\frac{N}{b})^{\frac{4.5863}{14}}p^{-1}_1 \\ (p_1 p_2, N)=1} }\frac{1}{p_1 p_2} f\left(14\left(\frac{\theta}{2}-\frac{\log p_1 p_2}{\log \frac{N}{b}}\right)\right)\\
\nonumber \geqslant& \frac{N^\theta}{a b\theta \log N}\frac{2\times 14 e^{-\gamma} C(abN)(1+o(1))}{\log N} \times \\
\nonumber &\sum_{\substack{(\frac{N}{b})^{\frac{1}{14}} \leqslant p_1<(\frac{N}{b})^{\frac{1}{8.8}} \leqslant p_2<(\frac{N}{b})^{\frac{4.5863}{14}}p^{-1}_1 \\ (p_1 p_2, N)=1} }\frac{1}{p_1 p_2} \frac{2 e^{\gamma} \log \left(14\left(\frac{\theta}{2}-\frac{\log p_1 p_2}{\log \frac{N}{b}}\right)-1\right)}{14\left(\frac{\theta}{2}-\frac{\log p_1 p_2}{\log \frac{N}{b}}\right)}\\
\nonumber \geqslant&(1+o(1)) \frac{4 C(abN) N^\theta}{ab\theta(\log N)^{2}} \left(\int_{\frac{1}{14}}^{\frac{1}{8.8}} \int_{\frac{1}{8.8}}^{\frac{4.5863}{14}-t_{1}} \frac{\log \left((7 \theta-1)-14\left(t_{1}+t_{2}\right)\right)}{t_{1} t_{2}\left(\frac{\theta}{2}-\left(t_{1}+t_{2}\right)\right)} d t_{1} d t_{2}\right)\\
\nonumber \geqslant&(1+o(1)) \frac{8 C(abN) N^\theta}{ab\theta(\log N)^{2}} \left(\int_{\frac{1}{14}}^{\frac{1}{8.8}} \int_{\frac{1}{8.8}}^{\frac{4.5863}{14}-t_{1}} \frac{\log \left((7 \theta-1)-14\left(t_{1}+t_{2}\right)\right)}{t_{1} t_{2}\left(\theta-2\left(t_{1}+t_{2}\right)\right)} d t_{1} d t_{2}\right),\\
\nonumber S_{2}^{\prime}=&S_{21}^{\prime}+S_{22}^{\prime}\\
\nonumber \geqslant&(1+o(1)) \frac{8 C(abN) N^\theta}{ab\theta(\log N)^{2}} \left(\int_{\frac{1}{14}}^{\frac{1}{8.8}} \int_{t_{1}}^{\frac{4.5863}{14}-t_{1}} \frac{\log \left((7 \theta-1)-14\left(t_{1}+t_{2}\right)\right)}{t_{1} t_{2}\left(\theta-2\left(t_{1}+t_{2}\right)\right)} d t_{1} d t_{2}\right)\\
\geqslant& 5.914688 \frac{C(abN) N^\theta}{ab(\log N)^{2}},
\end{align}
\begin{align}
\nonumber S_{31}^{\prime} \leqslant& \frac{N^\theta}{a b\theta \log N}\frac{2\times 14 e^{-\gamma} C(abN)(1+o(1))}{\log N}\sum_{\substack{(\frac{N}{b})^{\frac{1}{14}} \leqslant p<(\frac{N}{b})^{\frac{4.08631}{14}} \\ (p, N)=1} }\frac{1}{p} F\left(14\left(\frac{\theta}{2}-\frac{\log p}{\log \frac{N}{b}}\right)\right)\\
\nonumber \leqslant& \frac{N^\theta}{a b\theta \log N}\frac{2\times 14 e^{-\gamma} C(abN)(1+o(1))}{\log N} \int_{(\frac{N}{b})^{\frac{1}{14}}}^{(\frac{N}{b})^{\frac{4.08631}{14}}} \frac{1}{u \log u} F\left(14\left(\frac{\theta}{2}-\frac{\log u}{\log \frac{N}{b}}\right)\right) d u \\
\nonumber \leqslant&(1+o(1)) \frac{8 C(abN) N^\theta}{ab\theta^2(\log N)^{2}}\left(\log \frac{4.08631(14\theta-2)}{14\theta-8.17262}\right.\\
\nonumber&\left.+\int_{2}^{7\theta-2} \frac{\log (s-1)}{s} \log \frac{(7\theta-1)(7\theta-1-s)}{s+1} d s\right.\\
\nonumber&\left.+\int_{2}^{7\theta-4} \frac{\log (s-1)}{s} d s \int_{s+2}^{7\theta-2} \frac{1}{t} \log \frac{t-1}{s+1} \log \frac{(7\theta-1)(7\theta-1-t)}{t+1} d t\right)\\
\nonumber\leqslant& 24.63508 \frac{C(abN) N^\theta}{ab(\log N)^{2}},\\
\nonumber S_{32}^{\prime} \leqslant& \frac{N^\theta}{a b\theta \log N}\frac{2\times 14 e^{-\gamma} C(abN)(1+o(1))}{\log N}\sum_{\substack{(\frac{N}{b})^{\frac{1}{14}} \leqslant p<(\frac{N}{b})^{\frac{3.5863}{14}} \\ (p, N)=1} }\frac{1}{p} F\left(14\left(\frac{\theta}{2}-\frac{\log p}{\log \frac{N}{b}}\right)\right)\\
\nonumber \leqslant& \frac{N^\theta}{a b\theta \log N}\frac{2\times 14 e^{-\gamma} C(abN)(1+o(1))}{\log N} \int_{(\frac{N}{b})^{\frac{1}{14}}}^{(\frac{N}{b})^{\frac{3.5863}{14}}} \frac{1}{u \log u} F\left(14\left(\frac{\theta}{2}-\frac{\log u}{\log \frac{N}{b}}\right)\right) d u \\
\nonumber \leqslant&(1+o(1)) \frac{8 C(abN) N^\theta}{ab\theta^2(\log N)^{2}}\left(\log \frac{3.5863(14\theta-2)}{14\theta-7.1726}\right.\\
\nonumber&\left.+\int_{2}^{7\theta-2} \frac{\log (s-1)}{s} \log \frac{(7\theta-1)(7\theta-1-s)}{s+1} d s\right.\\
\nonumber&\left.+\int_{2}^{7\theta-4} \frac{\log (s-1)}{s} d s \int_{s+2}^{7\theta-2} \frac{1}{t} \log \frac{t-1}{s+1} \log \frac{(7\theta-1)(7\theta-1-t)}{t+1} d t\right)\\
\nonumber\leqslant& 21.808021 \frac{C(abN) N^\theta}{ab(\log N)^{2}},\\
S_{3}^{\prime}=& S_{31}^{\prime}+S_{32}^{\prime} \leqslant 46.443101 \frac{C(abN) N^\theta}{ab(\log N)^{2}}.
\end{align}

\subsection{Evaluation of $S_{4}^{\prime}, S_{7}^{\prime}$}
Let $D_{\mathcal{B}_2}=N^{\theta-1 / 2}(\log N)^{-B}$. By Chen's switching principle and similar arguments as in \cite{Cai2002}, we know that
\begin{equation}
|\mathcal{E}_2|<\left(\frac{N}{b}\right)^{\frac{2}{3}}, \quad
\left(\frac{N}{b}\right)^{\frac{1}{3}}<e \leqslant \left(\frac{N}{b}\right)^{\frac{2}{3}} \ \operatorname{for}\ e \in \mathcal{E}_2, \quad
S_{4}^{\prime} \leqslant S\left(\mathcal{B}_2;\mathcal{P},D_{\mathcal{B}_2}^{\frac{1}{2}}\right)+O\left(N^{\frac{2}{3}}\right).
\end{equation}
Then we can take
\begin{gather}
\nonumber X_{\mathcal{B}_2}=\sum_{\substack{(\frac{N}{b})^{\frac{1}{14}} \leqslant p_{1} \leqslant(\frac{N}{b})^{\frac{1}{3.1}}<p_{2} \leqslant(\frac{N}{b p_{1}})^{\frac{1}{2}}  \\ 0 \leqslant j \leqslant a-1,(j, a)=1}} \left(\pi\left(\frac{N}{b p_{1} p_{2}} ; a^{2}, N\left(b p_{1} p_{2}\right)_{a^{2}}^{-1}+j a\right)\right.\\
\left.-\pi\left(\frac{N-N^\theta}{b p_{1} p_{2}} ; a^{2}, N\left(b p_{1} p_{2}\right)_{a^{2}}^{-1}+j a\right)\right)
\end{gather}
so that $|\mathcal{B}_2| \sim X_{\mathcal{B}_2}$. By Lemma~\ref{Wfunction} for $z_{\mathcal{B}_2}=D_{\mathcal{B}_2}^{\frac{1}{2}}=N^{\frac{2\theta-1}{4}}(\log N)^{-B/2}$ we have
\begin{equation}
W(z_{\mathcal{B}_2})=\frac{8 e^{-\gamma} C(abN)(1+o(1))}{(2\theta-1)\log N}, \quad F(2)=e^{\gamma}.
\end{equation}
By Huxley's prime number theorem in short intervals and integeration by parts we get that
\begin{align}
\nonumber X_{\mathcal{B}_2} &=(1+o(1)) \sum_{(\frac{N}{b})^{\frac{1}{14}} \leqslant p_{1} \leqslant(\frac{N}{b})^{\frac{1}{3.1}}<p_{2} \leqslant(\frac{N}{b p_{1}})^{\frac{1}{2}}} \frac{\varphi(a) \frac{N^\theta}{b p_{1} p_{2}}}{\varphi\left(a^{2}\right) \log \left(\frac{N}{b p_{1} p_{2}}\right)} \\
\nonumber & =(1+o(1)) \frac{N^\theta}{a b} \sum_{(\frac{N}{b})^{\frac{1}{14}} \leqslant p_{1} \leqslant(\frac{N}{b})^{\frac{1}{3.1}}<p_{2} \leqslant(\frac{N}{b p_{1}})^{\frac{1}{2}}} \frac{1}{p_{1} p_{2} \log \left(\frac{N}{p_{1} p_{2}}\right)} \\
\nonumber& =(1+o(1)) \frac{N^\theta}{a b} \int_{(\frac{N}{b})^{\frac{1}{14}}}^{(\frac{N}{b})^{\frac{1}{3.1}}} \frac{d t}{t \log t} \int_{(\frac{N}{b})^{\frac{1}{3.1}}}^{(\frac{N}{bt})^{\frac{1}{2}}} \frac{d u}{u \log u \log \left(\frac{N}{u t}\right)}\\
& =(1+o(1)) \frac{N^\theta}{ab\log N} \int_{2.1}^{13} \frac{\log \left(2.1-\frac{3.1}{s+1}\right)}{s} d s .
\end{align}

To deal with the error terms, for an integer $n$ such that $(n, a b N)>1$, similarly to the discussion for $\eta\left(X_{\mathcal{B}_1}, n\right)$, we have $\eta\left(X_{\mathcal{B}_2}, n\right)=0$.
For a square-free integer $n$ such that $(n, a b N)=1$, if $n \mid \frac{N-b p_{1} p_{2} p_{3}}{a}$, then $(p_{1} ,n)=1$ and $(p_{2} , n)=1$. Moreover, if $\left(\frac{N-b p_{1} p_{2} p_{3}}{a n}, a\right)=1$, then we have $b p_{1} p_{2} p_{3} \equiv$ $N+j a n\left(\bmod a^{2} n\right)$ for some $j$ such that $0 \leqslant j \leqslant a-1$ and $(j, a)=1$. Conversely, if $b p_{1} p_{2} p_{3}=N+j a n+s a^{2} n$ for some integer $j$ such that $0 \leqslant j \leqslant a$ and $(j, a)=1$, some integer $n$ relatively prime to $p_{1} p_{2}$ such that $an \mid\left(N-b p_{1} p_{2} p_{3}\right)$, and some integer $s$, then $\left(\frac{N-b p_{1} p_{2} p_{3}}{a n}, a\right)=(-j, a)=1$. Since $j b p_{1} p_{2}$ runs through the reduced residues modulo $a$ when $j$ runs through the reduced residues modulo $a$ and $\pi\left(x ; k, 1,1\right)=\pi\left(\frac{x}{k} ; 1,1\right)$, for square-free integers $n$ such that $(n, a b N)=1$, we have
\begin{align}
\nonumber \left|\eta\left(X_{\mathcal{B}_2}, n\right)\right| =&\left|\sum_{\substack{a \in \mathcal{B}_2 \\
a \equiv 0(\bmod n)}} 1-\frac{\omega(n)}{n} X_{\mathcal{B}_2}\right|=\left|\sum_{\substack{a \in \mathcal{B}_2 \\
a \equiv 0(\bmod n)}} 1-\frac{X_{\mathcal{B}_2}}{\varphi(n)}\right| \\
\nonumber =&\left|\sum_{\substack{(\frac{N}{b})^{\frac{1}{14}} \leqslant p_{1} \leqslant(\frac{N}{b})^{\frac{1}{3.1}}<p_{2} \leqslant(\frac{N}{b p_{1}})^{\frac{1}{2}},(p_1 p_2,N)=1 \\
\left(p_{1} p_{2}, n\right)=1,
0 \leqslant j \leqslant a-1,(j, a)=1}} \left(\pi\left(N ; b p_{1} p_{2}, a^{2} n, N+j a n\right)-\pi\left(N-N^\theta ; b p_{1} p_{2}, a^{2} n, N+j a n\right)\right)\right.\\
\nonumber&\left. -\sum_{\substack{(\frac{N}{b})^{\frac{1}{14}} \leqslant p_{1} \leqslant(\frac{N}{b})^{\frac{1}{3.1}}<p_{2} \leqslant(\frac{N}{b p_{1}})^{\frac{1}{2}} \\
\left(p_{1} p_{2}, n\right)=1,
0 \leqslant j \leqslant a-1,(j, a)=1}} \frac{\left(\pi\left(\frac{N}{b p_{1} p_{2}} ; a^{2}, N\left(b p_{1} p_{2}\right)_{a^{2}}^{-1}+j a\right)-\pi\left(\frac{N-N^\theta}{b p_{1} p_{2}} ; a^{2}, N\left(b p_{1} p_{2}\right)_{a^{2}}^{-1}+j a\right)\right)}{\varphi(n)}\right| \\
\nonumber \ll& \left|\sum_{\substack{(\frac{N}{b})^{\frac{1}{14}} \leqslant p_{1} \leqslant(\frac{N}{b})^{\frac{1}{3.1}}<p_{2} \leqslant(\frac{N}{b p_{1}})^{\frac{1}{2}},(p_1 p_2,N)=1 \\
\left(p_{1} p_{2}, n\right)=1,
0 \leqslant j \leqslant a-1,(j, a)=1}}\left(\left(\pi\left(N ; b p_{1} p_{2}, a^{2} n, N+j a n\right)-\pi\left(N-N^\theta ; b p_{1} p_{2}, a^{2} n, N+j a n\right)\right)\right.\right. \\
\nonumber& \left.\left.-\frac{\left(\pi\left(\frac{N}{b p_{1} p_{2}} ; a^{2}, N\left(b p_{1} p_{2}\right)_{a^{2}}^{-1}+j a\right)-\pi\left(\frac{N-N^\theta}{b p_{1} p_{2}} ; a^{2}, N\left(b p_{1} p_{2}\right)_{a^{2}}^{-1}+j a\right)\right)}{\varphi(n)}\right)\right| \\
\nonumber& +\sum_{\substack{(\frac{N}{b})^{\frac{1}{14}} \leqslant p_{1} \leqslant(\frac{N}{b})^{\frac{1}{3.1}}<p_{2} \leqslant(\frac{N}{b p_{1}})^{\frac{1}{2}} \\
\left(p_{1} p_{2}, nN\right)>1,
0 \leqslant j \leqslant a-1,(j, a)=1}} \frac{\left(\pi\left(\frac{N}{b p_{1} p_{2}} ; a^{2}, N\left(b p_{1} p_{2}\right)_{a^{2}}^{-1}+j a\right)-\pi\left(\frac{N-N^\theta}{b p_{1} p_{2}} ; a^{2}, N\left(b p_{1} p_{2}\right)_{a^{2}}^{-1}+j a\right)\right)}{\varphi(n)}\\
\nonumber \ll&\left|\sum_{\substack{(\frac{N}{b})^{\frac{1}{14}} \leqslant p_{1} \leqslant(\frac{N}{b})^{\frac{1}{3.1}}<p_{2} \leqslant(\frac{N}{b p_{1}})^{\frac{1}{2}},(p_1 p_2,N)=1 \\
\left(p_{1} p_{2}, n\right)=1,
0 \leqslant j \leqslant a-1,(j, a)=1}}\left(\left(\pi\left(N ; b p_{1} p_{2}, a^{2} n, N+j a n\right)-\pi\left(N-N^\theta ; b p_{1} p_{2}, a^{2} n, N+j a n\right)\right)\right.\right.\\
\nonumber&\left.\left.-\frac{\left(\pi\left(N ; b p_{1} p_{2}, 1,1\right)-\pi\left(N-N^\theta ; b p_{1} p_{2}, 1,1\right)\right)}{\varphi\left(a^{2} n\right)}\right)\right| \\
\nonumber& +\left| \sum_{\substack{(\frac{N}{b})^{\frac{1}{14}} \leqslant p_{1} \leqslant(\frac{N}{b})^{\frac{1}{3.1}}<p_{2} \leqslant(\frac{N}{b p_{1}})^{\frac{1}{2}},(p_1 p_2,N)=1 \\
\left(p_{1} p_{2}, n\right)=1,
0 \leqslant j \leqslant a-1,(j, a)=1}}\right.\\
\nonumber&\left.\left(\frac{\left(\pi\left(\frac{N}{b p_{1} p_{2}} ; a^{2}, N\left(b p_{1} p_{2}\right)_{a^{2}}^{-1}+j a\right)-\pi\left(\frac{N-N^\theta}{b p_{1} p_{2}} ; a^{2}, N\left(b p_{1} p_{2}\right)_{a^{2}}^{-1}+j a\right)\right)}{\varphi(n)}\right.\right.\\
\nonumber&\left.\left.-\frac{\left(\pi\left(\frac{N}{b p_{1} p_{2}} ; 1,1\right)-\pi\left(\frac{N-N^\theta}{b p_{1} p_{2}} ; 1,1\right)\right)}{\varphi\left(a^{2} n\right)}\right)\right|+N^{\frac{13}{14}}(\log N)^{2}\\
\nonumber \ll&\left|\sum_{\substack{(\frac{N}{b})^{\frac{1}{14}} \leqslant p_{1} \leqslant(\frac{N}{b})^{\frac{1}{3.1}}<p_{2} \leqslant(\frac{N}{b p_{1}})^{\frac{1}{2}},(p_1 p_2,N)=1 \\
\left(p_{1} p_{2}, n\right)=1,
0 \leqslant j \leqslant a-1,(j, a)=1}}\left(\left(\pi\left(N ; b p_{1} p_{2}, a^{2} n, N+j a n\right)-\pi\left(N-N^\theta ; b p_{1} p_{2}, a^{2} n, N+j a n\right)\right)\right.\right.\\
\nonumber&\left.\left.-\frac{\left(\pi\left(N ; b p_{1} p_{2}, 1,1\right)-\pi\left(N-N^\theta ; b p_{1} p_{2}, 1,1\right)\right)}{\varphi\left(a^{2} n\right)}\right)\right| \\
\nonumber& +\frac{1}{\varphi(n)}\left| \sum_{\substack{(\frac{N}{b})^{\frac{1}{14}} \leqslant p_{1} \leqslant(\frac{N}{b})^{\frac{1}{3.1}}<p_{2} \leqslant(\frac{N}{b p_{1}})^{\frac{1}{2}},(p_1 p_2,N)=1 \\
\left(p_{1} p_{2}, n\right)=1,
0 \leqslant j \leqslant a-1,(j, a)=1}}\right.\\
\nonumber&\left.\left(\left(\pi\left(\frac{N}{b p_{1} p_{2}} ; a^{2}, N\left(b p_{1} p_{2}\right)_{a^{2}}^{-1}+j a\right)-\pi\left(\frac{N-N^\theta}{b p_{1} p_{2}} ; a^{2}, N\left(b p_{1} p_{2}\right)_{a^{2}}^{-1}+j a\right)\right)\right.\right.\\
&\left.\left.-\frac{\left(\pi\left(\frac{N}{b p_{1} p_{2}} ; 1,1\right)-\pi\left(\frac{N-N^\theta}{b p_{1} p_{2}} ; 1,1\right)\right)}{\varphi\left(a^{2} \right)}\right)\right|+N^{\frac{13}{14}}(\log N)^{2}.
\end{align}
By Lemma~\ref{BVshort} with
$$
g(k)= 
\begin{cases}
1, & \text { if } k \in \mathcal{E}_2 \\ 
0, & \text { otherwise }
\end{cases},
$$
we have
\begin{equation}
\sum_{\substack{n \leqslant D_{\mathcal{B}_2} \\ n \mid P(z_{\mathcal{B}_2})}} \left|\eta\left(X_{\mathcal{B}_2}, n\right)\right| \ll N^\theta(\log N)^{-5}.
\end{equation}

Then by (69)--(74) and some routine arguments we have
$$
S_{41}^{\prime} \leqslant (1+o(1)) \frac{8C(abN) N^\theta}{ab(2\theta-1)(\log N)^{2}} \int_{2.1}^{13} \frac{\log \left(2.1-\frac{3.1}{s+1}\right)}{s} d s.
$$
Similarly, we have
$$
S_{42}^{\prime} \leqslant(1+o(1)) \frac{8C(abN) N^\theta}{ab(2\theta-1)(\log N)^{2}}\int_{2.7}^{7.8} \frac{\log \left(2.7-\frac{3.7}{s+1}\right)}{s} d s,
$$
$$
S_{4}^{\prime}=S_{41}^{\prime}+S_{42}^{\prime} \leqslant (1+o(1)) \frac{8C(abN) N^\theta}{ab(2\theta-1)(\log N)^{2}} \left(\int_{2.1}^{13} \frac{\log \left(2.1-\frac{3.1}{s+1}\right)}{s} d s+\int_{2.7}^{7.8} \frac{\log \left(2.7-\frac{3.7}{s+1}\right)}{s} d s\right)
$$
\begin{equation}
\leqslant 13.953531 \frac{C(abN) N^\theta}{ab(\log N)^{2}},
\end{equation}
\begin{align}
\nonumber S_{71}^{\prime} &\leqslant(1+o(1)) \frac{8C(abN) N^\theta}{ab(2\theta-1)(\log N)^{2}}\int_{2}^{2.1} \frac{\log (s-1)}{s} d s\\
\nonumber S_{72}^{\prime} &\leqslant(1+o(1)) \frac{8C(abN) N^\theta}{ab(2\theta-1)(\log N)^{2}}\int_{2}^{2.7} \frac{\log (s-1)}{s} d s,
\end{align}
$$
S_{7}^{\prime}=S_{71}^{\prime}+S_{72}^{\prime} \leqslant(1+o(1)) \frac{8C(abN) N^\theta}{ab(2\theta-1)(\log N)^{2}}\left(\int_{2}^{2.1} \frac{\log (s-1)}{s} d s+\int_{2}^{2.7} \frac{\log (s-1)}{s} d s\right)
$$
\begin{equation}
\leqslant 0.771273 \frac{C(abN) N^\theta}{ab(\log N)^{2}}.
\end{equation}

\subsection{Evaluation of $S_{6}^{\prime}$}
Let $D_{\mathcal{C}_2}=N^{\theta-1 / 2}(\log N)^{-B}$. By Chen's switching principle and similar arguments as in \cite{Cai2015}, we know that
\begin{equation}
S_{61}^{\prime} \leqslant S\left(\mathcal{C}_2;\mathcal{P},D_{\mathcal{C}_2}^{\frac{1}{2}}\right)+O\left(D_{\mathcal{C}_2}^{\frac{1}{2}}\right).
\end{equation}
By Lemma~\ref{Wfunction} for $z_{\mathcal{C}_2}=D_{\mathcal{C}_2}^{\frac{1}{2}}=N^{\frac{2\theta-1}{4}}(\log N)^{-B/2}$ we have
\begin{equation}
W(z_{\mathcal{C}_2})=\frac{8 e^{-\gamma} C(abN)(1+o(1))}{(2\theta-1)\log N}, \quad F(2)=e^{\gamma}.
\end{equation}
By Lemma~\ref{buchstabshort} we have
\begin{align}
\nonumber |\mathcal{C}_2|&=
\sum_{\substack{mp_1 p_2 p_3 p_4 \in \mathcal{F}_2 
\\mp_1 p_2 p_3 p_4 \equiv N b_{a^2}^{-1} +j a (\bmod a^{2}) \\ 0 \leqslant j \leqslant a-1,(j, a)=1 }}1 \\
\nonumber &=\sum_{\substack{(\frac{N}{b})^{\frac{1}{14}} \leqslant p_1 < p_2 < p_3< p_4<(\frac{N}{b})^{\frac{1}{8.8}} \\ (p_1 p_2 p_3 p_4, N)=1} }\sum_{\substack{\frac{N-N^\theta}{bp_1 p_2 p_3 p_4}\leqslant m\leqslant \frac{N}{bp_1 p_2 p_3 p_4}\\\left(m, p_{1}^{-1} abN P\left(p_{2}\right)\right)=1 }}\frac{\varphi(a)}{\varphi(a^2)}\\
\nonumber &<(1+o(1))\frac{N^\theta}{ab}\sum_{(\frac{N}{b})^{\frac{1}{14}} \leqslant p_{1}<p_{2}<p_{3}<p_{4}<(\frac{N}{b})^{\frac{1}{8.8}}} \frac{0.5617}{p_{1} p_{2} p_{3} p_{4} \log p_{2}}\\
&=(1+o(1))\frac{0.5617N^{\theta}}{ab\log N}\int_{\frac{1}{14}}^{\frac{1}{8.8}} \frac{d t_{1}}{t_{1}} \int_{t_{1}}^{\frac{1}{8.8}} \frac{1}{t_{2}}\left(\frac{1}{t_{1}}-\frac{1}{t_{2}}\right) \log \frac{1}{8.8 t_{2}} d t_{2} .
\end{align}

To deal with the error terms, for an integer $n$ such that $(n, a b N)>1$, similarly to the discussion for $\eta\left(X_{\mathcal{C}_1}, n\right)$, we have $\eta\left(|{\mathcal{C}_2}|, n\right)=0$. 
For a square-free integer $n$ that is relatively prime to $a b N$, if $n \mid \frac{N-bm p_{1} p_{2} p_{3} p_4}{a}$, then $(p_{1} ,n)=1, (p_{2} ,n)=1, (p_{3} ,n)=1$ and $(p_{4} , n)=1$. Moreover, if $\left(\frac{N-bm p_{1} p_{2} p_{3} p_4}{a n}, a\right)=1$, then we have $bm p_{1} p_{2} p_{3} p_4 \equiv$ $N+j a n\left(\bmod a^{2} n\right)$ for some $j$ such that $0 \leqslant j \leqslant a-1$ and $(j, a)=1$. Conversely, if $bm p_{1} p_{2} p_{3} p_4=N+j a n+s a^{2} n$ for some integer $j$ such that $0 \leqslant j \leqslant a$ and $(j, a)=1$, some integer $n$ relatively prime to $p_{1} p_{2} p_3 p_4$ such that $an \mid\left(N-bm p_{1} p_{2} p_{3} p_4\right)$, and some integer $s$, then $\left(\frac{N-bm p_{1} p_{2} p_{3} p_4}{a n}, a\right)=(-j, a)=1$. Since $j b_{a^2 n}^{-1}$ runs through the reduced residues modulo $a$ when $j$ runs through the reduced residues modulo $a$, for a square-free integer $n$ relatively prime to $a b N$, we have
\begin{align}
\nonumber \left|\eta\left(|{\mathcal{C}_2}|, n\right)\right| =&\left|\sum_{\substack{a \in \mathcal{C}_2 \\
a \equiv 0(\bmod n)}} 1-\frac{\omega(n)}{n} |\mathcal{C}_2|\right|=\left|\sum_{\substack{a \in \mathcal{C}_2 \\
a \equiv 0(\bmod n)}} 1-\frac{|\mathcal{C}_2|}{\varphi(n)}\right| \\
\nonumber \ll&\left|\sum_{\substack{e \in \mathcal{F}_2 \\ (e,n)=1 \\ e \equiv N b_{a^2 n}^{-1} +j a n (\bmod a^{2} n) \\ 0 \leqslant j \leqslant a-1, (j,a)=1}}1-\frac{1}{\varphi(n)}\sum_{\substack{e \in \mathcal{F}_2 \\ (e,n)=1 \\ e \equiv N b_{a^2}^{-1} +j a  (\bmod a^{2} ) \\ 0 \leqslant j \leqslant a-1, (j,a)=1}}1\right|+\frac{1}{\varphi(n)}\sum_{\substack{e \in \mathcal{F}_2 \\ (e,n)>1 \\ e \equiv N b_{a^2}^{-1} +j a  (\bmod a^{2} ) \\ 0 \leqslant j \leqslant a-1, (j,a)=1}}1\\
\nonumber\ll&\left|\sum_{\substack{e \in \mathcal{F}_2 \\ (e,n)=1 \\ e \equiv N b_{a^2 n}^{-1} +j a n (\bmod a^{2} n) \\ 0 \leqslant j \leqslant a-1, (j,a)=1}}1-\frac{1}{\varphi(a^2 n)}\sum_{\substack{e \in \mathcal{F}_2 \\ (e,n)=1}}1\right|\\
\nonumber&+\left|\frac{1}{\varphi(n)}\sum_{\substack{e \in \mathcal{F}_2 \\ (e,n)=1 \\ e \equiv N b_{a^2}^{-1} +j a  (\bmod a^{2} ) \\ 0 \leqslant j \leqslant a-1, (j,a)=1}}1-\frac{1}{\varphi(a^2 n)}\sum_{\substack{e \in \mathcal{F}_2 \\ (e,n)=1}}1\right|+N^{\theta-\frac{1}{14}}(\log N)^{2}\\
\nonumber\ll&\left|\sum_{\substack{e \in \mathcal{F}_2 \\ (e,n)=1 \\ e \equiv N b_{a^2 n}^{-1} +j a n (\bmod a^{2} n) \\ 0 \leqslant j \leqslant a-1, (j,a)=1}}1-\frac{1}{\varphi(a^2 n)}\sum_{\substack{e \in \mathcal{F}_2 \\ (e,n)=1}}1\right|\\
&+\frac{1}{\varphi(n)}\left|\sum_{\substack{e \in \mathcal{F}_2 \\ (e,n)=1 \\ e \equiv N b_{a^2}^{-1} +j a  (\bmod a^{2} ) \\ 0 \leqslant j \leqslant a-1, (j,a)=1}}1-\frac{1}{\varphi(a^2 )}\sum_{\substack{e \in \mathcal{F}_2 \\ (e,n)=1}}1\right|+N^{\theta-\frac{1}{14}}(\log N)^{2}.
\end{align}
By Lemma~\ref{newmeanvalue}, we have
\begin{equation}
\sum_{\substack{n \leqslant D_{\mathcal{C}_2} \\ n \mid P(z_{\mathcal{C}_2})}} \left|\eta\left(|{\mathcal{C}_2}|, n\right)\right| \ll N^\theta(\log N)^{-5}.
\end{equation}

Then by (77)--(81) and some routine arguments we have
\begin{align}
\nonumber S_{61}^{\prime} &\leqslant(1+o(1)) \frac{0.5617 \times 8 C(abN) N^\theta}{ab(2\theta-1)(\log N)^2} \int_{\frac{1}{14}}^{\frac{1}{8.8}} \frac{d t_{1}}{t_{1}} \int_{t_{1}}^{\frac{1}{8.8}} \frac{1}{t_{2}}\left(\frac{1}{t_{1}}-\frac{1}{t_{2}}\right) \log \frac{1}{8.8 t_{2}} d t_{2} \\
& \leqslant 0.115227 \frac{C(abN) N^\theta}{ab(\log N)^{2}}.
\end{align}
Similarly, we have
\begin{align}
\nonumber S_{62}^{\prime}=& \sum_{(\frac{N}{b})^{\frac{1}{14}} \leqslant p_{1}<p_{2}<p_{3}<(\frac{N}{b})^{ \frac{1}{8.8}} \leqslant p_{4}<(\frac{N}{b})^{\frac{1.8}{8.8}} } S\left(\mathcal{A}_{p_{1} p_{2} p_{3} p_{4}} ; \mathcal{P}\left(p_{1}\right), p_{2}\right)\\
\nonumber &+\sum_{(\frac{N}{b})^{\frac{1}{14}} \leqslant p_{1}<p_{2}<p_{3}<(\frac{N}{b})^{ \frac{1}{8.8}}<(\frac{N}{b})^{ \frac{1.8}{8.8}} \leqslant p_{4}<(\frac{N}{b})^{\frac{4.5863}{14}}p^{-1}_3 } S\left(\mathcal{A}_{p_{1} p_{2} p_{3} p_{4}} ; \mathcal{P}\left(p_{1}\right), p_{2}\right)\\
\nonumber \leqslant& (1+o(1)) \frac{0.5617 \times 8 C(abN) N^\theta}{ab(2\theta-1)(\log N)^2} \left(22.8 \log \frac{14}{8.8}-10.4\right) \log 1.8 \\
\nonumber &+(1+o(1)) \frac{0.5644 \times 8 C(abN) N^\theta}{ab(2\theta-1)(\log N)^2} \int_{\frac{1}{14}}^{\frac{1}{8.8}} \frac{d t_{1}}{t_{1}} \int_{t_{1}}^{\frac{1}{8.8}} \frac{1}{t_{2}}\left(\frac{1}{t_{1}}-\frac{1}{t_{2}}\right) \log \left(\frac{8.8}{1.8}\left(\frac{4.5863}{14}-t_2\right)\right) d t_{2}\\
\leqslant& 0.654234 \frac{C(abN) N^\theta}{ab(\log N)^{2}}.
\end{align}
By (82) and (83) we have
\begin{align}
\nonumber S_{6}^{\prime}=S_{61}^{\prime}+S_{62}^{\prime} &\leqslant 0.115227 \frac{C(abN) N^\theta}{ab(\log N)^{2}} +0.654234 \frac{C(abN) N^\theta}{ab(\log N)^{2}}\\
&\leqslant 0.769461 \frac{C(abN) N^\theta}{ab(\log N)^{2}} .
\end{align}

\subsection{Evaluation of $S_{5}^{\prime}$} For $p \geqslant \left(\frac{N}{b}\right)^{\frac{4.08631}{14}}$ we have
$$
\underline{p}^{\prime\frac{1}{2.5}} \leqslant \left(\frac{N}{b}\right)^{\frac{1}{14}}, \quad S\left(\mathcal{A}_{p};\mathcal{P}, \left(\frac{N}{b}\right)^{\frac{1}{14}}\right) \leqslant S\left(\mathcal{A}_{p};\mathcal{P}, \underline{p}^{\prime\frac{1}{2.5}}\right).
$$
By Lemma~\ref{l35} we have
\begin{align}
\nonumber S_{51}^{\prime}&=\sum_{\substack{(\frac{N}{b})^{\frac{4.08631}{14}} \leqslant p<(\frac{N}{b})^{\frac{1}{3.1}} \\ (p, N)=1}} S\left(\mathcal{A}_{p};\mathcal{P},\left(\frac{N}{b}\right)^{\frac{1}{14}}\right)\\
&\leqslant \sum_{\substack{(\frac{N}{b})^{\frac{4.08631}{14}} \leqslant p<(\frac{N}{b})^{\frac{1}{3.1}} \\ (p, N)=1}} S\left(\mathcal{A}_{p};\mathcal{P},\underline{p}^{\prime\frac{1}{2.5}}\right)
\leqslant \Gamma_{1}^{\prime}-\frac{1}{2} \Gamma_{2}^{\prime}+\frac{1}{2} \Gamma_{3}^{\prime}+O\left(N^{\theta-\frac{1}{20}}\right) .
\end{align}
By Lemmas~\ref{l1}, ~\ref{l2}, ~\ref{Wfunction}, ~\ref{l3} and some routine arguments we get
\begin{align}
\nonumber\Gamma_{1}^{\prime}&=\sum_{\substack{(\frac{N}{b})^{\frac{4.08631}{14}} \leqslant p<(\frac{N}{b})^{\frac{1}{3.1}} \\ (p, N)=1}} S\left(\mathcal{A}_{p};\mathcal{P},\underline{p}^{\prime\frac{1}{3.675}}\right)\\
&\leqslant(1+o(1)) \frac{8C(abN) N^\theta}{ab\theta(\log N)^{2}}\left(\int_{\frac{4.08631}{14}}^{\frac{1}{3.1}} \frac{d t}{t(\theta-2 t)}\right)\left(1+\int_{2}^{2.675} \frac{\log (t-1)}{t} d t\right),\\
\nonumber\Gamma_{2}^{\prime}&=\sum_{\substack{(\frac{N}{b})^{\frac{4.08631}{14}} \leqslant p<(\frac{N}{b})^{\frac{1}{3.1}} \\ (p, N)=1}} \sum_{\substack{\underline{p}^{\prime\frac{1}{3.675}} \leqslant p_1<\underline{p}^{\prime\frac{1}{2.5}} \\ (p_1, N)=1}}S\left(\mathcal{A}_{p p_1};\mathcal{P}, \underline{p}^{\prime\frac{1}{3.675}}\right)\\
&\geqslant(1+o(1)) \frac{8C(abN) N^\theta}{ab\theta(\log N)^{2}}\left(\int_{\frac{4.08631}{14}}^{\frac{1}{3.1}} \frac{d t}{t(\theta-2 t)}\right)\left(\int_{1.5}^{2.675} \frac{\log \left(2.675-\frac{3.675}{t+1}\right)}{t} d t\right).
\end{align}
By arguments similar to the evaluation of $S_{8}$ in \cite{CL2011} we get that
\begin{align}
\nonumber \Gamma_{3}^{\prime}&=\sum_{\substack{(\frac{N}{b})^{\frac{4.08631}{14}} \leqslant p<(\frac{N}{b})^{\frac{1}{3.1}} \\ (p, N)=1}} 
\sum
_{
\substack{\underline{p}^{\prime\frac{1}{3.675}} \leqslant p_1<p_2<p_3<\underline{p}^{\prime\frac{1}{2.5}} \\ (p_1 p_2 p_3, N)=1}
}
S\left(\mathcal{A}_{p p_1 p_2 p_3};\mathcal{P}(p_1), p_2\right)\\
\nonumber &\leqslant (1+o(1))\frac{8C(abN)}{(2\theta-1)\log N} \sum_{\substack{(\frac{N}{b})^{\frac{4.08631}{14}} \leqslant p<(\frac{N}{b})^{\frac{1}{3.1}} \\ (p, N)=1}} 
\sum_{\substack{\underline{p}^{\frac{1}{3.675}} \leqslant p_1<p_2<p_3<\underline{p}^{\frac{1}{2.5}} \\ (p_1 p_2 p_3, N)=1}}\sum_{\substack{ m\leqslant \frac{N}{bp p_1 p_2 p_3}\\\left(m, p_{1}^{-1} abN P\left(p_{2}\right)\right)=1 }}\frac{\varphi(a)}{\varphi(a^2)}\\
\nonumber &\leqslant (1+o(1))\frac{8C(abN) N^\theta}{1.763(2\theta-1) ab \log N}\sum_{\substack{(\frac{N}{b})^{\frac{4.08631}{14}} \leqslant p<(\frac{N}{b})^{\frac{1}{3.1}} \\ (p, N)=1}} \frac{1}{p \log \underline{p}} \int_{\frac{1}{3.675}}^{\frac{1}{2.5}} \int_{t_{1}}^{\frac{1}{2.5}} \int_{t_{2}}^{\frac{1}{2.5}} \frac{d t_{1} d t_{2} d t_{3}}{t_{1} t_{2}^{2} t_{3}} \\
&\leqslant(1+o(1)) \frac{16C(abN) N^\theta}{1.763ab(2\theta-1)(\log N)^{2}}\left(\int_{\frac{4.08631}{14}}^{\frac{1}{3.1}} \frac{d t}{t(\theta-2 t)}\right)\left(6.175 \log \frac{3.675}{2.5}-2.35\right).
\end{align}
By (85)--(88) we have
$$
S_{51}^{\prime}=\sum_{\substack{(\frac{N}{b})^{\frac{4.08631}{14}} \leqslant p<(\frac{N}{b})^{\frac{1}{3.1}} \\ (p, N)=1}} S\left(\mathcal{A}_{p};\mathcal{P},\left(\frac{N}{b}\right)^{\frac{1}{14}}\right)
$$
$$
\begin{gathered}
\leqslant(1+o(1)) \frac{8C(abN) N^\theta}{ab\theta(\log N)^{2}}\left(\int_{\frac{4.08631}{14}}^{\frac{1}{3.1}} \frac{d t}{t(\theta-2 t)}\right)\times\\ 
\left(1+\int_{2}^{2.675} \frac{\log (t-1)}{t} d t-\frac{1}{2} \int_{1.5}^{2.675} \frac{\log \left(2.675-\frac{3.675}{t+1}\right)}{t} d t
+\frac{\theta}{1.763(2\theta-1)}\left(6.175 \log \frac{3.675}{2.5}-2.35\right)\right).
\end{gathered}
$$
Similarly, we have
$$
S_{52}^{\prime}=\sum_{\substack{(\frac{N}{b})^{\frac{3.5863}{14}} \leqslant p<(\frac{N}{b})^{\frac{1}{3.7}} \\ (p, N)=1} }S\left(\mathcal{A}_{p};\mathcal{P},\left(\frac{N}{b}\right)^{\frac{1}{8.8}}\right)
$$
$$
\begin{gathered}
\leqslant(1+o(1)) \frac{8C(abN) N^\theta}{ab\theta(\log N)^{2}}\left(\int_{\frac{3.5863}{14}}^{\frac{1}{3.7}} \frac{d t}{t(\theta-2 t)}\right)\times\\
\left(1+\int_{2}^{2.675} \frac{\log (t-1)}{t} d t-\frac{1}{2} \int_{1.5}^{2.675} \frac{\log \left(2.675-\frac{3.675}{t+1}\right)}{t} d t
+\frac{\theta}{1.763(2\theta-1)}\left(6.175 \log \frac{3.675}{2.5}-2.35\right)\right)
\end{gathered}
$$
$$
S_{5}^{\prime}=S_{51}^{\prime}+S_{52}^{\prime}
$$
$$
\begin{gathered}
\leqslant(1+o(1)) \frac{8C(abN) N^\theta}{ab\theta(\log N)^{2}}\left(\int_{\frac{4.08631}{14}}^{\frac{1}{3.1}} \frac{d t}{t(\theta-2 t)}+\int_{\frac{3.5863}{14}}^{\frac{1}{3.7}} \frac{d t}{t(\theta-2 t)}\right)\times\\ 
\left(1+\int_{2}^{2.675} \frac{\log (t-1)}{t} d t-\frac{1}{2} \int_{1.5}^{2.675} \frac{\log \left(2.675-\frac{3.675}{t+1}\right)}{t} d t
+\frac{\theta}{1.763(2\theta-1)}\left(6.175 \log \frac{3.675}{2.5}-2.35\right)\right)
\end{gathered}
$$
\begin{equation}
\leqslant 3.669999 \frac{C(abN) N^\theta}{ab(\log N)^{2}}.
\end{equation}

\subsection{Proof of theorem 1.2}
By (66)--(68), (75)--(76), (84) and (89) we get
$$
S_{1}^{\prime}+S_{2}^{\prime} \geqslant 66.37909 \frac{C(abN) N^\theta}{ab(\log N)^{2}},
$$
$$
S_{3}^{\prime}+S_{4}^{\prime}+S_{5}^{\prime}+S_{6}^{\prime}+2S_{7}^{\prime} \leqslant 66.378638 \frac{C(abN) N^\theta}{ab(\log N)^{2}},
$$
$$
4R_{a,b}^{\theta}(N) \geqslant (S_{1}^{\prime}+S_{2}^{\prime})-(S_{3}^{\prime}+S_{4}^{\prime}+S_{5}^{\prime}+S_{6}^{\prime}+2S_{7}^{\prime}) \geqslant 0.000452 \frac{C(abN) N^\theta}{ab(\log N)^{2}},
$$
$$
R_{a,b}^{\theta}(N) \geqslant 0.000113 \frac{C(abN) N^\theta}{ab(\log N)^{2}}.
$$
Theorem~\ref{t2} is proved.

\section{An outline of the proof of Theorems 1.3--1.8}
The proof of Theorems \ref{t3}--\ref{t5} is similar and even simpler than the proof of Theorems \ref{t1}--\ref{t2}. 

For Theorem~\ref{t3}, we only need Lemma~\ref{BVshort} and Remark~\ref{remark2} to deal with the sieve error terms involved instead of Lemma~\ref{newmeanvalue} (i.e. $\frac{5*0.97-3}{2}=0.925>\frac{12.2}{13.2}$). For example, let $D_{\mathcal{A}_{3}}=\left(\frac{N}{b}\right)^{0.97-1 / 2}\left(\log \left(\frac{N}{b}\right)\right)^{-B}$ and by Huxley's prime number theorem in short intervals, we can take  
\begin{align}
\nonumber X_{\mathcal{A}_3}=&\sum_{\substack{0 \leqslant k \leqslant b-1 \\(k, b)=1}}\left(\pi\left(\frac{N/2+N^{0.97}}{a} ; b^{2}, N a_{b^{2}}^{-1}+k b\right)-\pi\left(\frac{N/2-N^{0.97}}{a} ; b^{2}, N a_{b^{2}}^{-1}+k b\right)\right) \\
\sim & \frac{\varphi(b)\left(\pi\left(\frac{N/2+N^{0.97}}{a}\right)-\pi\left(\frac{N/2-N^{0.97}}{a}\right)\right)}{ \varphi\left(b^{2}\right)} \sim \frac{2N^{0.97}}{ab \log N}
\end{align}
and we can construct the sets $\mathcal{B}$, $\mathcal{C}$, $\mathcal{E}$ and $\mathcal{F}$ for Theorem~\ref{t3} similar to those of Theorem~\ref{t1} and \cite{CL1999}.

The proof of Theorems~\ref{t4}--\ref{t41} is very similar to that of Theorem~\ref{t1}. For example, let $D_{\mathcal{A}_{4}}=\left(\frac{N}{b}\right)^{1 / 2}\left(\log \left(\frac{N}{b}\right)\right)^{-B}$, we can take
\begin{equation}
X_{\mathcal{A}_4} \sim \frac{1}{\varphi(c)}X_{\mathcal{A}_1} \sim \frac{N}{\varphi(c) a b \log N}.
\end{equation}
We can construct the sets $\mathcal{B}$, $\mathcal{C}$, $\mathcal{E}$ and $\mathcal{F}$ for Theorems~\ref{t4}--\ref{t41} similar to those of Theorem~\ref{t1}. The infinite set of primes used in the proof of Theorems~\ref{t4}--\ref{t43} is $\mathcal{P}^{\prime}=\{p : (p, Nc)=1\}$, so by using the similar arguments to those of Lemma~\ref{Wfunction}, for $j=4,5,6$ we have
\begin{equation}
W^{\prime}(z_{\mathcal{A}_{j}})=\prod_{\substack{p<z \\ (p,Nc)=1}}\left(1-\frac{\omega(p)}{p}\right)=\prod_{\substack{p \mid c \\ p \nmid N\\ p>2}} \left(\frac{p-1}{p-2}\right) \frac{2\alpha e^{-\gamma} C(abN)(1+o(1))}{\log N}.
\end{equation}
To deal with the error terms involved, we need to modify our Lemmas~\ref{l3}--\ref{remark1}. We can do that by using the similar arguments to those of Kan and Shan's paper \cite{KanShan} and we refer the interested readers to check it.
For Theorem~\ref{t41}, we need Lemma~\ref{almostBV} to control the sieve error terms with "large" $c$.

The proof of Theorems~\ref{t42}--\ref{t43} is like a combination of the proof of Theorems~\ref{t2}--\ref{t3} and Theorem~\ref{t4}. For example, let $D_{\mathcal{A}_{5}}=\left(\frac{N}{b}\right)^{\theta / 2}\left(\log \left(\frac{N}{b}\right)\right)^{-B}$ and $D_{\mathcal{A}_{6}}=\left(\frac{N}{b}\right)^{0.97- 1/ 2}\left(\log \left(\frac{N}{b}\right)\right)^{-B}$, we can take
\begin{equation}
X_{\mathcal{A}_5} \sim \frac{1}{\varphi(c)}X_{\mathcal{A}_2} \sim \frac{N^\theta}{\varphi(c) a b \theta \log N} \quad \text{and} \quad X_{\mathcal{A}_6} \sim \frac{1}{\varphi(c)}X_{\mathcal{A}_3} \sim \frac{2N^{0.97}}{\varphi(c) a b \log N}.
\end{equation}
We can construct the sets $\mathcal{B}$, $\mathcal{C}$, $\mathcal{E}$ and $\mathcal{F}$ for Theorems~\ref{t42}--\ref{t43} similar to those of Theorem~\ref{t2} and \cite{CL1999}. To deal with the sieve error terms involved, we also need to modify our Lemmas~\ref{BVshort}--\ref{newmeanvalue} by using the similar arguments to those of \cite{KanShan}. 
Our Lemmas~\ref{almostBV}--\ref{almostnewmean} will help us if we want to combine Theorems~\ref{t2}--\ref{t3} with Theorem~\ref{t41} and get similar results to Theorems~\ref{t42}--\ref{t43} with "large" $c$.

Finally, in order to prove Theorem~\ref{t5}, we need Lemma~\ref{upperboundsieve} to give an upper bound. Then we can treat $\Upsilon_1$ and $\Upsilon_2$ by arguments involved in evaluation of $S_{1}, S_{2}, S_{3}$, and $\Upsilon_3$ by similar arguments involved in evaluation of $S_{6}$.

\section*{Acknowledgements} 
The author would like to thank Huixi Li and Guang--Liang Zhou for providing information about the papers \cite{LiHuixi2019} \cite{LIHUIXI} \cite{RossPhD} and for some helpful discussions.

\bibliographystyle{plain}
\bibliography{bib}
\end{document}